\documentclass[12pt]{article}
\usepackage{graphicx}
\usepackage{amsmath,amsthm,amssymb,amsfonts,euscript,enumerate}
\usepackage{amsthm}
\usepackage{color,wrapfig}
\usepackage{pxfonts}
\usepackage{verbatim}
\newtheorem{thm}{Theorem}[section]
\newtheorem{cor}[thm]{Corollary}
\newtheorem{lem}[thm]{Lemma}

\newtheorem{rmk}[thm]{Remark}

\newcommand{\R}{{\mathbb{R}}}

\newcommand{\1}{\partial}
\newcommand{\2}{\overline}
\newcommand{\3}{\varepsilon}
\newcommand{\4}{\widetilde}

\def\ni{\noindent}

\begin{document}
\title{Existence and properties of ancient solutions of the
Yamabe flow}
\author{Shu-Yu Hsu\\
%\thanks{ }\\
Department of Mathematics\\
National Chung Cheng University\\
168 University Road, Min-Hsiung\\
Chia-Yi 621, Taiwan, R.O.C.\\
e-mail: shuyu.sy@gmail.com}
\date{June 11, 2016}
\smallbreak \maketitle
\begin{abstract}
Let $n\ge 3$ and $m=\frac{n-2}{n+2}$. We construct  $5$-parameters, $4$-parameters, and $3$-parameters ancient solutions of the equation $v_t=(v^m)_{xx}+v-v^m$, $v>0$, in $\R\times (-\infty,T)$ for some $T\in\R$. This equation arises in the study of Yamabe flow. We obtain various properties of the ancient solutions of this equation including exact decay rate of ancient solutions as $|x|\to\infty$. We also prove that both the $3$-parameters ancient solution and the $4$-parameters ancient solution are singular limit solution of the $5$-parameters ancient solutions. We also prove the uniqueness of the $4$-parameters ancient solutions. As a consequence we prove that the $4$-parameters ancient solutions that we construct coincide with the $4$-parameters ancient solutions constructed by P.~Daskalopoulos, M.~del Pino, J.~King, and N.~Sesum in \cite{DPKS2}. 
\end{abstract}

\vskip 0.2truein

Key words: ancient solutions, Yamabe flow, decay rate, singular limit solution

AMS 2010 Mathematics Subject Classification: Primary 35K55, 53C44
Secondary  35A01, 35B44

\vskip 0.2truein
\setcounter{section}{0}

\section{Introduction}
\setcounter{equation}{0}
\setcounter{thm}{0}

As observed by P.~Daskalopoulos, M.~del Pino, J.~King, M.~S\'aez, N.~Sesum, and others \cite{DPKS1}, \cite{DPKS2}, \cite{PS}, the metric $g=u^{\frac{4}{n-2}}dy^2$ satisfies the Yamabe flow \cite{B1}, \cite{B2},
\begin{equation}\label{yamabe-flow}
\frac{\1 g}{\1 t}=-Rg
\end{equation}
on $\R^n$, $n\ge 3$, for $0<t<T$, where $R$ is the scalar curvature of the metric $g$, if and only if $u$ satisfies
\begin{equation}\label{u-eqn}
(u^p)_t=\frac{n-1}{m}\Delta u,\quad u>0,\quad\mbox{ in }\R^n\times (0,T)
\end{equation}
where 
\begin{equation*}\label{m-p-value}
m=\frac{n-2}{n+2},\quad p=\frac{n+2}{n-2},
\end{equation*}
and $\Delta$ is the Euclidean laplacian on $\R^n$.
When $u$ is radially symmetric, we can write $g=u^{\frac{4}{n-2}}dy^2=w^{\frac{4}{n-2}}g_{cyl}$ where 
$$
w(x,t)=|y|^{\frac{n-2}{2}}u(y,t),\quad x=\log |y|,
$$ 
and $g_{cyl}=dx^2+g_{S^{n-1}}$ is the cylindrical metric on $S^{n-1}\times\R$, with $g_{S^{n-1}}$ being the standard metric on the unit sphere $S^{n-1}$ in $\R^n$. Since
\begin{equation*}
R=w^{-\frac{n+2}{n-2}}\biggl (-\frac{4(n-1)}{n-2}\Delta_{g_{cyl}}w+R_{g_{cyl}}w\biggr),
\end{equation*} 
where $\Delta_{g_{cyl}}w$ and $R_{g_{cyl}}=(n-1)(n-2)$ are the laplace operator and scalar curvature with respect to the cylindrical metric $g_{cyl}$, by \eqref{yamabe-flow} $w$ satisfies
\begin{equation}\label{w-eqn}
(w^{\frac{n+2}{n-2}})_t=\frac{n-1}{m}\biggl(w_{xx}-\frac{(n-2)^2}{4}w\biggr),\quad w>0,
\end{equation}
in $\R\times (0,T)$. Let $\tau=-\log (T-t)$ and
\begin{equation*}
\4{u}(x,\tau)=(T-t)^{-\frac{n-2}{4}}w(x,t).
\end{equation*}
Then by \eqref{w-eqn},
\begin{equation}\label{tilde-u-eqn}
(\4{u}^p)_{\tau}=\frac{n-1}{m}\biggl(\4{u}_{xx}+\frac{n-2}{4(n-1)}\4{u}^p-\frac{(n-2)^2}{4}\4{u}\biggr),\quad \4{u}>0,
\end{equation}
in $\R\times (-\log T,\infty)$. Let $\4{g}(y,\tau)=\4{u}(x,\tau)^{\frac{4}{n-2}}g_{cyl}$ with $x=\log |y|$. Then $\4{g}=(T-t)^{-1}g$. By \eqref{yamabe-flow}, $\4{g}$ satisfies the normalized Yamabe flow,
\begin{equation*}\label{normalized-yamabe-flow}
\frac{\1\4{g}}{\1\tau}=-(R_{\4{g}}-1)\4{g}
\end{equation*}
in $\R^n\times (-\log T,\infty)$ where $R_{\4{g}}$ is the scalar curvature of $\4{g}$.
Let 
\begin{equation}\label{hat-u-defn}
\hat{u}(x,\tau)=[(n-1)(n-2)]^{-\frac{n-2}{4}}\4{u}\left(\frac{2x}{n-2},\frac{4\tau}{n+2}\right)\quad\mbox{ and }\quad v(x,\tau)=\hat{u}(x,\tau)^p.
\end{equation}
Then \eqref{tilde-u-eqn} is equivalent to
\begin{equation*}\label{hat-u-eqn}
(\hat{u}^p)_{\tau}=\hat{u}_{xx}+\hat{u}^p-\hat{u},\quad\hat{u}>0
\end{equation*}
or
\begin{equation}\label{yamabe-ode}
v_{\tau}=(v^m)_{xx}+v-v^m,\quad v>0.
\end{equation} 
Hence existence of ancient radially symmetric solutions of \eqref{yamabe-flow} with metric $g=u^{\frac{4}{n-2}}dy^2$ is equivalent to the existence of ancient solutions of \eqref{yamabe-ode} in $\R\times (-\infty,T)$ for some constant $T\in\R$.

Existence of $5$-parameters and $4$-parameters ancient solutions of \eqref{yamabe-ode}
for some $T\in\R$ have been constructed by P.~Daskalopoulos, M.~del Pino, J.~King and N.~Sesum, in \cite{DPKS1}, \cite{DPKS2}. In this paper we will construct new $5$-parameters, $4$-parameters, and $3$-parameters ancient solutions of the equation \eqref{yamabe-ode}. As a result of our construction of solutions we obtain exact decay rate of the ancient solutions of \eqref{yamabe-ode}. 

A natural question to ask is whether the $4$-parameters ancient solutions of \eqref{yamabe-ode} that we construct is equal to the $4$-parameters ancient solutions of \eqref{yamabe-ode} constructed by P.~Daskalopoulos, M.~del Pino, J.~King, and N.~Sesum in \cite{DPKS2}. We answer this in the affirmative and prove that the $4$-parameters ancient solutions that we construct coincide with the $4$-parameters ancient solutions of \eqref{yamabe-ode} constructed by P.~Daskalopoulos, M.~del Pino, J.~King, and N.~Sesum in \cite{DPKS2}. In particular under a mild decay condition on the $4$-parameters ancient solutions of \eqref{yamabe-ode} we prove the uniqueness of the $4$-parameters ancient solutions. 

In the paper \cite{HN} F.~Hamel and N.~Nadirashvili proved 
various properties of the ancient solutions of the equation

\begin{equation}\label{kpp-eqn}
u_t=u_{xx}+f(u), u>0,\quad\mbox{ in }\R\times (-\infty,T)
\end{equation}
for some $T\in\R$ where $f(s)>0$ for $0<s<1$, $f(0)=f(1)=0$, $f'(0)>0>f'(1)$ and $f'(s)\le f'(0)$ for any $s\in [0,1]$. In this paper we will prove that many properties of the ancient solutions of \eqref{kpp-eqn} remains valid for the ancient solutions of \eqref{yamabe-ode}. In particular  
the $4$-parameters ancient solution is the singular limit solution of the $5$-parameters ancient solutions and the $3$-parameters ancient solution is the singular limit solution of the $4$-parameters ancient solution, etc.  

Let $u$ be a radially symmetric solution of \eqref{u-eqn} and $\2{u}(y,t)=u(y,t)^p$. Then $\2{u}$ satisfies
\begin{equation}\label{fde}
\2{u}_t=\frac{n-1}{m}\Delta\2{u}^m 
\end{equation}
and
\begin{equation}\label{tilde-u-v-relation}
\2{u}(y,t)=[(n-1)(n-2)(T-t)]^{\frac{1}{1-m}}|y|^{-\frac{2}{1-m}}v\left(\frac{n-2}{2}x,\frac{n+2}{4}\tau\right)
\end{equation}
where $x=\log |y|$, $\tau=-\log (T-t)$, and $v$ is given by \eqref{hat-u-defn}. If $\2{u}$ is a backward radially symmetric self-similar solution of \eqref{fde} that is
\begin{equation}\label{self-similar-soln}
\2{u}(y,t)=(T-t)^{\alpha}f((T-t)^{\beta}|y|)
\end{equation}
where $f$ is a radially symmetric solution of 
\begin{equation}\label{elliptic-eqn}
\Delta f^m+\alpha f+\beta x\cdot\nabla f=0\quad\mbox{ in }\R^n
\end{equation}
and $\beta=\frac{\lambda}{2m}>0$, $\alpha=\frac{2\beta+1}{1-m}$, are some constants, then by the discussion in \cite{DKS} the corresponding function $v(x,\tau)=v_{\lambda}(x,\tau)=v_{\lambda}(x-\lambda\tau)$ of \eqref{tilde-u-v-relation} is a travelling wave solution of \eqref{yamabe-ode} in $\R\times\R$ with
\begin{equation*}
v_{\lambda}(\log(|z|^{\frac{n-2}{2}}))=[(n-1)(n-2)]^{-\frac{1}{1-m}}|z|^{\frac{2}{1-m}}f(z)\quad\forall z\in\R^n
\end{equation*}
or equivalent
\begin{equation}\label{v-f-relation}
v_{\lambda}(x)=[(n-1)(n-2)]^{-\frac{1}{1-m}}e^{px}f\left(e^{\frac{2x}{n-2}}\right)\quad\forall x\in\R.
\end{equation}
By Theorem 1.1 of \cite{H2} for any $\alpha=\frac{2\beta+1}{1-m}$,
\begin{equation*}\label{f-value-at-0}
\beta=\frac{\lambda}{2m}\ge\frac{m}{n-2-nm}=\frac{1}{2}\quad\Leftrightarrow\quad\lambda\ge m=\frac{n-2}{n+2}
\end{equation*}
and $\mu>0$, there exists a unique radially symmetric solution $f$ of \eqref{elliptic-eqn} satisfying $f(0)=\mu$.
By Theorem 1.2 of \cite{H3} and \eqref{v-f-relation} when 
\begin{equation*}
\beta=\frac{\lambda}{2m}\ge\frac{1}{n-2}\quad\Leftrightarrow\quad\lambda\ge\frac{2}{n+2},
\end{equation*} 
the corresponding travelling wave solution $v_{\lambda}(x-\lambda\tau)$ of \eqref{yamabe-ode} in $\R\times \R$ satisfies 
\begin{equation}\label{v-lambda-+limit}
\lim_{x\to -\infty}v_{\lambda}(x)=0, \quad \lim_{x\to\infty}v_{\lambda}(x)=1
\end{equation}
and 
\begin{equation}\label{v-lambda-x=-infty-behavior}
e^{-px}v_{\lambda}(x)\approx C\quad\mbox{ as }x\to -\infty
\end{equation}
for some constant $C>0$. Then by the intermediate value theorem there exists $x_0\in\R$ such that $v_{\lambda}(x_0)=\frac{1}{2}$. By translation if necessary we may assume that $v_{\lambda}(0)=\frac{1}{2}$.
Hence for any $\lambda\ge\max(\frac{n-2}{n+2},\frac{2}{n+2})$ there exists a travelling wave solution $v_{\lambda}(x,\tau)=v_{\lambda}(x-\lambda\tau)$ of \eqref{yamabe-ode}  which satisfies $v_{\lambda}(0)=\frac{1}{2}$ and $v_{\lambda}(x)$ satisfies 
\begin{equation}\label{v-lambda-eqn}
(v^m)_{xx}+\lambda v_x+v-v^m=0,\quad v>0, \quad\mbox{ in }\R
\end{equation} 
and \eqref{v-f-relation} for some radially symmetric solution $f$ of \eqref{elliptic-eqn}. By the results of \cite{DKS}, \cite{DPKS1} and\cite{DPKS2}, for any $\lambda>1$, there exist positive constants $C_{\lambda}$ and $\gamma_{\lambda}$ such that
\begin{equation}\label{v-lambda-x-infty-behavior}
v_{\lambda}(x)=1-C_{\lambda}e^{-\gamma_{\lambda}x}+o(e^{-\gamma_{\lambda}x})\quad\mbox{ and }\quad
v_{\lambda}'(x)=C_{\lambda}\gamma_{\lambda}e^{-\gamma_{\lambda}x}+o(e^{-\gamma_{\lambda}x})\quad\mbox{ as }x\to\infty
\end{equation}
where 
\begin{equation*}
\gamma_{\lambda}=\frac{\lambda p -\sqrt{\lambda^2 p^2-4(p-1)}}{2}
\end{equation*} 
is the smallest root of the equation
\begin{equation}\label{gamma-eqn}
\gamma^2-\lambda p\gamma +p-1=0. 
\end{equation}
Note that for any $\lambda>1$, $\lambda'>1$, $h,h'\in\R$, both $v_{\lambda,h}(x,\tau):=v_{\lambda}(x-\lambda\tau +h)$ and $\2{v}_{\lambda',h'}(x,\tau):=v_{\lambda'}(-x-\lambda'\tau +h')$ are travelling wave solutions of \eqref{yamabe-ode}. By \eqref{v-lambda-x=-infty-behavior} and
\eqref{v-lambda-x-infty-behavior},
\begin{equation}\label{v-lambda-h-value-at-infty}
v_{\lambda,h}(x,\tau)=O(e^{p(x-\lambda\tau +h)})\quad\mbox{ as }x\to -\infty\quad\mbox{ and }\quad\2{v}_{\lambda',h'}(x,\tau)=O(e^{p(-x-\lambda'\tau +h')})\quad\mbox{ as }x\to\infty
\end{equation}
and
\begin{equation}\label{v-lambda-h-value-at-infty2}
\left\{\begin{aligned}
&v_{\lambda,h}(x,\tau)=1-C_{\lambda}e^{-\gamma_{\lambda}(x-\lambda\tau +h)}+o(e^{-\gamma_{\lambda}(x-\lambda\tau +h)})\quad\mbox{ as }\quad x-\lambda\tau +h\to\infty\\
&\2{v}_{\lambda',h'}(x,\tau)=1-C_{\lambda'}e^{-\gamma_{\lambda'}(-x-\lambda'\tau +h')}+o(e^{-\gamma_{\lambda'}(-x-\lambda'\tau +h')})\quad\mbox{ as }\quad -x-\lambda'\tau +h'\to \infty.
\end{aligned}\right.
\end{equation}
Note that by \eqref{v-lambda-x=-infty-behavior} and \eqref{v-lambda-eqn} there exists a constant $C>0$ such that
\begin{equation}\label{v-lambda--infty-eqn}
((v_{\lambda}^m)_x+\lambda v_{\lambda})_x=v_{\lambda}^m-v_{\lambda}=Ce^x+o(e^x)\quad\mbox{ as }x\to-\infty.
\end{equation}
By \eqref{v-lambda-x=-infty-behavior} and the mean value theorem for any $i\in\mathbb{N}$ there exists a constant $x_i\in (-i-1,-i)$ such that
\begin{equation}\label{v'-lambda--infty}
|v_{\lambda}'(x_i)|=|v_{\lambda}(-i-1)-v_{\lambda}(-i)|\le Ce^{-pi}\to 0\quad\mbox{ as }i\to\infty.
\end{equation}
Integrating \eqref{v-lambda--infty-eqn} over $(x_i,x)$ and letting $i\to\infty$,  by \eqref{v-lambda-x=-infty-behavior} and \eqref{v'-lambda--infty},
\begin{align}
&(v_{\lambda}^m)_x+\lambda v_{\lambda}=Ce^x+o(e^x)\quad\mbox{ as }x\to-\infty\notag\\
\Rightarrow\quad &(v_{\lambda}^m)_x=Ce^x+o(e^x)\quad\mbox{ as }x\to-\infty\notag\\
\Rightarrow\quad &v_{\lambda,x}(x)=v_{\lambda}(x)^{1-m}(C'e^x+o(e^x))\quad\mbox{ as }x\to-\infty\notag\\
\Rightarrow\quad &v_{\lambda,x}(x)=C''e^{px}+o(e^{px})\quad\mbox{ as }x\to-\infty\label{v-lambda--infty-behaviour}
\end{align}
for some constants $C'>0$, $C''>0$. Note that by \eqref{v-lambda-x-infty-behavior} and \eqref{v-lambda--infty-behaviour} $v_{\lambda}'\in L^{\infty}(\R)$ for any $\lambda>1$.
Let $k_0>0$, $h_0,h_0'\in\R$. We choose $\tau_0'\in\R$ such that $k_0e^{\frac{p-1}{p}\tau_0'}<1/2$. Let
\begin{equation*}\label{xi-k-defn}
\xi_k(\tau)=(1-ke^{\frac{p-1}{p}\tau})^{\frac{p}{p-1}}\quad\forall 0<k\le k_0,\tau\le\tau_0'.
\end{equation*}
Then by direct computation $\xi_k$ satisfies (cf. \cite{DPKS2}),
\begin{equation*}
\xi_k'(\tau)=\xi_k(\tau)-\xi_k(\tau)^m\quad\forall 0<k\le k_0,\tau\le\tau_0'
\end{equation*}
and
\begin{equation}\label{xi-k-infinity}
\xi_k(\tau)=1-\frac{pk}{p-1}e^{\frac{p-1}{p}\tau}+o(e^{\frac{p-1}{p}\tau})\quad\mbox{ as }\tau\to -\infty.
\end{equation}
Hence $\xi_k$ is a solution of \eqref{yamabe-ode} in $\R\times (-\infty,\tau_0')$. For any $\lambda>1$, $\lambda'>1$, $h,h'\in\R$, $0<k\le k_0$, let 
\begin{align*}
&f_{\lambda,\lambda',h,h',k}(x,\tau)=(v_{\lambda,h}(x,f(\tau))^{1-p}+\2{v}_{\lambda',h'}(x,f(\tau))^{1-p}+\xi_k(\tau)^{1-p}-2)^{-\frac{1}{p-1}}\quad\forall x\in\R,\tau\le\tau_0',\\
&\2{f}_{\lambda,\lambda',h,h',k}(x,\tau)=min (v_{\lambda,h}(x,f(\tau)), \2{v}_{\lambda',h'}(x,f(\tau)),\xi_k(\tau))\quad\forall x\in\R,\tau\le\tau_0',\\
&f_{\lambda,\lambda',h,h'}(x,\tau)=(v_{\lambda,h}(x,f(\tau))^{1-p}+\2{v}_{\lambda',h'}(x,f(\tau))^{1-p}-1)^{-\frac{1}{p-1}}\quad\forall x\in\R,\tau\in\R,\\
&\2{f}_{\lambda,\lambda',h,h'}(x,\tau)=min (v_{\lambda,h}(x,f(\tau)), \2{v}_{\lambda',h'}(x,f(\tau)))\quad\forall x\in\R,\tau\in\R,\\
&f_{\lambda,h,k}(x,\tau)=(v_{\lambda,h}(x,f(\tau))^{1-p}+\xi_k(\tau)^{1-p}-1)^{-\frac{1}{p-1}}\quad\forall x\in\R,\tau\le\tau_0',\\
&\2{f}_{\lambda,h,k}(x,\tau)=min (v_{\lambda,h}(x,f(\tau)),\xi_k(\tau))\quad\forall x\in\R,\tau\le\tau_0',
\end{align*}
where
\begin{equation}\label{f-defn}
f(\tau)=\tau\left(1+qe^{\frac{p-1}{p}\tau}\right)\quad\forall\tau\le\tau_0'
\end{equation}
for some constant $q=q(p)>0$. 
Note that 
\begin{equation*}
f_{\lambda,\lambda',h,h',k}\le\min (f_{\lambda,\lambda',h,h'},f_{\lambda,h,k})\quad\mbox{ and }\quad
\2{f}_{\lambda,\lambda',h,h',k}\le\min (\2{f}_{\lambda,\lambda',h,h'}\, ,\2{f}_{\lambda,h,k})
\quad\mbox{ in }\R\times (-\infty,\tau_0')
\end{equation*}
for any $\lambda>1$, $\lambda'>1$, $h\ge h_0$, $h'\ge h_0'$ and $0<k\le k_0$. Note that for $\tau<0$ sufficiently small the last two terms of (3.5) of \cite{DPKS2} are negative only if the constant $q$ there is negative instead of positive since $(\tau e^{\frac{p-1}{p}\tau})'<0$ for $\tau<0$ sufficiently small. Hence the constant $q$ in (1.30) of \cite{DPKS2} should be negative instead of positive in order for the function $w_{\lambda,\lambda',h,h',k}$ there to be a supersolution of (1.16) of \cite{DPKS2}. Thus by the proof of  \cite{DPKS2} there exist constants $\tau_0=\tau_0(h_0,h_0',k_0)<\tau_0'$ and  $q=q(p)>0$ such that $f_{\lambda,\lambda',h,h',k}$ and $f_{\lambda,\lambda',h,h'}$  are  subsolutions of \eqref{yamabe-ode} in $\R\times (-\infty,\tau_0)$ for any $\lambda>1$, $\lambda'>1$, $h\ge h_0$, $h'\ge h_0$ and $0<k\le k_0$. By an argument similar to the proof of  \cite{DPKS2} we can choose the constants $\tau_0=\tau_0(h_0,h_0',k_0)<\tau_0'$ and  $q=q(p)>0$ such that  $f_{\lambda,h,k}$ is also a  subsolution of \eqref{yamabe-ode} in $\R\times (-\infty,\tau_0)$ for any $\lambda>1$,  $h\ge h_0$ and $0<k\le k_0$. By choosing $\tau_0<\min (\tau_0',-p/(p-1))$ to be sufficiently small we also have  $0<f'(\tau)<1$ for any $\tau\le\tau_0$. We will assume $\lambda>1$ and $v_{\lambda}$ is the solution of  \eqref{v-lambda-eqn} which satisfies 
\begin{equation*}
v_{\lambda}(0)=1/2, \quad \lim_{x\to -\infty}v_{\lambda}(x)=0\quad\mbox{ and }\quad \lim_{x\to \infty}v_{\lambda}(x)=1
\end{equation*} 
for the rest of the paper.
In this paper we will prove the following main results.

\begin{thm}\label{5-parameters-soln-thm}
There exists $\2{\tau}_0<\tau_0$ such that for any $\lambda>1$, $\lambda'>1$, $h\ge h_0$, $h'\ge h_0'$ and $0<k\le k_0$,  there exists a solution $v=v_{\lambda,\lambda',h,h',k}\in C^{2,1}(\R\times (-\infty,\2{\tau}_0])$ of 
\eqref{yamabe-ode} in $\R\times (-\infty,\2{\tau}_0)$ which satisfies 
\begin{equation}\label{v-lambda-lambda'-h-h'-k-lower-upper-bd}
f_{\lambda,\lambda',h,h',k}(x,\tau)\le v_{\lambda,\lambda',h,h',k}(x,\tau)\le\2{f}_{\lambda,\lambda',h,h',k}(x,\tau)\quad\forall x\in\R,\tau<\2{\tau}_0
\end{equation}
and the following holds.
\begin{enumerate}

\item[(i)] For any $x\in\R$, $v(x,\tau)$ is a decreasing function of $\tau$ and $v(x,\tau)\to 1$ as $\tau\to -\infty$.

\item[(ii)]  For any $\tau<\2{\tau}_0$, $v(x,\tau)\to 0$ as $|x|\to\infty$ and there exists $x_0(\tau)\in\R$ such that $v_x(x_0(\tau),\tau)=0$, $v_x(x,\tau)>0$ if $x<x_0(\tau)$ and $v_x(x,\tau)<0$ if $x>x_0(\tau)$. Furthermore if $\lambda=\lambda'$, then $x_0(\tau)=\frac{h'-h}{2}$ and $v(x,\tau)$ is symmetric with respect to $x_0:=\frac{h'-h}{2}$ for any $\tau<\2{\tau}_0$.

\item[(iii)] $v(x_0(\tau),\tau)=\underset{x\in\R}\max\, v(x,\tau)\approx\xi_k(\tau)$ as $\tau\to -\infty$.

\item[(iv)] $v_{\lambda,\lambda',h,h',k}$ is increasing in $h$ and $h'$ and decreasing in $0<k\le k_0$.

\item[(v)] As $\tau\to -\infty$, we have:

\begin{enumerate}
\item[(a)] If $c>\lambda$, then $v(x+c\tau,\tau)\to 0$ uniformly on $(-\infty,A]$ for any $A\in\R$.

\item[(b)] If $-\lambda'<c<\lambda$, then $v(x+c\tau,\tau)\to 1$ uniformly in any compact subset of $\R$.

\item[(c)] If $c<-\lambda'$, then $v(x+c\tau,\tau)\to 0$ uniformly on $[A,\infty)$ for any $A\in\R$.

\item[(d)] $v(x+\lambda\tau,\tau)\to v_{\lambda}(x+h)$ uniformly on $(-\infty,A]$ for any $A\in\R$.

\item[(e)] $v(x-\lambda'\tau,\tau)\to v_{\lambda'}(-x+h')$ uniformly on $[A,\infty)$ for any $A\in\R$.

\end{enumerate}

\end{enumerate}

\end{thm}

\begin{thm}\label{4-parameters-soln-thm}
There exists a constant $\2{\tau}_0<\tau_0$ such that for any $\lambda>1$, $\lambda'>1$, $h\ge h_0$, $h'\ge h_0'$, there exists a unique solution $v=v_{\lambda,\lambda',h,h'}\in C^{2,1}(\R\times (-\infty,\2{\tau}_0])$ of \eqref{yamabe-ode} in $\R\times (-\infty,\2{\tau}_0)$ which satisfies 
\begin{equation}\label{v-lambda-lambda'-h-h'-upper-lower-bd}
f_{\lambda,\lambda',h,h'}(x,\tau)\le v_{\lambda,\lambda',h,h'}(x,\tau)\le \2{f}_{\lambda,\lambda',h,h'}(x,\tau)\quad\forall x\in\R,\tau<\2{\tau}_0.
\end{equation}
and 
\begin{equation}\label{v-v-compare}
v_{\lambda,\lambda',h,h'}(x,\tau)\ge v_{\lambda,\lambda',h,h',k}(x,\tau)\quad\forall x\in\R,\tau<\2{\tau}_0,0<k\le k_0
\end{equation}
where $v_{\lambda,\lambda',h,h',k}$ is as constructed in Theorem \ref{5-parameters-soln-thm}.
Moreover the following holds.
\begin{enumerate}
\item[(i)] For any $x\in\R$, $v(x,\tau)$ is a decreasing function of $\tau$ and $v(x,\tau)\to 1$ as $\tau\to -\infty$.

\item[(ii)]  For any $\tau<\2{\tau}_0$, $v(x,\tau)\to 0$ as $|x|\to\infty$ and there exists $x_0(\tau)\in\R$ such that $v_x(x_0(\tau),\tau)=0$, $v_x(x,\tau)>0$ if $x<x_0(\tau)$ and $v_x(x,\tau)<0$ if $x>x_0(\tau)$. Further more if $\lambda=\lambda'$, then $x_0(\tau)=\frac{h'-h}{2}$ and $v(x,t)$ is symmetric with respect to $x_0:=\frac{h'-h}{2}$ for any $\tau<\2{\tau}_0$.

\item[(iii)] $v(x_0(\tau),\tau)=\underset{x\in\R}\max\, v(x,\tau)\approx\underset{x\in\R}\max\,\2{f}_{\lambda,\lambda',h,h'}(x,\tau)$ as $\tau\to -\infty$.

\item[(iv)] $v_{\lambda,\lambda',h,h'}$ is increasing in $h$ and $h'$.

\item[(v)] As $\tau\to -\infty$, we have:

\begin{enumerate}
\item[(a)] If $c>\lambda$, then $v(x+c\tau,\tau)\to 0$ uniformly on $(-\infty,A]$ for any $A\in\R$.

\item[(b)] If $-\lambda'<c<\lambda$, then $v(x+c\tau,\tau)\to 1$ uniformly in any compact subset of $\R$.

\item[(c)] If $c<-\lambda'$, then $v(x+c\tau,\tau)\to 0$ uniformly on $[A,\infty)$ for any $A\in\R$.

\item[(d)] $v(x+\lambda\tau,\tau)\to v_{\lambda}(x+h)$ uniformly on $(-\infty,A]$ for any $A\in\R$.

\item[(e)] $v(x-\lambda'\tau,\tau)\to v_{\lambda'}(-x+h')$ uniformly on $[A,\infty)$ for any $A\in\R$.

\end{enumerate}

\item[(vi)] As $h\to\infty$, $v_{\lambda,\lambda',h,h'}$ converges uniformly on  $[A,\infty)\times [\tau_1,\2{\tau}_0]$ for any  $A\in\R$ and $\tau_1<\2{\tau}_0$ to $\2{v}_{\lambda',h'}(x,f(\tau))$. Similarly as $h'\to\infty$, $v_{\lambda,\lambda',h,h'}$ converges uniformly on $(-\infty,A]\times [\tau_1,\2{\tau}_0]$ for any  $A\in\R$ and $\tau_1<\2{\tau}_0$ to $v_{\lambda,h}(x,f(\tau))$.

\item[(vii)] As $k\to 0$, the solution $v_{\lambda,\lambda',h,h',k}$ of \eqref{yamabe-ode} in $\R\times (-\infty,\tau_0)$ given by Theorem \ref{5-parameters-soln-thm} increases and converges uniformly in $C^{2,1}(K)$ for any  every compact subset $K$ of $\R\times (-\infty,\2{\tau}_0]$
to  the unique solution  $v_{\lambda,\lambda',h,h'}$ of \eqref{yamabe-ode} in $\R\times (-\infty,\2{\tau}_0)$ which satisfies \eqref{v-lambda-lambda'-h-h'-upper-lower-bd}.  

\end{enumerate}
\end{thm}

\begin{thm}\label{3-parameters-soln-thm}
There exists a constant $\2{\tau}_0<\tau_0$ such that for any $\lambda>1$, $h\ge h_0$ and $0<k\le k_0$, there exists a solution $v=v_{\lambda,h,k}\in C^{2,1}(\R\times (-\infty,\2{\tau}_0])$ of 
\eqref{yamabe-ode} in $\R\times (-\infty,\tau_0)$ which satisfies 
\begin{equation}\label{v-lambda-h-k-upper-lower-bd}
f_{\lambda,h,k}(x,\tau)\le v_{\lambda,h,k}(x,\tau)\le\2{f}_{\lambda,h,k}(x,\tau)\quad\forall x\in\R,\tau<\2{\tau}_0
\end{equation}
and
\begin{equation}\label{v-v-compare2}
v_{\lambda,h,k}(x,\tau)\ge v_{\lambda,\lambda',h,h',k}(x,\tau)\quad\forall x\in\R,\tau<\2{\tau}_0,\lambda'>1, h'\ge h_0'.
\end{equation}
where $v_{\lambda,\lambda',h,h',k}$ is as constructed in Theorem \ref{5-parameters-soln-thm}. Moreover the following holds.
\begin{enumerate}

\item[(i)] For any $x\in\R$, $v(x,\tau)$ is a decreasing function of $\tau$ and $v(x,\tau)\to 1$ as $\tau\to -\infty$.

\item[(ii)] For any $\tau<\2{\tau}_0$, $v_x(x,\tau)>0$ for any $x\in\R$. 
\item[(iii)]  For any $\tau<\2{\tau}_0$, $v(x,t)\to 0$ as $x\to -\infty$ and  $v(x,t)\to\xi_k(\tau)$ as $x\to\infty$.

\item[(iv)] $v_{\lambda,h,k}$ is increasing in $h$ and decreasing in $0<k\le k_0$. 

\item[(v)] As $\tau\to -\infty$, we have:

\begin{enumerate}
\item[(a)] If $c>\lambda$, then $v(x+c\tau,\tau)\to 0$ uniformly on $(-\infty,A]$ for any $A\in\R$.

\item[(b)] If $c<\lambda$, then $v(x+c\tau,\tau)\approx \xi_k(\tau)$ uniformly on $[A,\infty)$ for any $A\in\R$.

\item[(c)] $v(x+\lambda\tau,\tau)\to v_{\lambda}(x+h)$ uniformly on any compact subset of $\R$.

\end{enumerate}

\item[(vi)] As $k\to 0$, $v_{\lambda,h,k}$ converges uniformly on every compact subset of $\R\times (-\infty,\2{\tau}_0]$ to  $v_{\lambda,h}(x,f(\tau))$.

\item[(vii)] As $h\to\infty$, $v_{\lambda,h,k}$ converges to $\xi_k$ uniformly on $[A,\infty)\times [\tau_1,\2{\tau}_0]$ for any $A\in\R$ and $\tau_1<\2{\tau}_0$. 

\item[(viii)] As $h'\to\infty$, the solution $v_{\lambda,\lambda',h,h',k}$ of \eqref{yamabe-ode} in $\R\times (-\infty,\tau_0)$ given by Theorem \ref{5-parameters-soln-thm} increases and converges uniformly in $C^{2,1}(K)$ for any  every compact subset $K$ of $\R\times (-\infty,\2{\tau}_0]$ to a solution $v$ of \eqref{yamabe-ode} in $\R\times (-\infty,\tau_0)$  which  satisfies \eqref{v-lambda-h-k-upper-lower-bd} and \eqref{v-v-compare2} with $v_{\lambda,h,k}$ there being replaced by $v$.

\end{enumerate}

\end{thm}

\begin{rmk}
By (iii) and (v) of Theorem \ref{5-parameters-soln-thm}, (v) of Theorem \ref{4-parameters-soln-thm} and (iii) and (v) of Theorem \ref{3-parameters-soln-thm}, for any $\lambda_1>1$,
$\lambda_2>1$, $\lambda_1'>1$, $\lambda_2'>1$, $h_1\ge h_0$, $h_1'\ge h_0'$, $h_2\ge h_0'$, $h_2'\ge h_0'$, $k_1, k_2\in (0,k_0]$, the following holds.
\begin{enumerate}
\item[(i)] $v_{\lambda_1,\lambda_1',h_1,h_1',k_1}\ne v_{\lambda_2,\lambda_2',h_2,h_2',k_2}$\quad\mbox{ if }\quad$(\lambda_1,\lambda_1',h_1,h_1', k_1)\ne (\lambda_2,\lambda_2',h_2,h_2',k_2)$
\item[(ii)] $v_{\lambda_1,\lambda_1',h_1,h_1'}\ne v_{\lambda_2,\lambda_2',h_2,h_2'}$\qquad\mbox{ if }\quad$(\lambda_1,\lambda_1',h_1,h_1')\ne (\lambda_2,\lambda_2',h_2,h_2')$
\item[(iii)] $v_{\lambda_1,h_1,k_1}\ne v_{\lambda_2,h_2,k_2}$\qquad\qquad\mbox{ if }\quad$(\lambda_1,h_1,k_1)\ne (\lambda_2,h_2,k_2)$.
\end{enumerate}
\end{rmk}

\begin{rmk}
By choosing $\2{\tau}_0$ to be sufficiently small we may assume without loss of generality that the constant $\2{\tau}_0$ is the same in Theorem \ref{5-parameters-soln-thm}, Theorem \ref{4-parameters-soln-thm} and Theorem \ref{3-parameters-soln-thm}.
\end{rmk}

\begin{rmk}
Existence of $4$-parameters and $5$-parameters solutions of \eqref{yamabe-ode}  are also constructed in \cite{DPKS1} and \cite{DPKS2}. However properties (ii)--(v) of Theorem \ref{5-parameters-soln-thm} for the $5$-parameters solutions of \eqref{yamabe-ode}, properties (ii)--(vii) of Theorem \ref{4-parameters-soln-thm} for the $4$-parameters solutions of \eqref{yamabe-ode} and the uniqueness of  the $4$-parameters solutions $v_{\lambda,\lambda',h,h'}$ in Theorem \ref{4-parameters-soln-thm}  are new results.
\end{rmk}

The plan of the paper is as follows. In section 2 we will prove the comparison principle and local existence solutions of \eqref{yamabe-ode}. In section 3 we will prove the existence of various ancient solutions and various properties of the ancient solutions of \eqref{yamabe-ode} stated in Theorem \ref{5-parameters-soln-thm}, Theorem \ref{4-parameters-soln-thm} and Theorem \ref{3-parameters-soln-thm}.

We start with some definitions. For any $R>0$ and $l\in\mathbb{N}$, let $B_R=\{x\in\R^l:|x|\le R\}$. For any $\tau_2>\tau_1$,  we say that $v$ is a solution of \eqref{yamabe-ode} in $\R\times (\tau_1,\tau_2)$ ($B_R\times (\tau_1,\tau_2)$, respectively) if $0<v\in C^2(\R)$ ($0<v\in C^2(B_R)$, respectively) is a classical solution of \eqref{yamabe-ode} in $\R\times (\tau_1,\tau_2)$ ($B_R\times (\tau_1,\tau_2)$, respectively). We say that $v$ is a subsolution (supersolution, respectively) of \eqref{yamabe-ode} in $\R\times (\tau_1,\tau_2)$ if $0<v\in C(\R\times (\tau_1,\tau_2))\cap L^{\infty}(\R\times (\tau_1,\tau_2))$ satisfies
\begin{align*}
\int_{B_R}v(x,\tau_4)\eta (x,\tau_4)\,dx\le &\int_{\tau_3}^{\tau_4}\int_{B_R}\left[v\eta_{\tau}+v^m\eta_{xx}+(v-v^m)\eta\right]\,dx\,d\tau
-\int_{\tau_3}^{\tau_4}\int_{\1 B_R}v^m\frac{\1 \eta}{\1 n}\,d\sigma\,d\tau\\
&\qquad +\int_{B_R}v(x,\tau_3)\eta (x,\tau_3)\,dx\quad\forall \tau_1<\tau_3<\tau_4<\tau_2
\end{align*}
($\ge$, respectively) for any $R>0$ and function $0\le\eta\in C^{2,1}(\2{B}_R\times [\tau_3,\tau_4])$ satisfying $\eta\equiv 0$ on $\1 B_R\times [\tau_3,\tau_4]$ where $\1/\1 n$ is the exterior normal derivative with respect to the unit outward normal on $\1 B_R$. For any $0\le v_0\in L_{loc}^1(\R)$, we say that a solution (subsolution, supersolution, respectively) $v$ of \eqref{yamabe-ode} in $\R\times (\tau_1,\tau_2)$ has initial value $v_0$ at $\tau_1$ if  $v(\cdot,\tau)\to v_0$ in $L^1_{loc}(\R)$ as $\tau\searrow\tau_1$.

For any $R>0$, $0\le v_0\in L^{\infty}(-R,R)$, $0\le g_0\in L^{\infty}(\{\pm R\}\times (\tau_1,\tau_2))$ we say that $v$ is a solution (subsolution, supersolution, respectively) of 
\begin{equation}\label{v-bd-cylinder-soln-eqn}
\left\{\begin{aligned}
&v_{\tau}=(v^m)_{xx}+v-v^m\quad\mbox{ in }(-R,R)\times (\tau_1,\tau_2)\\
&v=g_0\qquad\qquad\qquad\mbox{ on }\{\pm R\}\times [\tau_1,\tau_2)\\
&v(x,\tau_1)=v_0(x)\qquad\quad\mbox{ on }(-R,R)
\end{aligned}\right.
\end{equation}
if $0<v\in C((-R,R)\times (\tau_1,\tau_2))\cap L^{\infty}((-R,R)\times (\tau_1,\tau_2))$ satisfies
\begin{align}\label{soln-integral-defn}
\int_{-R}^Rv(x,\tau_3)\eta (x,\tau_3)\,dx=&\int_{\tau_1}^{\tau_3}\int_{-R}^R\left[v\eta_{\tau}+v^m\eta_{xx}+(v-v^m)\eta\right]\,dx\,d\tau
-\int_{\tau_1}^{\tau_3}\int_{\1 B_R}g_0^m\frac{\1 \eta}{\1 n}\,d\sigma\,d\tau\notag\\
&\qquad +\int_{-R}^Rv_0(x)\eta(x,\tau_1)\,dx\quad\forall \tau_1<\tau_3<\tau_2
\end{align}
($\le, \ge$, respectively) for any function $0\le\eta\in C^{2,1}(\2{B}_R\times [\tau_1,\tau_3])$ satisfying $\eta\equiv 0$ on $\{\pm R\}\times [\tau_1,\tau_3]$  where $\1/\1 n$ is the derivative with respect to the unit normal on $\1 B_R$. 

For any $R>0$, $0\le v_0\in L^{\infty}(-R,R)$, $g_0\in L^{\infty}(\{\pm R\}\times [\tau_1,\tau_2))$ we say that $v$ is a solution (subsolution, supersolution, respectively) of 
\begin{equation}\label{v-bd-cylinder-neumann-eqn}
\left\{\begin{aligned}
&v_{\tau}=(v^m)_{xx}+v-v^m\quad\mbox{ in }(-R,R)\times (\tau_1,\tau_2)\\
&\frac{\1 v^m}{\1 n}=g_0\qquad\qquad\quad\mbox{ on }\{\pm R\}\times [\tau_1,\tau_2)\\
&v(x,\tau_1)=v_0(x)\qquad\quad\mbox{ on }(-R,R)
\end{aligned}\right.
\end{equation}
if $0<v\in C((-R,R)\times (\tau_1,\tau_2))\cap L^{\infty}((-R,R)\times (\tau_1,\tau_2))$ satisfies
\begin{align*}\label{neumann-soln-integral-defn}
\int_{-R}^Rv(x,\tau_3)\eta (x,\tau_3)\,dx=&\int_{\tau_1}^{\tau_3}\int_{-R}^R\left[v\eta_{\tau}+v^m\eta_{xx}+(v-v^m)\eta\right]\,dx\,d\tau
+\int_{\tau_1}^{\tau_3}\int_{\1 B_R}g_0\eta\,d\sigma\,d\tau\notag\\
&\qquad +\int_{-R}^Rv_0(x)\eta(x,\tau_1)\,dx\quad\forall \tau_1<\tau_3<\tau_2
\end{align*}
($\le, \ge$, respectively) for any function $0\le\eta\in C^{2,1}(\2{B}_R\times [\tau_1,\tau_3])$ satisfying $\eta_x\equiv 0$ on $\{\pm R\}\times [\tau_1,\tau_3]$. 

For any set $A\subset\R$, we let $\chi_A$ be the characteristic function of the set $A$. For any $a\in\R$, we let $a_+=\max(a,0)$ and $a_-=-\min(a,0)$.

\begin{rmk} Suppose $v$ is a solution of \eqref{yamabe-ode} in $B_R\times (\tau_1,\tau_2)$. Then since $inf_{B_{R'}\times [\tau_3,\tau_4]}v>0$ for any $0<R'<R$ and $\tau_1<\tau_3<\tau_4<\tau_2$, the equation
\eqref{yamabe-ode} for $v$ is uniformly parabolic on every compact subset of $B_R\times (\tau_1,\tau_2)$. Hence by the standard Schauder estimates \cite{LSU} and a bootrap argument 
$v\in C^{\infty}(B_R\times (\tau_1,\tau_2))$. 
\end{rmk}

\section{Local existence and comparison principles}
\setcounter{equation}{0}
\setcounter{thm}{0}

In this section we will prove various local existence and comparison principles for the solutions of \eqref{yamabe-ode}.

\begin{lem}\label{v-lambda-monotone-lem}
For any $\lambda>1$, 
\begin{equation}\label{v-lambda-increasing}
v_{\lambda}'(x)>0\quad\forall x\in\R.
\end{equation}
Hence 
\begin{equation}\label{v-lambda-lower-upper-bd}
0<v_{\lambda}(x)<1\quad\forall x\in\R, \lambda>1.
\end{equation} 
\end{lem}
\begin{proof}
This result is used without proof in \cite{DPKS1}. For the sake of completeness we will give a proof of this result here. 
By \eqref{v-f-relation} and Lemma 3.1 of \cite{H3}, for any $\lambda>1$, $x\in R$, 
\begin{align*}
v_{\lambda}(x)=&\frac{(r^2f(r)^{1-m})^{\frac{n+2}{4}}}{[(n-1)(n-2)]^{\frac{1}{1-m}}}\\
\Rightarrow\quad v_{\lambda}'(x)=&\frac{(n+2)(r^2f(r)^{1-m})^{\frac{n-6}{4}}}{4[(n-1)(n-2)]^{\frac{1}{1-m}}}(2rf(r)^{1-m}+(1-m)r^2f(r)^{-m}f'(r))\\
=&\frac{rf(r)^{1-m}(r^2f(r)^{1-m})^{\frac{n-6}{4}}}{[(n-1)(n-2)]^{\frac{1}{1-m}}}\left(\frac{2}{1-m}+\frac{rf'(r)}{f(r)}\right)\\
>&0
\end{align*}
where $r=e^{\frac{2x}{n-2}}$. By \eqref{v-lambda-+limit} and \eqref{v-lambda-increasing} we get \eqref{v-lambda-lower-upper-bd} and the lemma follows.
\end{proof}

\begin{lem}\label{sub-supersoln-comparison-lem}
Let $\tau_1<\tau_2$, $M>0$, and $v_1$, $v_2\in C(\R\times (\tau_1,\tau_2))\cap L^{\infty}(\R\times (\tau_1,\tau_2))$, be subsolution and supersolution of \eqref{yamabe-ode} in $\R\times (\tau_1,\tau_2)$ with initial values $v_{0,1}, v_{0,2}\in L^{\infty}(\R)$, at $\tau_1$  respectively 
such that $0<v_i\le M$ in $\R\times [\tau_1,\tau_2)$ for $i=1,2$ and for any constant $R>0$, there exists a constant $C_R>0$ such that
\begin{equation}\label{v1-v2-local-lower-bd}
v_i(x,\tau)\ge C_R\quad\forall |x|\le R,\tau_1<\tau<\tau_2,i=1,2. 
\end{equation}
Then 
\begin{equation}\label{integral-comparison-ineqn}
\int_{\R}(v_1(x,\tau)-v_2(x,\tau))_+\,dx\le\int_{\R}e^{(1-mM^{m-1})(\tau-\tau_1)}(v_{0,1}(x)-v_{0,2}(x))_+\,dx
\end{equation}
hold for any  $\tau_1\le\tau<\tau_2$ if $(v_{0,1}-v_{0,2})_+\in L^1(\R)$. Hence if $v_{0,1}(x)\le v_{0,2}(x)$ a.e. $x\in\R$, then
\begin{equation}\label{v1<v2}
v_1(x,\tau)\le v_2(x,\tau)\quad\forall x\in\R, \tau_1<\tau<\tau_2.
\end{equation}  
If $v_1$, $v_2$, are also solutions of \eqref{yamabe-ode} in $\R^n\times (\tau_1,\tau_2)$ and $v_{0,1}-v_{0,2}\in L^1(\R)$, then
\begin{equation}\label{integral-comparison-ineqn2}
\int_{\R}|v_1(x,\tau)-v_2(x,\tau)|\,dx\le\int_{\R}e^{(1-mM^{m-1})(\tau-\tau_1)}|v_{0,1}(x)-v_{0,2}(x)|\,dx
\end{equation}
holds for any $\tau_1\le\tau<\tau_2$.
\end{lem}
\begin{proof}
We will use a modification of the proof of Lemma 2.3 of \cite{DK} and Theorem 2.1 of \cite{PV} to prove the lemma. 
We first suppose that $(v_{0,1}-v_{0,2})_+\in L^1(\R)$. Let $\tau_1<\tau_3<\tau_2$, $R_0>0$, $\theta\in C_0^{\infty}(\R)$ such that $0\le\theta\le 1$ in $\R$ and supp $\theta\subset B_{R_0}$. Let 
\begin{equation*}
A=\left\{\begin{aligned}
&\frac{v_1^m-v_2^m}{v_1-v_2}\quad\mbox{ if }v_1\ne v_2\\
&mv_1^{m-1}\qquad\mbox{if }v_1=v_2
\end{aligned}\right.
\end{equation*}
Then 
\begin{equation*}
mM^{m-1}\le A(x,\tau)\le m C_R^{m-1}\quad\forall |x|\le R, \tau_1\le\tau<\tau_2, R>0.
\end{equation*}  
We choose a sequence of functions $\{A_k\}_{k=1}^{\infty}\subset C^{\infty}(\R\times [\tau_1,\tau_2))$ such that $A_k$ converges uniformly to $A$ on every compact subset of $\R\times (\tau_1,\tau_2)$ as $k\to\infty$ and
\begin{equation}\label{Ak-bd}
mM^{m-1}\le A_k(x,\tau)\le m C_{R+1}^{m-1}\quad\forall |x|\le R, \tau_1\le\tau\le\tau_2, k\in\mathbb{N}, R>0.
\end{equation}
For any $R>R_0+2$, $k\in\mathbb{N}$, let $\eta_{R,k}\in C^{\infty}(\2{B}_R\times [\tau_1,\tau_3])$ be the solution of 
\begin{equation}\label{adjoint-eqn}
\left\{\begin{aligned}
\eta_{\tau}+A_k\eta_{xx}+(1-A_k)\eta&=0\quad\mbox{ in }(-R,R)\times (\tau_1,\tau_3)\\
\eta&=0\quad\mbox{ on }\{\pm R\}\times [\tau_1,\tau_3]\\
\eta(x,\tau_3)&=\theta(x)\quad\forall x\in (-R,R).
\end{aligned}\right.
\end{equation}
By the maximum principle, $\eta_{R,k}\ge 0$ in $[-R,R]\times [\tau_1,\tau_3]$. Hence $\1 \eta_{R,k}/\1 n\le 0$ on $\{\pm R\}\times [\tau_1,\tau_3)$.
Then by \eqref{adjoint-eqn},
\begin{align}\label{eta-rk-bd}
&(\eta_{R,k})_{\tau}+A_k(\eta_{R,k})_{xx}+(1-mM^{m-1})\eta_{R,k}\ge 0\quad\mbox{ in }(-R,R)\times (\tau_1,\tau_3)\notag\\
\Rightarrow\quad&(e^{(1-mM^{m-1})\tau}\eta_{R,k})_{\tau}+A_k(e^{(1-mM^{m-1})\tau}\eta_{R,k})_{xx}\ge 0\quad\mbox{ in }(-R,R)\times (\tau_1,\tau_3)\notag\\
\Rightarrow\quad&0\le\eta_{R,k}(x,\tau)\le e^{(1-mM^{m-1})(\tau_3-\tau)}\theta (x)\le e^{(1-mM^{m-1})(\tau_3-\tau)}\quad\forall |x|\le R,\tau_1\le\tau\le\tau_3.
\end{align}
By \eqref{adjoint-eqn},
\begin{align*}
\frac{1}{2}\frac{\1}{\1\tau}\int_{\R}\eta_{R,k,x}^2\,dx=&\int_{\R}\eta_{R,k,x}(\eta_{R,k})_{x,\tau}\,dx
=-\int_{\R}(\eta_{R,k})_{xx}\eta_{R,k,\tau}\,dx\notag\\
=&\int_{\R}A_k\eta_{R,k,xx}^2\,dx+\int_{\R}(1-A_k)\eta_{R,k}\eta_{R,k,xx}\, dx.
\end{align*}
Hence
\begin{align}\label{eta-l2}
\int_{\tau_1}^{\tau_3}\int_{\R}A_k\eta_{R,k,xx}^2\,dx\,dt\le&\frac{1}{2}\int_{\R}\theta_x^2\,dx+\int_{\tau_1}^{\tau_3}\int_{\R}(A_k-1)\eta_{R,k}\eta_{R,k,xx}\,dx\,dt\notag\\
\le&\frac{1}{2}\int_{\R}\theta_x^2\,dx+\frac{1}{2}\int_{\tau_1}^{\tau_3}\int_{\R}A_k\eta_{R,k,xx}^2\,dx\,dt
+\frac{1}{2}\int_{\tau_1}^{\tau_3}\int_{\R}\frac{(A_k-1)^2}{A_k^2}\eta_{R,k}^2\,dx\,dt.
\end{align}
Thus by \eqref{Ak-bd}, \eqref{eta-rk-bd} and \eqref{eta-l2},
\begin{equation}\label{eta-xx-integral-bd}
\int_{\tau_1}^{\tau_3}\int_{\R}A_k\eta_{R,k,xx}^2\,dx\,dt\le\int_{\R}\theta_x^2\,dx+\int_{\tau_1}^{\tau_3}\int_{\R}\frac{(A_k-1)^2}{A_k^2}\eta_{R,k}^2\,dx\,dt
\le\int_{\R}\theta_x^2\,dx+C_R'
\end{equation}
for some constant $C_R'>0$.  Since $v_1$, $v_2$, are subsolution and supersolution of \eqref{yamabe-ode} in $\R\times (\tau_1,\tau_2)$, by \eqref{eta-rk-bd} and \eqref{eta-xx-integral-bd},
\begin{align}\label{v1-v2-integral-ineqn}
&\int_{B_{R_0}}(v_1-v_2)(x,\tau_3)\theta(x)\,dx\notag\\
\le &\int_{\tau_1}^{\tau_3}\int_{B_R}(v_1-v_2)\left[\eta_{R,k,\tau}+A(\eta_{R,k})_{xx}+(1-A)\eta_{R,k}\right]\,dx\,dt+\int_{B_R}(v_{0,1}(x)-v_{0,2}(x))\eta_{R,k}(x,\tau_1)\,dx\notag\\
&\qquad +\int_{\tau_1}^{\tau_3}\int_{\1 B_R}(v_2^m-v_1^m)\frac{\1 \eta_{R,k}}{\1 n}\,d\sigma \,dt\notag\\
\le&\int_{\tau_1}^{\tau_3}\int_{B_R}(v_1-v_2)\left[(A-A_k)(\eta_{R,k})_{xx}+(A_k-A)\eta_{R,k}\right]\,dx\,dt+\int_{B_R}(v_{0,1}(x)-v_{0,2}(x))\eta_{R,k}(x,\tau_1)\,dx\notag\\
&\qquad +\int_{\tau_1}^{\tau_3}\int_{\1 B_R}v_1^m\left|\frac{\1 \eta_{R,k}}{\1 n}\right|\,d\sigma \,dt\notag\\
\le&2M\left(\int_{\tau_1}^{\tau_3}\int_{B_R}\frac{(A-A_k)^2}{A_k}\,dx\,dt\right)^{\frac{1}{2}}\left(\int_{\R}\theta_x^2\,dx+C_R'\right)^{\frac{1}{2}}+C_1\int_{\tau_1}^{\tau_3}\int_{B_R}|A_k-A|\,dx\,dt\notag\\
&\qquad +\int_{B_R}(v_{0,1}(x)-v_{0,2}(x))\eta_{R,k}(x,\tau_1)\,dx
+M^m\int_{\tau_1}^{\tau_3}\int_{\1 B_R}\left|\frac{\1 \eta_{R,k}}{\1 n}\right|\,d\sigma\,dt.
\end{align}
for some constant $C_1>0$ where $\1/\1 n$ is the exterior derivative with respect to the unit outward normal on $\1 B_R$.
We will now estimate the last term on the right hand side of \eqref{v1-v2-integral-ineqn}. As observed in \cite{DPKS1} the function
\begin{equation*}
\phi_0(x)=\left(\frac{k_ne^{\frac{2x}{n-2}}}{1+e^{\frac{4x}{n-2}}}\right)^{\frac{n+2}{2}},\quad k_n=\sqrt{\frac{4n}{n-2}},
\end{equation*}
satisfies \eqref{v-lambda-eqn} with $\lambda=0$. Hence
\begin{equation}\label{phi-0-xx-eqn}
m\phi_{0,xx}=-\phi_0^{2-m}+\phi_0+m(1-m)\phi_0^{-1}\phi_{0,x}^2\quad\mbox{ in }\R.
\end{equation}
Note that $\phi_0(x)$ is an even function and it is monotone decreasing in $x\ge 0$. Moreover $\phi_0(x)\le C_2e^{-p|x|}$ for any $x\in\R$ where $C_2=k_n^{\frac{n+2}{2}}$ and 
\begin{align}\label{phi-0-x-limit}
&\phi_{0,x}(x)=\frac{n+2}{n-2}\cdot\left(\frac{1-e^{\frac{4x}{n-2}}}{1+e^{\frac{4x}{n-2}}}\right)\phi_0(x)\quad\forall x\in\R\notag\\
\Rightarrow\quad&\phi_{0,x}(x)\approx -p\,\mbox{sign}\,(x)\phi_0(x)\quad\mbox{ as }|x|\to\infty\quad\mbox{ and }\quad |\phi_{0,x}(x)|\le p\phi_0(x)\quad\forall x\in\R.
\end{align}
Let $\phi(x,\tau)=K_1e^{\tau_3-\tau}\phi_0(mx)$ where $K_1=\phi_0(m R_0)^{-1}$. We claim that $\phi$ is a supersolution of \eqref{adjoint-eqn}.
To prove the claim we observe that  by \eqref{phi-0-xx-eqn} and \eqref{phi-0-x-limit},
\begin{align*}
\phi_{xx}(x,\tau)=&m^2K_1e^{\tau_3-\tau}\phi_{0,xx}(mx)\\
=&mK_1e^{\tau_3-\tau}\left[-\phi_0(mx)^{2-m}+\phi_0(mx)+m(1-m)\phi_0(mx)^{-1}\phi_{0,x}(mx)^2\right]\\
\le&m\phi(x,\tau)\left[1+m(1-m)\phi_0(mx)^{-2}\phi_{0,x}(mx)^2\right]\\
\le&\phi(x,\tau)\quad\forall x\in\R,\tau_1\le\tau\le\tau_3.
\end{align*}
Hence
\begin{equation}\label{phi-mu-ineqn}
\phi_{\tau}+A_k\phi_{xx}+(1-A_k)\phi\le\left(-1+A_k+(1-A_k)\right)\phi=0\quad\forall x\in\R,\tau_1\le\tau\le\tau_3.
\end{equation}
Now 
\begin{equation}\label{phi-mu>theta}
\phi(x,\tau_3)=\frac{\phi_0(mx)}{\phi_0(m R_0)}\ge 1\ge\theta (x)\quad\forall |x|\le R_0.
\end{equation}
By \eqref{phi-mu-ineqn} and \eqref{phi-mu>theta}, $\phi$ is a supersolution of \eqref{adjoint-eqn}. Hence by the maximum principle,
\begin{equation}\label{eta-r-k<phi}
0\le\eta_{R,k}(x,\tau)\le\phi(x,\tau)\le C_2K_1e^{\tau_3-\tau}e^{-|x|}\quad\forall |x|\le R,\tau_1\le\tau\le\tau_3, k\in\mathbb{N}.
\end{equation} 
Let 
\begin{equation*}
H(x,\tau)=K_1e^{\tau_3-\tau}(R-|x|)\phi_0(m (R-1)).
\end{equation*}
Then by \eqref{eta-r-k<phi},
\begin{equation*}
\left\{\begin{aligned}
&H_{\tau}+A_kH_{xx}+(1-A_k)H=-A_kH\le 0\quad\forall R-1\le|x|\le R, \tau_1\le\tau\le\tau_3\\
&H(x,\tau)=\phi(R-1,\tau)\ge\eta_{R,k}(x,\tau)\qquad\quad\forall |x|=R-1, \tau_1\le\tau\le\tau_3\\
&H(x,\tau)\ge 0=\eta_{R,k}(x,\tau)\qquad\qquad\qquad\quad\forall |x|=R, \tau_1\le\tau\le\tau_3\\
&H(x,\tau_3)\ge 0=\eta_{R,k}(x,\tau_3)\qquad\qquad\qquad\forall R-1\le|x|\le R.
\end{aligned}\right.
\end{equation*}
Hence by the maximum principle in $(B_R\setminus B_{R-1})\times (\tau_1,\tau_3)$,
\begin{align}\label{eta-r-k-derivative-bd}
&0\le\eta_{R,k}(x,\tau)\le H(x,\tau)\quad\forall R-1\le |x|\le R, \tau_1\le\tau\le\tau_3,k\in\mathbb{N}\notag\\
\Rightarrow\quad&0\ge\frac{\1\eta_{R,k}}{\1 n}(x,\tau)\ge\frac{\1 H}{\1 n}(x,\tau)\ge -e^{\tau_3-\tau}K_1\phi_0(m (R-1))\ge -C_3e^{\tau_3-\tau}e^{-R}
\end{align}
for any $|x|=R>R_0+2$ and $\tau_1\le\tau\le\tau_3$ where $C_3=ek_n^{\frac{n+2}{2}}\phi_0(mR_0)^{-1}>0$.
Letting $k\to\infty$ in \eqref{v1-v2-integral-ineqn}, by \eqref{eta-rk-bd}, \eqref{eta-r-k<phi} and \eqref{eta-r-k-derivative-bd}, for any $R_0>0$, $R>R_0+2$ and $\theta\in C_0^{\infty}(B_{R_0})$ such that  $0\le\theta(x)\le 1$ for any $x\in\R$, we have for any $\tau_1<\tau_3<\tau_2$,
\begin{align}\label{v1-v2-compare2}
&\int_{B_{R_0}}(v_1(x,\tau_3)-v_2(x,\tau_3))\theta(x)\,dx\notag\\
\le&\int_{B_R}e^{(1-mM^{m-1})(\tau-\tau_1)}(v_{0,1}(x)-v_{0,2}(x))_+\,dx+2(\tau_3-\tau_1)C_3 M^me^{\tau_3-\tau_1}e^{-R}\notag\\
\Rightarrow\quad&\int_{B_{R_0}}(v_1(x,\tau_3)-v_2(x,\tau_3))\theta(x)\,dx
\le\int_{\R}e^{(1-mM^{m-1})(\tau-\tau_1)}(v_{0,1}(x)-v_{0,2}(x))_+\,dx\quad\mbox{ as }R\to\infty.
\end{align}
We now choose a sequence of functions $\{\theta_i\}_{I=1}^{\infty}\subset C_0^{\infty}(B_{R_0})$, $0\le\theta_i\le 1$ for all $i\in\mathbb{N}$, such that $\theta_i(x)\nearrow(\mbox{ sign} (v_1(x)-v_2(x)))\chi_{B_{R_0}}(x)$ in $B_{R_0}$ as $i\to\infty$. Putting $\theta=\theta_i$ and letting first $i\to\infty$ and then $R_0\to\infty$ in \eqref{v1-v2-compare2}, \eqref{integral-comparison-ineqn} follows. Hence if $v_{0,1}(x)\le v_{0,2}(x)$ a.e. $x\in\R$, then \eqref{v1<v2} holds. Similarly if $v_1$ and $v_2$ are also solutions of \eqref{yamabe-ode} in $\R^n\times (\tau_1,\tau_2)$ and $v_{0,1}-v_{0,2}\in L^1(\R)$, then for any $\tau_1\le\tau<\tau_2$,
\begin{equation}\label{integral-comparison-ineqn3}
\int_{\R}(v_1(x,\tau)-v_1(x,\tau))_-\,dx\le\int_{\R}e^{(1-mM^{m-1})(\tau-\tau_1)}(v_{0,1}(x)-v_{0,2}(x))_-\,dx.
\end{equation}
By \eqref{integral-comparison-ineqn} and \eqref{integral-comparison-ineqn3}, \eqref{integral-comparison-ineqn2} follows. 
\end{proof}

Similarly we have the following two lemmas.

\begin{lem}\label{bded-domain-comparison-lem}
Let $R>0$, $\tau_1<\tau_2$, $v_{0,1}, v_{0,2}\in L^{\infty}(-R,R)$ be such that $v_{0,2}\ge v_{0,1}\ge 0$ in $(-R,R)$ and let $g_1,g_2\in L^{\infty}(\{\pm R\}\times (\tau_1,\tau_2))$ be  such that $g_2\ge g_1\ge 0$ on $\{\pm R\}\times (\tau_1,\tau_2)$. Let $v_1, v_2\in C((-R,R)\times (\tau_1,\tau_2))\cap L^{\infty}((-R,R)\times (\tau_1,\tau_2))$ be subsolution and supersolution of \eqref{v-bd-cylinder-soln-eqn} in $(-R,R)\times (\tau_1,\tau_2)$ with $v_0=v_{0,1}, v_{0,2}$ and $g_0=g_1$, $g_2$ resepectively. Suppose $v_1$, $v_2$, satisfies \eqref{v1-v2-local-lower-bd} for some constant $C_R>0$. Then
\begin{equation*}
v_1(x,\tau)\le v_2(x,\tau)\quad\forall |x|\le R, \tau_1<\tau<\tau_2.
\end{equation*}  
\end{lem}

\begin{lem}\label{bded-domain-neumann-comparison-lem}
Let $R>0$, $\tau_1<\tau_2$, $v_{0,1}, v_{0,2}\in L^{\infty}(-R,R)$ be such that $v_{0,2}\ge v_{0,1}\ge 0$ in $(-R,R)$ and let $g_1,g_2\in L^{\infty}(\{\pm R\}\times (\tau_1,\tau_2))$ be  such that $g_2\ge g_1\ge 0$ on $\{\pm R\}\times (\tau_1,\tau_2)$. Let $v_1, v_2\in C((-R,R)\times (\tau_1,\tau_2))\cap L^{\infty}((-R,R)\times (\tau_1,\tau_2))$ be subsolution and supersolution of \eqref{v-bd-cylinder-neumann-eqn} in $(-R,R)\times (\tau_1,\tau_2)$ with $v_0=v_{0,1}, v_{0,2}$ and $g_0=g_1$, $g_2$ resepectively. Suppose $v_1$, $v_2$, satisfies \eqref{v1-v2-local-lower-bd} for some constant $C_R>0$. Then
\begin{equation*}
v_1(x,\tau)\le v_2(x,\tau)\quad\forall |x|\le R, \tau_1<\tau<\tau_2.
\end{equation*}  
\end{lem}

\begin{lem}\label{critical-pt-x-y-z-lem}
For any $\lambda>1$, $\lambda'>1$, $h,h'\in\R$ and $\tau\in\R$. Then there exists a constant $\2{\tau}_0<\tau_0$ such that the following holds.
\begin{enumerate}
\item[(i)] (cf. \cite{DPKS1}) For any $\tau\le\2{\tau}_0$, there exists a unique constant $x(\tau)\in\R$ such that
\begin{equation}\label{f-bar-value0}
\2{f}_{\lambda,\lambda',h,h'}(x,\tau)=\left\{\begin{aligned}
&v_{\lambda}(x-\lambda f(\tau) +h)\qquad\forall x\le x(\tau)\\
&v_{\lambda'}(-x-\lambda'f(\tau) +h')\quad\forall x\ge x(\tau)
\end{aligned}\right.
\end{equation}
where
\begin{equation}\label{x-tau-value}
x(\tau)=\frac{\gamma_{\lambda}-\gamma_{\lambda'}}{p}\tau+\frac{1}{\gamma_{\lambda}+\gamma_{\lambda'}}\left(\log\frac{C_{\lambda}}{C_{\lambda'}}+h'\gamma_{\lambda'}-h\gamma_{\lambda}\right)+o(1)
\end{equation}
and
\begin{align}\label{f-bar-max-value-expansion}
\2{f}_{\lambda,\lambda',h,h'}(x(\tau),\tau)=&\max_{x\in\R}\2{f}_{\lambda,\lambda',h,h'}(x,\tau)=v_{\lambda,h}(x(\tau),f(\tau))=\2{v}_{\lambda',h'}(x(\tau),f(\tau))\notag\\
=&1-C_{\lambda,\lambda',h,h'}e^{d\tau}+o(e^{d\tau})
\end{align}
for some constant $C_{\lambda,\lambda',h,h'}>0$ where 
\begin{equation}\label{d-defn}
d=\frac{\gamma_{\lambda}\gamma_{\lambda'}+p-1}{p}.
\end{equation}

\item[(ii)]  For any $\tau\le\2{\tau}_0$, there exist unique constants
$y(\tau)<x(\tau)<z(\tau)$ such that for any $0<k\le k_0$,
\begin{equation}\label{f-bar-value}
\2{f}_{\lambda,\lambda',h,h',k}(x,\tau)=\left\{\begin{aligned}
&v_{\lambda}(x-\lambda f(\tau) +h)\qquad\forall x\le y(\tau)\\
&\xi_k(\tau)\qquad\qquad\qquad\quad\forall y(\tau)\le x\le z(\tau)\\
&v_{\lambda'}(-x-\lambda'f(\tau) +h')\quad\forall x\ge z(\tau).
\end{aligned}\right.
\end{equation}
Moreover,
\begin{equation}\label{y-tau-z-tau-infinity}
\left\{\begin{aligned}
&y(\tau)=\frac{\gamma_{\lambda}}{p}\tau+\frac{1}{\gamma_{\lambda}}\log\left(\frac{(p-1)C_{\lambda}}{k}\right)-h+o(1)\\
&z(\tau)=-\frac{\gamma_{\lambda'}}{p}\tau-\frac{1}{\gamma_{\lambda'}}\log\left(\frac{(p-1)C_{\lambda'}}{k}\right)+h'+o(1)
\end{aligned}\right.
\qquad\mbox{ for }\,\tau\le\2{\tau}_0.
\end{equation} 
\end{enumerate} 
\end{lem}
\begin{proof}
Since $f(\tau)\approx\tau$ as $\tau\to -\infty$, (i) follows by an argument similar to the proof of Lemma 2.1 of \cite{DPKS1} and we only need to prove (ii). For any $M>0$ we let 
\begin{equation*}
x_1(\tau)=\frac{\gamma_{\lambda}}{p}\tau-M, \quad x_2(\tau)=\left(\frac{\gamma_{\lambda}-\gamma_{\lambda'}}{p}\right)\tau\quad\mbox{ and }\quad
x_3(\tau)=-\frac{\gamma_{\lambda'}}{p}\tau+M.
\end{equation*}
Then 
\begin{equation*}
x_1(\tau)<x_2(\tau)<x_3(\tau)\quad\forall \tau<0.
\end{equation*}
Hence by \eqref{gamma-eqn},
\begin{equation}\label{x1-2-3-infty-behaviour}
\left\{\begin{aligned}
&x_1(\tau)-\lambda f(\tau) +h=\frac{\gamma_{\lambda}-\lambda p}{p}\tau +h+o(1)
=-\frac{(p-1)}{p\gamma_{\lambda}}\tau -M+h+o(1)\to\infty\\
&x_2(\tau)-\lambda f(\tau) +h=\frac{\gamma_{\lambda}-\lambda p-\gamma_{\lambda'}}{p}\tau +h+o(1)
=-\frac{(\gamma_{\lambda}\gamma_{\lambda'}+p-1)}{p\gamma_{\lambda}}\tau +h+o(1)\to\infty\\
&-x_3(\tau)-\lambda' f(\tau)+h' +o(1)=\frac{\gamma_{\lambda'}-\lambda' p}{p}\tau -M +h'+o(1)
=-\frac{(p-1)}{p\gamma_{\lambda'}}\tau -M+h'+o(1)\to\infty
\end{aligned}\right.
\end{equation}
as $\tau\to-\infty$.
Let $0<k\le k_0$. By \eqref{v-lambda-h-value-at-infty2}, \eqref{xi-k-infinity} and \eqref{x1-2-3-infty-behaviour}, as $\tau\to -\infty$,
\begin{align}\label{v-lambda>xi-k}
v_{\lambda}(x_2(\tau)-\lambda f(\tau) +h)=&1-C_{\lambda}e^{\frac{(\gamma_{\lambda}\gamma_{\lambda'}+p-1)}{p}\tau -\gamma_{\lambda}h}
+o\left(e^{\frac{(\gamma_{\lambda}\gamma_{\lambda'}+p-1)}{p}\tau -\gamma_{\lambda}h}\right)\notag\\
=&1-o\left(e^{\frac{p-1}{p}\tau}\right)\notag\\
>&1-\frac{pk}{p-1}e^{\frac{p-1}{p}\tau}+o\left(e^{\frac{p-1}{p}\tau}\right)\notag\\
=&\xi_k(\tau)
\end{align}
and similarly,
\begin{equation}\label{v-lambda'>xi-k}
v_{\lambda'}(-x_2(\tau)-\lambda' f(\tau) +h')=1-C_{\lambda'}e^{\frac{(\gamma_{\lambda}\gamma_{\lambda'}+p-1)}{p}\tau -\gamma_{\lambda'}h'}
+o\left(e^{\frac{(\gamma_{\lambda}\gamma_{\lambda'}+p-1)}{p}\tau -\gamma_{\lambda'}h'}\right)
>\xi_k(\tau)\quad\mbox{ as }\tau\to -\infty.
\end{equation}
By Lemma \ref{v-lambda-increasing}, $v_{\lambda}(x-\lambda f(\tau) +h)$ is a strictly monotone increasing function from 0 to 1 and $v_{\lambda'}(-x-\lambda' f(\tau) +h')$ is a strictly monotone decreasing function from 1 to zero of $x\in\R$. Hence by \eqref{v-lambda-h-value-at-infty2}, \eqref{xi-k-infinity}, \eqref{v-lambda>xi-k}, and \eqref{v-lambda'>xi-k}, there exists a constant $\2{\tau}_0<\tau_0$ such that for any $\tau\le\2{\tau}_0$ there exist unique constants $y(\tau)<x(\tau)<z(\tau)$ such that for any $0<k\le k_0$ \eqref{f-bar-value} holds and $y(\tau)<x_2(\tau)<z(\tau)$. By \eqref{v-lambda-h-value-at-infty2}, \eqref{xi-k-infinity} and \eqref{x1-2-3-infty-behaviour}, as $\tau\to -\infty$,
\begin{equation*}
v_{\lambda}(x_1(\tau)-\lambda f(\tau) +h)=1-C_{\lambda}e^{\frac{(p-1)}{p}\tau +\gamma_{\lambda}(M-h)}
+o\left(e^{\frac{(p-1)}{p}\tau +\gamma_{\lambda}(M-h)}\right)
<1-\frac{pk}{p-1}e^{\frac{p-1}{p}\tau}+o(e^{\frac{p-1}{p}\tau})
=\xi_k(\tau)
\end{equation*}
if we choose $M$ to be sufficiently large. Hence by choosing $\2{\tau}_0$ sufficiently small and $M$  sufficiently large, 
\begin{equation}\label{x1-y-x2-compare}
x_1(\tau)<y(\tau)<x_2(\tau)\quad\forall\tau\le\2{\tau}_0.
\end{equation} 
Then by \eqref{x1-2-3-infty-behaviour} and \eqref{x1-y-x2-compare}, $y(\tau)-\lambda f(\tau)+h\to\infty$ as $\tau\to -\infty$. Hence for $\2{\tau}_0$  sufficiently small and $M$  sufficiently large, 
\begin{align}\label{y(tau)-value}
v_{\lambda, h}(y(\tau),f(\tau))=&1-C_{\lambda}e^{-\gamma_{\lambda}(y(\tau)-\lambda f(\tau)+h)}+o\left(e^{-\gamma_{\lambda}(y(\tau)-\lambda f(\tau)+h)}\right)\notag\\
=&1-C_{\lambda}e^{-\gamma_{\lambda}(y(\tau)-\lambda \tau+h)}+o\left(e^{-\gamma_{\lambda}(y(\tau)-\lambda \tau+h)}\right)\notag\\
=&1-\frac{pk}{p-1}e^{\frac{p-1}{p}\tau}+o\left(e^{\frac{p-1}{p}\tau}\right)=\xi_k(\tau)
\end{align}
holds for any $\tau\le\2{\tau}_0$. Similarly by \eqref{v-lambda-h-value-at-infty2}, \eqref{xi-k-infinity},  and \eqref{x1-2-3-infty-behaviour}, for $M$ large and $\2{\tau}_0$  sufficiently small,
we have 
\begin{align*}
&x_2(\tau)<z(\tau)<x_3(\tau)\quad\forall\tau\le\2{\tau}_0\\
\Rightarrow\quad&-z(\tau)-\lambda' f(\tau)+h'\to\infty\quad\mbox{ as }\quad \tau\to -\infty.
\end{align*} 
Hence for $\2{\tau}_0$ sufficiently small and $M$  sufficiently large, 
\begin{align}\label{z(tau)-value}
\2{v}_{\lambda', h'}(z(\tau),f(\tau))=&1-C_{\lambda'}e^{-\gamma_{\lambda'}(-z(\tau)-\lambda' f(\tau)+h')}+o\left(e^{-\gamma_{\lambda'}(-z(\tau)-\lambda'f(\tau)+h')}\right)\notag\\
=&1-C_{\lambda'}e^{-\gamma_{\lambda'}(-z(\tau)-\lambda'\tau +h')}+o\left(e^{-\gamma_{\lambda'}(-z(\tau)-\lambda'\tau +h')}\right)\notag\\
=&1-\frac{pk}{p-1}e^{\frac{p-1}{p}\tau}+o\left(e^{\frac{p-1}{p}\tau}\right)=\xi_k(\tau).
\end{align}
holds for any $\tau\le\2{\tau}_0$.
By \eqref{y(tau)-value} and \eqref{z(tau)-value}, \eqref{y-tau-z-tau-infinity} follows.
\end{proof}

\begin{lem}\label{elliptic-supersoln-lem}
Let $\lambda>1$, $\lambda'>1$,  $h\ge h_0$, $h'\ge h_0'$,  and let $\2{\tau}_0$ be as in Lemma \ref{critical-pt-x-y-z-lem}. Then for any $0<k\le k_0$ and $a>-\2{\tau}_0$, the function $\2{f}_{a}(x):=\2{f}_{\lambda,\lambda',h,h',k}(x,-a)$ satisfies
\begin{equation}\label{ellliptic-eqn2}
(v^m)''+v-v^m\le 0\quad\mbox{ in distribution sense in }\R.
\end{equation}
\end{lem}
\begin{proof}
Since $v_{\lambda,h}$ satisfies \eqref{v-lambda-eqn}, by Lemma \ref{v-lambda-monotone-lem},
\begin{equation}\label{v-lambda-supersoln}
(v_{\lambda,h}^m)_{xx}(x,f(-a))+v_{\lambda,h}(x,f(-a))-v_{\lambda,h}(x,f(-a))^m=-\lambda v_{\lambda}'(x-\lambda f(-a)+h)<0\quad\mbox{ in }\R.
\end{equation}
Similarly
\begin{equation}\label{v-bar-lambda-supersoln}
(\2{v}_{\lambda',h'}^m)_{xx}(x,f(-a))+\2{v}_{\lambda',h'}(x,f(-a))-\2{v}_{\lambda',h'}(x,f(-a))^m=-\lambda v_{\lambda'}'(-x-\lambda'f(-a)+h')<0\quad\mbox{ in }\R.
\end{equation}
Since $0<\xi_k(\tau)<1$ for all $\tau<\tau_0$,
\begin{equation}\label{xi-k-supsoln}
\xi_{k,xx}+\xi_k-\xi_k^m=\xi_k-\xi_k^m<0\quad\mbox{ in }\R\quad\forall \tau<\tau_0.
\end{equation}
By \eqref{f-bar-value}, \eqref{v-lambda-supersoln}, \eqref{v-bar-lambda-supersoln}  and \eqref{xi-k-supsoln},
\begin{equation}\label{f-bar-supersoln}
(\2{f}_{a}^m)_{xx}(x)+\2{f}_{a}(x)-\2{f}_{a}(x)^m<0\quad\forall x\in\R\setminus
\{y(-a),z(-a)\}.
\end{equation}
Let $0\le\eta\in C_0^2(\R)$. For any $0<\3<(z(-a)-y(-a))/4$, let $\psi_{\3}\in C_0^2(\R)$, $0\le\psi_{\3}\le 1$, be such that $\psi_{\3}(x)=1$ for any $x\in (y(-a)-\frac{\3}{2},y(-a)+\frac{\3}{2})\cup (z(-a)-\frac{\3}{2},z(-a)+\frac{\3}{2})$, $\psi_{\3}(x)=0$ for any $|x-y(-a)|\ge\3$ and $|x-z(-a)|\ge\3$, and $|\psi_{\3}'(x)|\le C/\3$ on $\R$ for some constant $C>0$. Then we can write $\eta=\eta_{1,\3}+\eta_{2,\3}$ where $\eta_{1,\3}=\eta\psi_{\3}$ and $\eta_{2,\3}=\eta(1-\psi_{\3})$. By \eqref{f-bar-value} and \eqref{f-bar-supersoln},
\begin{align}\label{f-bar-supersoln-ineqn}
&\int_{\R}\left[\2{f}_{a}^m\eta''+(\2{f}_{a}-\2{f}_{a}^m)\eta\right]\,dx\notag\\
=&\int_{\R}\left[\2{f}_{a}^m\eta_{1,\3}''+(\2{f}_{a}-\2{f}_{a}^m)\eta_{1,\3}\right]\,dx
+\int_{\R}\left[\2{f}_{a}^m\eta_{2,\3}''+(\2{f}_{a}-\2{f}_{a}^m)\eta_{2,\3}\right]\,dx\notag\\
=&-\int_{\R}(\2{f}_{a}^m)_x\eta\psi_{\3}'\,dx-\int_{\R}(\2{f}_{a}^m)_x\eta'\psi_{\3}\,dx+\int_{\R}(\2{f}_{a}-\2{f}_{a}^m)\eta_{1,\3}\,dx
+\int_{\R}\left[(\2{f}_{a}^m)_{xx}+\2{f}_{a}-\2{f}_{a}^m\right]\eta_{2,\3}\,dx\notag\\
\le &-\int_{y(-a)-\3}^{y(-a)}(\2{f}_{a}^m)_x\eta\psi_{\3}'\,dx-\int_{z(-a)}^{z(-a)+\3}(\2{f}_{a}^m)_x\eta\psi_{\3}'\,dx-\int_{\R}(\2{f}_{a}^m)_x\eta'\psi_{\3}\,dx+\int_{\R}(\2{f}_{a}-\2{f}_{a}^m)\eta_{1,\3}\,dx\notag\\
=&:I_1+I_2+I_3+I_4.
\end{align}
Now 
\begin{equation}\label{I3-4-estimates}
|I_3|+|I_4|\le C\3\to 0\quad\mbox{ as }\3\to 0,
\end{equation}
\begin{equation}\label{I1-estimates}
I_1=-\int_{y(-a)-\3}^{y(-a)}(v_{\lambda}^m)'(x-\lambda f( -a)+h)\eta(x)\psi_{\3}'(x)\,dx\to -(v_{\lambda}^m)_x(y(-a)-\lambda f( - a)+h)\eta(y(-a))
\end{equation}
as $\3\to 0$ and
\begin{equation}\label{I2-estimates}
I_2=\int_{z(-a)}^{z(-a)+\3}(v_{\lambda'}^m)_x(-x-\lambda'f(-a)+h')\eta\psi_{\3}'\,dx\to -(v_{\lambda'}^m)_x(-z(-a)-\lambda'f(-a)+h')\eta(z(-a))
\end{equation}
as $\3\to 0$. 
Letting $\3\to 0$ in \eqref{f-bar-supersoln-ineqn}, by \eqref{I3-4-estimates}, \eqref{I1-estimates} and \eqref{I2-estimates}, for any  $0\le \eta\in C_0^2(\R)$,
\begin{align*}
&\int_{\R}\left[\2{f}_{a}^m\eta''+(\2{f}_{a}-\2{f}_{a}^m)\eta\right]\,dx\\
=&-(v_{\lambda}^m)_x(y(-a)-\lambda f(-a)+h)\eta(y(-a))-(v_{\lambda'}^m)_x(-z(-a)-\lambda'f(-a)+h')\eta(z(-a))\le 0.
\end{align*}
Hence $\2{f}_{a}$ satisfies of \eqref{ellliptic-eqn2} and the lemma follows.
\end{proof}

By a similar argument we have the following lemma.

\begin{lem}\label{elliptic-supersoln-lem2}
Let $\lambda>1$,  $h\ge h_0$,  and let $\2{\tau}_0$ be as in Lemma \ref{critical-pt-x-y-z-lem}. Then for any $0<k\le k_0$ and $a>-\2{\tau}_0$, the function $\2{f}_{\lambda,h,k}(x,-a)$ satisfies \eqref{ellliptic-eqn2}.
\end{lem}

\begin{lem}\label{elliptic-supersoln-lem3}(cf. proof of Lemma 2.4 of \cite{DPKS1})
Let $\lambda>1$, $\lambda'>1$,  $h\ge h_0$, $h'\ge h_0'$ and let $\2{\tau}_0$ be as in Lemma \ref{critical-pt-x-y-z-lem}. Then for any  $a>-\2{\tau}_0$, the function $\2{f}_{\lambda,\lambda', h,h'}(x,-a)$ satisfies \eqref{ellliptic-eqn2}.
\end{lem}

\begin{lem}\label{local-existence-lem}
Let $\lambda>1$, $\lambda'>1$,  $h\ge h_0$, $h'\ge h_0'$, $0<k\le k_0$ and let $\2{\tau}_0<0$ be as in Lemma \ref{critical-pt-x-y-z-lem}. Let $R>0$, $a>-\2{\tau}_0$ and $v_0\in L^{\infty}(-R,R)$ be such that 
\begin{equation}\label{v0-lower-upper-bd0}
f_{\lambda,\lambda',h,h',k}(x,-a)\le v_0(x)\le\|v_0\|_{L^{\infty}(-R,R)}<1\quad\mbox{ a.e. }x\in (-R,R).
\end{equation} 
Then there exists a unique solution $v_R\in C^{2,1}([-R,R]\times (-a,\2{\tau}_0])\cap L^{\infty}((-R,R)\times (-a,\2{\tau}_0))$ of 
\begin{equation}\label{v-soln-bd-domain}
\left\{\begin{aligned}
&v_{\tau}=(v^m)_{xx}+v-v^m\qquad\mbox{ in }(-R,R)\times(-a,\2{\tau}_0)\\
&v(x,\tau)=f_{\lambda,\lambda',h,h',k}(x,\tau)\quad\mbox{ on }\{\pm R\}\times(-a,\2{\tau}_0)\\
&v(x,-a)=v_0(x)\qquad\qquad\mbox{ on }(-R,R)
\end{aligned}\right.
\end{equation}
which satisfies
\begin{equation}\label{v-n1-r-lower-upper-bd3}
f_{\lambda,\lambda',h,h',k}(x,\tau)\le v_R(x,\tau)\le \|v_0\|_{L^{\infty}(-R,R)}<1\quad\forall |x|\le R,-a<\tau<\2{\tau}_0.
\end{equation}
\end{lem}
\begin{proof}
By Lemma \ref{bded-domain-comparison-lem} the solution of \eqref{v-soln-bd-domain} is unique. Hence we only need to prove existence of solution of \eqref{v-soln-bd-domain}. We will use a modification of proof of Theorem 3.5 of \cite{Hu} to prove the existence of solution of \eqref{v-soln-bd-domain}. We divide the proof into two cases.

\noindent{\bf Case 1}: $v_0\in C^{\infty}([-R,R])$ and there exists $0<\delta<R$ such that $v_0(x)=f_{\lambda,\lambda',h,h',k}(x,-a)$ for all $R-\delta\le |x|\le R$.

\noindent Let 
\begin{equation*}
a_1=\min_{\substack{|x|\le R\\-a\le\tau\le\2{\tau}_0}}f_{\lambda,\lambda',h,h',k}^m(x,\tau)\quad\mbox{ and }\quad M_1=\|v_0\|_{L^{\infty}(-R,R)}.
\end{equation*} 
Then $0<a_1\le M_1^m<1$. Let $0<b(s)\in C^{\infty}(\R)$ be a monotone increasing function such that $b(s)=(a_1/3)^{1-p}$ for all $s\le a_1/3$,  $b(s)=(3M_1^m)^{1-p}$ for all $s\ge 3M_1^m$, and $b(s)=s^{1-p}$ for all $a_1/2\le s\le 2M_1^m$. By standard parabolic theory \cite{LSU} there exists a solution $u_R\in C^{2,1}([-R,R]\times [-a,\2{\tau}_0])$ of
\begin{equation*}
\left\{\begin{aligned}
&pu_{\tau}=b(u)u_{xx}+u-ub(u)\quad\,\,\mbox{ in }(-R,R)\times(-a,\2{\tau}_0)\\
&u(x,\tau)=f_{\lambda,\lambda',h,h',k}(x,\tau)^m\qquad\mbox{ on }\{\pm R\}\times(-a,\2{\tau}_0)\\
&u(x,-a)=v_0(x)^m\qquad\qquad\quad\mbox{ on }(-R,R).
\end{aligned}\right.
\end{equation*}
Since $u_R\in C^{2,1}([-R,R]\times [-a,\2{\tau}_0])$, there exists $\3\in (0,\2{\tau}_0+a)$ such that $a_1/2\le u_R(x,\tau)\le 2M_1^m$ for any $|x|\le R$, $-a\le\tau\le -a+\3$. Hence $b(u_R(x,\tau))=u_R(x,\tau)^{1-p}$ for any $|x|\le R$, $-a\le\tau\le -a+\3$. Thus $u_R$ satisfies 
\begin{equation}
pu_{\tau}=u^{1-p}u_{xx}+u-u^{2-p}\quad\mbox{ in }(-R,R)\times(-a,-a+\3).
\end{equation}
Let $v_R=u_R^p$. Then $v_R\in C^{2,1}([-R,R]\times [-a,-a+\3])$  satisfies 
\begin{equation}\label{v-R-Dirichlet-problem}
\left\{\begin{aligned}
&v_{\tau}=(v^m)_{xx}+v-v^m\qquad\mbox{ in }(-R,R)\times(-a,-a+\3)\\
&v(x,\tau)=f_{\lambda,\lambda',h,h',k}(x,\tau)\quad\mbox{ on }\{\pm R\}\times(-a,-a+\3)\\
&v(x,-a)=v_0(x)\qquad\qquad\mbox{ on }(-R,R).
\end{aligned}\right.
\end{equation}
Since $M_1$ is a supersolution of \eqref{v-R-Dirichlet-problem} and   $f_{\lambda,\lambda',h,h',k}$ is a subsolution \cite{DPKS2} of \eqref{v-R-Dirichlet-problem}, by Lemma \ref{bded-domain-comparison-lem},
\begin{align}
&f_{\lambda,\lambda',h,h',k}(x,\tau)\le v_R(x,\tau)\le M_1\quad\forall |x|\le R,-a\le\tau\le -a+\3
\label{u-n1-r-lower-upper-bd0}\\
\Rightarrow\quad&a_1\le u_R(x,\tau)\le M_1^m\quad\forall |x|\le R,-a\le\tau\le -a+\3\label{u-n1-r-lower-upper-bd}
\end{align} 
Let $(-a,T)$, $-a+\3\le T\le\2{\tau}_0$, be the maximal time interval such that $v_R\in C^{2,1}([-R,R]\times [-a,T])$ satisfies
\begin{equation}\label{v-n1-R-soln-eqn1}
\left\{\begin{aligned}
&v_{\tau}=(v^m)_{xx}+v-v^m\quad\mbox{ in }(-R,R)\times(-a,T)\\
&v(x,\tau)=f_{\lambda,\lambda',h,h',k}(x,\tau)\quad\mbox{ on }\{\pm R\}\times(-a,T)\\
&v(x,-a)=v_0(x)\qquad\quad\mbox{ on }(-R,R)
\end{aligned}\right.
\end{equation} 
and \eqref{u-n1-r-lower-upper-bd0} in $[-R,R]\times [-a,T]$. Suppose $T<\2{\tau}_0$. Since $u_R\in C(\2{B}_R\times [-a,T])$ satisfies \eqref{u-n1-r-lower-upper-bd} in $[-R,R]\times [-a,T]$, there exists a constant $T_1\in (T,\2{\tau}_0)$ such that $a_1/2\le u_R(x,\tau)\le 2M_1^m$ for any $|x|\le R$, $-a\le\tau\le T_1$. By repeating the above argument $v_R\in C^{2,1}([-R,R]\times [-a,T_1])$ satisfies \eqref{u-n1-r-lower-upper-bd0} and \eqref{v-n1-R-soln-eqn1} with $-a+\3$ and $T$ being replaced by $T_1$. This contradicts the choice of $T$. Hence $T=\2{\tau}_0$ and $v_R\in C^{2,1}([-R,R]\times [-a,T])$ satisfies \eqref{v-soln-bd-domain} and \eqref{u-n1-r-lower-upper-bd0} in $[-R,R]\times [-a,T]$ and \eqref{v-n1-r-lower-upper-bd3} follows.

\noindent{\bf Case 2}: $v_0\in L^{\infty}(-R,R)$

\noindent By \eqref{v0-lower-upper-bd0} we can choose a sequence of functions $\{v_{0,i}\}_{i=1}^{\infty}\subset C^{\infty}([-R,R])$ satisfying
\begin{equation}\label{v0i-lower-upper-bd0}
f_{\lambda,\lambda',h,h',k}(x,-a)\le v_0(x)\le v_{0,i+1}(x)\le v_{0,i}(x)\le \|v_{0,i}\|_{L^{\infty}(-R,R)}<1\quad\forall  x\in (-R,R), i\in\mathbb{N}
\end{equation} 
and
\begin{equation*}
v_{0,i}(x)=f_{\lambda,\lambda',h,h',k}(x,-a)\quad\forall iR/(i+1)\le |x|\le R, i\in\mathbb{N}
\end{equation*}
with $v_{0,i}$ converges to  $v_0$  in $L^1(-R,R)$ and $\|v_{0,i}\|_{L^{\infty}}\to \|v_0\|_{L^{\infty}}$ as $i\to\infty$. For each $i\in\mathbb{N}$, by case 1 and Lemma \ref{bded-domain-comparison-lem} there exists a solution $v_{R,i}\in C^{2,1}([-R,R]\times [-a,\2{\tau}_0])$ of \eqref{v-soln-bd-domain} with $v_0$ being replaced $v_{0,i}$ which satisfies 
\begin{equation}\label{v_Ri-lower-upper-bd}
f_{\lambda,\lambda',h,h',k}(x,\tau)\le v_{R,i+1}(x,\tau)\le v_{R,i}(x,\tau)\le \|v_{0,i}\|_{L^{\infty}(-R,R)}<1\quad\forall |x|\le R,-a\le\tau<\2{\tau}_0, i\in\mathbb{N}.
\end{equation}
By \eqref{v_Ri-lower-upper-bd} the equation \eqref{yamabe-ode} for $\{v_{R,i}\}_{i=1}^{\infty}$ is uniformly parabolic on $[-R,R]\times [-a,\2{\tau}_0]$. Hence by the standard Schauder estimates \cite{LSU} the sequence $\{v_{R,i}\}_{i=1}^{\infty}$ is equi-continuous in $C^{2,1}(K)$ for any compact set $K\subset [-R,R]\times (-a,\2{\tau}_0]$. Thus the sequence $\{v_{R,i}\}_{i=1}^{\infty}$ will decrease uniformly in $C^{2,1}(K)$
to some function $v_R\in C^{2,1}([-R,R]\times (-a,\2{\tau}_0])$ as $i\to\infty$ for any compact subset $K$ of $[-R,R]\times (-a,\2{\tau}_0]$. 
Putting $v=v_{R,i}$, $g_0=v_{0,i}$, $v_0=v_{0,i}$ in \eqref{soln-integral-defn} and letting $i\to\infty$, we get that $v_R$ satisfies
\eqref{soln-integral-defn} with $g_0=v_0$. Hence $v_R$ is a solution of \eqref{v-soln-bd-domain}. Letting $i\to\infty$ in
\eqref{v_Ri-lower-upper-bd}, we get \eqref{v-n1-r-lower-upper-bd3}  and the lemma follows.  
\end{proof}

By a similar argument we have the following two lemmas.

\begin{lem}\label{local-existence-lem2}
Let $\lambda>1$, $\lambda'>1$,  $h\ge h_0$, $h'\ge h_0'$, $R>0$,  and let $\2{\tau}_0<0$ be as in Lemma \ref{critical-pt-x-y-z-lem}. Let $a>-\2{\tau}_0$ and $v_0\in L^{\infty}(-R,R)$ be such that 
\begin{equation*}
f_{\lambda,\lambda',h,h'}(x,-a)\le v_0(x)\le\|v_0\|_{L^{\infty}(-R,R)}<1\quad\mbox{ a.e. }x\in (-R,R).
\end{equation*} 
Then  there exists a unique solution $v_R\in  C^{2,1}([-R,R]\times (-a,\2{\tau}_0])\cap L^{\infty}((-R,R)\times (-a,\2{\tau}_0))$ of 
\begin{equation*}
\left\{\begin{aligned}
&v_{\tau}=(v^m)_{xx}+v-v^m\qquad\mbox{ in }(-R,R)\times(-a,\2{\tau}_0)\\
&v(x,\tau)=f_{\lambda,\lambda',h,h'}(x,\tau)\quad\,\,\mbox{ on }\{\pm R\}\times(-a,\2{\tau}_0)\\
&v(x,-a)=v_0(x)\qquad\qquad\mbox{ on }(-R,R)
\end{aligned}\right.
\end{equation*} 
which satisfies
\begin{equation*}
f_{\lambda,\lambda',h,h'}(x,\tau)\le v_R(x,\tau)\le \|v_0\|_{L^{\infty}(-R,R)}<1\quad\forall |x|\le R,-a<\tau<\2{\tau}_0.
\end{equation*}
\end{lem}

\begin{lem}\label{local-existence-lem3}
Let $\lambda>1$,  $h\ge h_0$, $0<k\le k_0$, $R>0$,  and let $\2{\tau}_0<0$ be as in Lemma \ref{critical-pt-x-y-z-lem}. Let $a>-\2{\tau}_0$ and $v_0\in L^{\infty}(-R,R)$  be such that 
\begin{equation*}
f_{\lambda,h,k}(x,-a)\le v_0(x)\le\|v_0\|_{L^{\infty}(-R,R)}<1\quad\mbox{ a.e }x\in (-R,R).
\end{equation*} 
Then  there exists a unique solution $v_R\in C^{2,1}([-R,R]\times (-a,\2{\tau}_0])\cap L^{\infty}((-R,R)\times (-a,\2{\tau}_0))$ of 
\begin{equation*}
\left\{\begin{aligned}
&v_{\tau}=(v^m)_{xx}+v-v^m\qquad\mbox{ in }(-R,R)\times(-a,\2{\tau}_0)\\
&v(x,\tau)=f_{\lambda,h,k}(x,\tau)\qquad\,\,\mbox{ on }\{\pm R\}\times(-a,\2{\tau}_0)\\
&v(x,-a)=v_0(x)\qquad\qquad\mbox{ on }(-R,R)
\end{aligned}\right.
\end{equation*} 
which satisfies
\begin{equation*}
f_{\lambda,h,k}(x,\tau)\le v_R(x,\tau)\le \|v_0\|_{L^{\infty}(-R,R)}<1\quad\forall |x|\le R,-a<\tau<\2{\tau}_0.
\end{equation*}
\end{lem}

\begin{lem}\label{neumann-existence-lem}
Let $\lambda>1$, $\lambda'>1$,  $h\ge h_0$, $h'\ge h_0'$, $0<k\le k_0$. Let $\2{\tau}_0<0$ be as in Lemma \ref{critical-pt-x-y-z-lem} and $a>-\2{\tau}_0$. Then there exists a constant $R_0=R_0(a)>0$ such that for any $R\ge R_0$ and $v_0\in L^{\infty}(-R,R)$ satisfying \eqref{v0-lower-upper-bd0}, there exists a unique solution $v_R\in C^{2,1}([-R,R]\times (-a,\2{\tau}_0])\cap L^{\infty}(B_R\times (-a,\2{\tau}_0))$ of 
\begin{equation}\label{v-neumann-eqn}
\left\{\begin{aligned}
&v_{\tau}=(v^m)_{xx}+v-v^m\quad\mbox{ in }(-R,R)\times(-a,\2{\tau}_0)\\
&(v^m)_x=(f_{\lambda,\lambda',h,h',k}^m)_x\quad\,\,\mbox{ on }\{\pm R\}\times(-a,\2{\tau}_0)\\
&v(x,-a)=v_0(x)\qquad\quad\mbox{ on }(-R,R)
\end{aligned}\right.
\end{equation}
which satisfies \eqref{v-n1-r-lower-upper-bd3}. Moreover if there exists $x_0\in (-R,R)$ such that
\begin{equation}\label{v0-behaviour}
\left\{\begin{aligned}
&v_0(x)\mbox{ is monotone increasing on }[-R,x_0]\\
&v_0(x)\mbox{ is monotone decreasing on }[x_0,R],
\end{aligned}\right.
\end{equation}
then for any $-a<\tau<\2{\tau}_0$ there exists $x_R(\tau)\in (-R,R)$ such that 
\begin{equation}\label{vrx-sign0}
\left\{\begin{aligned}
&v_{R,x}(x,\tau)>0>v_{R,x}(y,\tau)\quad\forall -R\le x<x_R(\tau)<y<R, -a<\tau<\2{\tau}_0\\
&v_{R,x}(x_R(\tau),\tau)=0\quad\,\,\forall -a<\tau<\2{\tau}_0.
\end{aligned}\right.
\end{equation}
\end{lem}
\begin{proof}
By Lemma \ref{bded-domain-neumann-comparison-lem} the solution of \eqref{v-neumann-eqn} is unique. Hence we only need to prove existence of solution of \eqref{v-neumann-eqn}. Let $y(\tau)$ and $z(\tau)$ be as in Lemma \ref{critical-pt-x-y-z-lem}. Let 
\begin{equation*}\label{Ro>y-z-tau}
R_0=R_0(a)>\max\left(\lambda\2{\tau}_0-h,,\lambda'\2{\tau}_0-h', \max_{-a\le\tau\le\2{\tau}_0}(|y(\tau)|,|z(\tau)|)\right)
\end{equation*}
be a constant to be determined later and let $R\ge R_0$. Then 
\begin{equation}\label{v-lambda-h-v-lambda'-h'>1/2}
v_{\lambda,h}(R,\tau)\ge v_{\lambda}(0)=\frac{1}{2}\quad\mbox{ and }\quad \2{v}_{\lambda',h'}(-R,\tau)\ge v_{\lambda'}(0)=\frac{1}{2}\quad\forall \tau\le\2{\tau}_0.
\end{equation}
We divide the existence proof into two cases.

\noindent{\bf Case 1}: $v_0\in C^{\infty}([-R,R])$ and there exists a constant  $\delta\in (0,R)$ such that $v_0(x)=f_{\lambda,\lambda',h,h',k}(x,-a)$ for any $R-\delta\le |x|\le R$.

By an argument similar to the proof of Lemma \ref{local-existence-lem} there exists $\3\in (0,\2{\tau}_0)$ such that there exists a unique solution $v_R\in C^{2,1}([-R,R]\times [-a,-a+\3])$ of 
\begin{equation}\label{v-neumann-eqn2}
\left\{\begin{aligned}
&v_{\tau}=(v^m)_{xx}+v-v^m\quad\mbox{ in }(-R,R)\times(-a,-a+\3)\\
&(v^m)_x=(f_{\lambda,\lambda',h,h',k}^m)_x\quad\,\,\mbox{ on }\{\pm R\}\times(-a,-a+\3)\\
&v(x,-a)=v_0(x)\qquad\quad\mbox{ on }(-R,R).
\end{aligned}\right.
\end{equation}
By \eqref{v-lambda-x=-infty-behavior}, \eqref{v-lambda-x-infty-behavior}, \eqref{v-lambda-h-value-at-infty}, \eqref{v-lambda--infty-behaviour} and \eqref{v-lambda-h-v-lambda'-h'>1/2},
\begin{align}\label{f-lambda-lambda'-h-h'-k-bdary-derivative1}
&\frac{\1}{\1 x}f_{\lambda,\lambda',h,h',k}(R,\tau)\notag\\
=&\left[v_{\lambda,h}(R,f(\tau))^{-p}v_{\lambda}'(R-\lambda f(\tau)+h)-\2{v}_{\lambda',h'}(R,f(\tau))^{-p}v_{\lambda'}'(-R-\lambda'f(\tau)+h')\right]f_{\lambda,\lambda',h,h',k}(R,\tau)^p\notag\\
\le&\left[2^p\left(C_{\lambda}\gamma_{\lambda}e^{-\gamma_{\lambda}(R-\lambda f(\tau)+h)}+o(e^{-\gamma_{\lambda}(R-\lambda f(\tau)+h)})\right)-Ce^{({p^2}-p)(R+\lambda'f(\tau)-h')}\right]f_{\lambda,\lambda',h,h',k}(R,\tau)^p\notag\\
<&0\quad\forall -a\le\tau\le\2{\tau}_0
\end{align}
if $R_0$ is sufficiently large where $f$ is given by \eqref{f-defn} and
\begin{align}\label{f-lambda-lambda'-h-h'-k-bdary-derivative2}
&\frac{\1}{\1 x}f_{\lambda,\lambda',h,h',k}(-R,\tau)\notag\\
=&\left[v_{\lambda,h}(-R,f(\tau))^{-p}v_{\lambda}'(-R-\lambda f(\tau)+h)-\2{v}_{\lambda',h'}(-R,f(\tau))^{-p}v_{\lambda'}'(R-\lambda'f(\tau)+h')\right]f_{\lambda,\lambda',h,h',k}(-R,\tau)^p\notag\\
\ge&\left[Ce^{(p^2-p)(R+\lambda f(\tau)-h)}-2^p\left(C_{\lambda'}\gamma_{\lambda'}e^{-\gamma_{\lambda'}(R-\lambda'f(\tau)+h')}+o(e^{-\gamma_{\lambda'}(R-\lambda'f(\tau)+h')})\right)\right]f_{\lambda,\lambda',h,h',k}(-R,\tau)^p\notag\\
>&0\quad\forall -a\le\tau\le\2{\tau}_0
\end{align}
if $R_0$ is sufficiently large. We now choose $R_0$ sufficiently large such that both \eqref{f-lambda-lambda'-h-h'-k-bdary-derivative1} and \eqref{f-lambda-lambda'-h-h'-k-bdary-derivative2} hold.
Then by \eqref{f-lambda-lambda'-h-h'-k-bdary-derivative1} and \eqref{f-lambda-lambda'-h-h'-k-bdary-derivative2},  $M:=\|v_0\|_{L^{\infty}(-R,R)}<1$ is a supersolution of \eqref{v-neumann-eqn2}. On the other hand $f_{\lambda,\lambda',h,h',k}$ is a subsolution of \eqref{v-neumann-eqn2}.  Hence 
by  Lemma \ref{bded-domain-neumann-comparison-lem}, $v_R$ satisfies \eqref{v-n1-r-lower-upper-bd3} in $\2{B}_R\times [-a,-a+\3]$.  

Let $(-a,T)$, $-a+\3\le T\le\2{\tau}_0$, be the maximal time interval of existence of solution $v_R\in C^{2,1}([-R,R]\times [-a,T])$ of 
\begin{equation}\label{v-neumann-eqn3}
\left\{\begin{aligned}
&v_{\tau}=(v^m)_{xx}+v-v^m\quad\mbox{ in }(-R,R)\times(-a,T)\\
&(v^m)_x=(f_{\lambda,\lambda',h,h',k}^m)_x\quad\,\,\mbox{ on }\{\pm R\}\times(-a,T)\\
&v(x,-a)=v_0(x)\qquad\quad\mbox{ on }(-R,R)
\end{aligned}\right.
\end{equation}
which satisfies \eqref{v-n1-r-lower-upper-bd3} in $[-R,R]\times [-a,T]$. Suppose $T<\2{\tau}_0$. Then by \eqref{v-n1-r-lower-upper-bd3} and an argument similar to the proof of Lemma \ref{local-existence-lem} there exists a constant $T_1\in (T,\2{\tau}_0)$ such that $v_R$ can be extended to a solution  of \eqref{v-neumann-eqn3} in $(-R,R)\times (-a,T_1)$ that satisfies \eqref{v-n1-r-lower-upper-bd3} in $[-R,R]\times [-a,T_1]$ and $v_R\in C^{2,1}([-R,R]\times [-a,T_1])$. This contradicts the choice of $T$. Hence $T=\2{\tau}_0$.

\noindent{\bf Case 2}: $v_0\in L^{\infty}(-R,R)$ 

\noindent We choose a sequence of function $\{v_{0,i}\}_{i=1}^{\infty}\subset C^{\infty}([-R,R])$  
satisfying \eqref{v0i-lower-upper-bd0} and
\begin{equation}\label{v0i-bdary-behaviour}
v_{0,i}(x)=f_{\lambda,\lambda',h,h',k}(x,-a)\quad\forall (1-2^{-1-i})R\le |x|\le R, i\in\mathbb{N}
\end{equation}
and $v_{0,i}$ converges to  $v_0$  in $L^1(-R,R)$ and $\|v_{0,i}\|_{L^{\infty}(-R,R)}\to \|v_0\|_{L^{\infty}(-R,R)}$ as $i\to\infty$.
By case 1 for any $i\in\mathbb{N}$ there exists a solution $v_{R,i}\in C^{2,1}([-R,R]\times [-a,\2{\tau}_0])$ of \eqref{v-neumann-eqn} with $v_0$ being replaced by $v_{0,i}$ and $v_{R,i}$  satisfies \eqref{v_Ri-lower-upper-bd}. Then by \eqref{v_Ri-lower-upper-bd} the equation
\eqref{yamabe-ode} for the sequence $\{v_{R,i}\}_{i=1}^{\infty}$ is uniformly parabolic on $[-R,R]\times [-a,\2{\tau}_0]$. Hence by the Schauder's estimates \cite{LSU} the sequence $\{v_{R,i}\}_{i=1}^{\infty}$ is equi-Holder continuous in $C^{2,1}(K)$ for any compact subset
$K$ of $[-R,R]\times (-a,\2{\tau}_0]$. Thus by the Ascoli Theorem and a diagonalization argument the sequence $\{v_{R,i}\}_{i=1}^{\infty}$ has a subsequence $\{v_{R,i_k}\}_{k=1}^{\infty}$ that converges in $C^{2,1}(K)$ to some function $v_R\in C^{2,1}([-R,R]\times (-a,\2{\tau}_0])$ as $k\to\infty$ for any compact subset $K$ of $[-R,R]\times (-a,\2{\tau}_0]$. 
Since $v_{R,i}\in C^{2,1}([-R,R]\times (-a,\2{\tau}_0])$ satisfies \eqref{v-neumann-eqn},
\begin{align}\label{vri-integral-eqn}
\int_{-R}^Rv_{R,i}(x,\tau)\eta (x,\tau)\,dx=&\int_{-a}^{\tau}\int_{-R}^R\left[v_{R,i}\eta_{\tau}+v_{R,i}^m\eta_{xx}+(v_{R,i}-v_{R,i}^m)\eta\right]\,dx\,d\tau\notag\\
&\qquad +\int_{-a}^{\tau}\int_{\1 B_R}\eta\frac{\1}{\1 n}f_{\lambda,\lambda',h,h,k}^m\,d\sigma\,d\tau
 +\int_{-R}^Rv_{0,i}(x)\eta(x,\tau_1)\,dx
\end{align}
holds for any $-a<\tau<\2{\tau}_0$, $i\in\mathbb{N}$, and function $0\le\eta\in C^{2,1}([-R,R]\times [-a,\tau])$ satisfying $\eta_x\equiv 0$ on $\{\pm R\}\times [-a,\tau]$. Putting $i=i_k$ in \eqref{vri-integral-eqn} and letting $k\to\infty$,
\begin{align*}
\int_{-R}^Rv_R(x,\tau)\eta (x,\tau)\,dx=&\int_{-a}^{\tau}\int_{-R}^R\left[v_R\eta_{\tau}+v_R^m\eta_{xx}+(v_R-v_R^m)\eta\right]\,dx\,d\tau\notag\\
&\qquad +\int_{-a}^{\tau}\int_{\1 B_R}\eta\frac{\1}{\1 n}f_{\lambda,\lambda',h,h,k}^m\,d\sigma\,d\tau
+\int_{-R}^Rv_0(x)\eta(x,\tau_1)\,dx
\end{align*}
holds for any $-a<\tau<\2{\tau}_0$, $i\in\mathbb{N}$, and function $0\le\eta\in C^{2,1}([-R,R]\times [-a,\tau])$ satisfying $\eta_x\equiv 0$ on $\{\pm R\}\times [-a,\tau]$.  Letting $i=i_k\to\infty$ in \eqref{v_Ri-lower-upper-bd}, we get
\eqref{v-n1-r-lower-upper-bd3}. Hence by Lemma \ref{bded-domain-neumann-comparison-lem} $v_R\in C^{2,1}([-R,R]\times (-a,\2{\tau}_0])\cap L^{\infty}((-R,R)\times (-a,\2{\tau}_0))$ is the unique solution of \eqref{v-neumann-eqn}.  Thus
$v_i$ converges in $C^{2,1}(K)$ to $v_R$ as $i\to\infty$ for any compact subset $K$ of $[-R,R]\times (-a,\2{\tau}_0]$.  

Suppose now $v_0$ satisfies \eqref{v0-behaviour} for some $x_0\in (-R,R)$.  We claim that for any $-a<\tau<\2{\tau}_0$ there exists $x_R(\tau)\in (-R,R)$ such that \eqref{vrx-sign0} holds. We divide the proof of the claim into two cases.

\noindent{\bf Case A}: $v_0\in C^{\infty}([-R,R])$ and there exists a constant  $\delta\in (0,R)$ such that $v_0(x)=f_{\lambda,\lambda',h,h',k}(x,-a)$ for any $R-\delta\le |x|\le R$ and
\begin{equation}\label{v0-behaviour2}
\left\{\begin{aligned}
&v_0'(x)>0>v_0'(y)\quad\forall -R\le x\le x_0-\3_1,x_0+\3_1\le x\le R\\
&v_0''(x)<0\qquad\qquad\forall x\in [x_0-\3_1,x_0+\3_1]
\end{aligned}\right.
\end{equation}
for some constants $x_0\in (-R,R)$ and $\3_1\in \left(0,\frac{1}{2}\min (R-x_0,R+x_0)\right)$. Let $v_R\in C^{2,1}(\2{B}_R\times [-a,\2{\tau}_0])$ be the solution of \eqref{v-neumann-eqn} given by case 1 above. By \eqref{v-neumann-eqn}, $v_{R,x}$ satisfies
\begin{equation}\label{v-x-eqn0}
w_{\tau}=mv^{m-1}w_{xx}-3m(1-m)v^{m-2}ww_x+m(1-m)(2-m)v^{m-3}w^3+w-mv^{m-1}w
\end{equation}
in $(-R,R)\times (-a,\2{\tau}_0)$. Since $v_R\in C^{2,1}([-R,R]\times [-a,\2{\tau}_0])$, by \eqref{v0-behaviour2}  there exists $T_2\in (-a,\2{\tau}_0)$ such that
\begin{equation}\label{vrx-sign}
\left\{\begin{aligned}
&v_{R,x}(x,\tau)>0>v_{R,x}(y,\tau)\quad\forall -R\le x\le x_0-\3_1, x_0+\3_1\le y\le R, -a\le\tau\le T_2\\
&v_{R,xx}(x,\tau)<0\quad\forall x_0-\3_1\le x\le x_0+\3_1, -a\le\tau\le T_2.\end{aligned}\right.
\end{equation}
By \eqref{vrx-sign} for any $-a\le\tau\le T_2$ there exists a unique $x_R(\tau)\in (x_0-\3_1,x_0+\3_1)$ such that
\begin{equation*}
\left\{\begin{aligned}
&v_{R,x}(x,\tau)>0>v_{R,x}(y,\tau)\quad\forall -R\le x<x_R(\tau)<y\le R, -a<\tau\le T_2\\
&v_{R,x}(x_R(\tau),\tau)=0\quad\forall -a\le\tau\le T_2.
\end{aligned}\right.
\end{equation*}
Let $(-a,T_2')$, $T_2\le T_2'\le\2{\tau}_0$, be the maximal interval such that for any $-a<\tau<T_2'$ there exists a unique $x_R(\tau)\in (-R,R)$ such that
\begin{equation}\label{vrx-sign2}
\left\{\begin{aligned}
&v_{R,x}(x,\tau)>0>v_{R,x}(y,\tau)\quad\forall -R\le x<x_R(\tau)<y\le R, -a<\tau< T_2'\\
&v_{R,x}(x_R(\tau),\tau)=0\quad\forall -a<\tau< T_2'.
\end{aligned}\right.
\end{equation}
Suppose $T_2'<\2{\tau}_0$. Then by compactness there exists a sequence $-a<\tau_i<T_2'$, $\tau_i\to T_2'$ as $i\to\infty$, such that $x_R(\tau_i)$ converges to some point $x_R(T_2')\in [-R,R]$ as $i\to\infty$. Since
\begin{equation}\label{vrx=0}
v_{R,x}(x_R(T_2'),T_2')=\lim_{i\to\infty}v_{R,x}(x_R(\tau_i),\tau_i)=0,
\end{equation}
by \eqref{v-neumann-eqn}, \eqref{f-lambda-lambda'-h-h'-k-bdary-derivative1}  and \eqref{f-lambda-lambda'-h-h'-k-bdary-derivative2}, $x_R(T_2')\in (-R+\3_2,R-\3_2)$ for some constant $\3_2\in (0,R/2)$. Let 
\begin{equation*}
D=\left\{(x,\tau):-R< x<x_R(\tau), -a<\tau\le T_2'\right\}
\end{equation*}
By \eqref{vrx-sign2},
\begin{equation}\label{vrx-sign3}
v_{R,x}(x,\tau)\ge 0\quad\forall (x,\tau)\in\2{D}.
\end{equation}
By \eqref{v-n1-r-lower-upper-bd3} the equation \eqref{v-x-eqn0} for $v_{R,x}$ is uniformly parabolic on $[-R,R]\times [-a,\2{\tau}_0]$. Hence by
\eqref{v-neumann-eqn}, \eqref{f-lambda-lambda'-h-h'-k-bdary-derivative2}, \eqref{v-x-eqn0}, \eqref{vrx-sign3} and the strong maximum principle,
\begin{equation}\label{vrx>0}
v_{R,x}(x,\tau)>0\quad\forall -R\le x<x_R(\tau), -a<\tau\le T_2'.
\end{equation}
Similarly,
\begin{equation}\label{vrx<0}
v_{R,x}(x,\tau)<0\quad\forall x_R(\tau)<x\le R, -a<\tau\le T_2'.
\end{equation}
We next observe that since the equation \eqref{v-x-eqn0} for $v_{R,x}$ is uniformly parabolic on $[-R,R]\times [-a,\2{\tau}_0]$,
the results of Lemma 2.4 of \cite{H1} (cf. \cite{A}, \cite{CP}, \cite{M}) remains valid for the solution $v_{R,x}$ of \eqref{v-x-eqn0}. Hence there exist constants $0<\delta_1<\min\left(\frac{R-x_R(T_2')}{2},\frac{R+x_R(T_2')}{2}\right)$ and $0<\delta_2< \2{\tau}_0-T_2'$ such that for any $T_2'\le\tau\le T_2'+\delta_2$ there exists $x_R(\tau)\in\left(x_R(T_2')-\delta_1,x_R(T_2')+\delta_1\right)$ such that
\begin{equation}\label{vrx-sign4}
\left\{\begin{aligned}
&v_{R,x}(x,\tau)>0>v_{R,x}(y,\tau)\quad\forall x_R(T_2')-\delta_1\le x<x_R(\tau)<y\le x_R(T_2')+\delta_1, T_2'\le\tau\le T_2'+\delta_2\\
&v_{R,x}(x_R(\tau),\tau)=0\qquad\qquad\forall T_2'\le\tau\le T_2'+\delta_2.
\end{aligned}\right.
\end{equation}
Since $v_R\in C^{2,1}([-R,R]\times [-a,\2{\tau}_0])$, by \eqref{vrx>0} and \eqref{vrx<0} there exists a constant $0<\delta_3<\delta_2$ such that
\begin{equation}\label{vrx><0}
v_{R,x}(x,\tau)>0>v_{R,x}(y,\tau)\quad\forall -R\le x\le x_R(T_2')-\delta_1, x_R(T_2')+\delta_1\le y\le R,T_2'\le\tau\le T_2'+\delta_3.
\end{equation}
By \eqref{vrx-sign4} and \eqref{vrx><0}, \eqref{vrx-sign2} holds with $T_2'$ replaced by $T_2'+\delta_3$. This contradicts the choice of $T_2'$. Hence $T_2'=\2{\tau}_0$. Hence for any $-a<\tau<\2{\tau}_0$ there exists $x_R(\tau)\in (-R,R)$ such that  \eqref{vrx-sign0} holds.

\noindent{\bf Case B}: $v_0\in L^{\infty}(-R,R)$ 

\noindent We choose a sequence of function $\{v_{0,i}\}_{i=1}^{\infty}\subset C^{\infty}([-R,R])$ 
satisfying \eqref{v0i-lower-upper-bd0}, \eqref{v0i-bdary-behaviour}, $v_{0,i}$ converges to  $v_0$  in $L^1(-R,R)$ and $\|v_{0,i}\|_{L^{\infty}(-R,R)}\to \|v_0\|_{L^{\infty}(-R,R)}$ as $i\to\infty$, and
\begin{equation}\label{v0i-behaviour}
\left\{\begin{aligned}
&v_{0,i}'(x)>0>v_{0,i}'(y)\quad\forall -R\le x\le x_{0,i}-\delta_0, x_{0,i}+\delta_0\le y\le R\\
&v_{0,i}''(x)<0\qquad\qquad\,\forall x\in [x_{0,i}-\delta_0,x_{0,i}+\delta_0]
\end{aligned}\right.
\end{equation}
holds for some sequences $\{x_{0,i}\}_{i=1}^{\infty}\subset (x_0-\delta_0,x_0+\delta_0)$ where $\delta_0=min\left(\frac{R-x_0}{4},\frac{R+x_0}{4}\right)$ such that $x_{0,i}\to x_0$ as $i\to\infty$. For any $i\in\mathbb{N}$ let $v_{R,i}\in C^{2,1}([-R,R]\times [-a,\2{\tau}_0])$ be the solution of \eqref{v-neumann-eqn} with $v_0$ being replaced by $v_{0,i}$ such that  $v_{R,i}$ satisfies \eqref{v-n1-r-lower-upper-bd3}. By case A, for any $i\in\mathbb{N}$ and $-a<\tau<\2{\tau}_0$ there exists $x_{R,i}(\tau)\in (-R,R)$ such that
\begin{equation}\label{virx-sign0}
\left\{\begin{aligned}
&v_{R,i,x}(x,\tau)>0>v_{R,i,x}(y,\tau)\quad\forall -R\le x<x_{R,i}(\tau)<y\le R, -a<\tau<\2{\tau}_0\\
&v_{R,i,x}(x_{R,i}(\tau),\tau)=0\quad\qquad\quad\forall -a<\tau<\2{\tau}_0.
\end{aligned}\right.
\end{equation}
By case 2, $v_i$ converges in $C^{2,1}(K)$ to the solution $v_R\in C^{2,1}([-R,R]\times (-a,\2{\tau}_0])\cap L^{\infty}((-R,R)\times (-a,\2{\tau}_0))$ of \eqref{v-neumann-eqn} as $i\to\infty$ for any compact subset $K$ of $[-R,R]\times (-a,\2{\tau}_0]$ and $v_R$ satisfies \eqref{v-n1-r-lower-upper-bd3}.  
For any $-a<\tau<\2{\tau}_0$, let 
\begin{equation}
D_0(\tau)=\left\{x\in [-R,R]:v_{R,x}(x,\tau)=0\right\}. 
\end{equation}
Since $[-R,R]$ is compact, for any $-a<\tau<\2{\tau}_0$ the sequence $\{x_{R,i}(\tau)\}_{i=1}^{\infty}$ has a convergence subsequence $\{x_{R,i_k}(\tau)\}_{k=1}^{\infty}$. Let $x_R(\tau)=\lim_{k\to\infty}x_{R,i_k}(\tau)$.  Then by \eqref{virx-sign0}, $v_{R,x}(x_R(\tau),\tau)=0$. Hence $D_0(\tau)\ne\phi$.

Since $v_R\in C^{2,1}([-R,R]\times (-a,\2{\tau}_0])$, by \eqref{v-neumann-eqn}, \eqref{f-lambda-lambda'-h-h'-k-bdary-derivative1} and\eqref{f-lambda-lambda'-h-h'-k-bdary-derivative2}, for any $-a<\tau<\2{\tau}_0$ there exists a constant $\3_{\tau}\in (0,R)$ such that
\begin{align}
&v_{R,x}(x,\tau)>0>v_{R,x}(y,\tau)\quad\forall -R\le x\le -R+\3_{\tau},R-\3_{\tau}\le y\le R\label{vrx-sign6}\\
\Rightarrow\quad&D_0(\tau)\subset [-R+\3_{\tau},R-\3_{\tau}].\label{vrx-sign5}
\end{align}
We claim that $D_0(\tau)$ is a singleton for any $-a<\tau<\2{\tau}_0$. Suppose not. Then there exists $\tau_1\in (-a,\2{\tau}_0)$ such that $D_0(\tau_1)$ is not a singleton. By the discussion on P.241 of \cite{SGKM} and \cite {CP} $v_{R,x}(x,\tau_1)$ is an analytic function of $x\in (-R,R)$. Hence either all the zeros of $v_{R,x}(x,\tau_1)$ are isolated zeros or 
\begin{equation}\label{vrx-all=0}
v_{R,x}(x,\tau_1)=0\quad\forall |x|\le R.
\end{equation}  
By \eqref{vrx-sign5}, \eqref{vrx-all=0} is not possible. Hence all the zeros of $v_{R,x}(x,\tau_1)$ are isolated zeros. Thus $D_0(\tau_1)$ has a finite number of elements and we can write $D_0(\tau_1)=\{x_1,...,x_k\}$ for some $k\ge 2$ such that $-R<x_1<x_2<\cdots<x_k<R$. Let $x_0=-R$ and $x_{k+1}=R$. By \eqref{vrx-sign6},
\begin{equation*}
v_{R,x}(x,\tau_1)>0>v_{R,x}(y,\tau_1)\quad\forall -R\le x<x_1,x_k<y\le R.
\end{equation*}
Hence there exists $k_0\in\{1,\dots,k\}$ such that 
\begin{equation*}
v_{R,x}(x,\tau_1)>0>v_{R,x}(y,\tau_1)\quad\forall x_{k_0-1}<x<x_{k_0}<y<x_{k_0+1}.
\end{equation*}
Without loss of generality we may assume that 
\begin{equation}\label{vrx-sign7}
v_{R,x}(x,\tau_1)>0>v_{R,x}(y,\tau_1)\quad\forall -R<x<x_1<y<x_2.
\end{equation}
Let $\delta_1=\min\left(\frac{x_1+R}{4},\frac{x_2-x_1}{4}\right)$. Since $v_{R,i,x}(\cdot,\tau_1)$ converges uniformly to $v_R(\cdot,\tau_1)$ on $[-R,R]$ as $i\to\infty$, by \eqref{vrx-sign7} there exists $i_0\in\mathbb{N}$ such that
\begin{equation}\label{vrx-sign8}
v_{R,i,x}(x_1-\delta_1,\tau_1)>0>v_{R,i,x}(x_1+\delta_1,\tau_1)\quad\forall i\ge i_0.
\end{equation}
By \eqref{vrx-sign8} and the intermediate value theorem for any $i\ge i_0$, there exist $z_i\in (x_1-\delta_1,x_1+\delta_1)$ such that 
\begin{equation}\label{zi-eqn}
v_{R,i,x}(z_i,\tau_1)=0\quad\forall i\ge i_0.
\end{equation}
By compactness the sequence $\{z_i\}$ has a convergence subsequence which we may assume without loss of generality to be the sequence itself that converges to some point $z_0\in [x_1-\delta_1,x_1+\delta_1]$. Letting $i\to\infty$ in \eqref{zi-eqn},    
\begin{equation}\label{z0-eqn}
v_{R,x}(z_0,\tau_1)=0.
\end{equation}
By \eqref{vrx-sign7} and \eqref{z0-eqn},  $z_0=x_1$. By \eqref{virx-sign0} and \eqref{zi-eqn}, $z_i=x_{R,i}(\tau_1)$ for all $i\ge i_0$. Putting $\tau=\tau_1$ in \eqref{virx-sign0} and letting $i\to\infty$,
\begin{equation}\label{vrx-sign23}
v_{R,x}(x,\tau_1)\ge 0\ge v_{R,x}(y,\tau_1)\quad\forall -R\le x\le x_1\le y\le R. 
\end{equation}
Since all the zeros of $v_{R,x}(x,\tau_1)$ are isolated zeros, by \eqref{vrx-sign9}, 
\begin{equation}\label{vrx-sign9}
v_{R,x}(x,\tau_1)<0\quad\forall x_1<x<x_3, x\ne x_2. 
\end{equation}
Let $\delta_2=\min\left(\frac{x_2-x_1}{4},\frac{x_3-x_2}{4}\right)$. Then by \eqref{vrx-sign9} there exists $\delta_3\in (0,\tau_1+a)$ such that
\begin{equation}\label{vrx-sign10}
v_{R,x}(x_2-\delta_2,\tau)<0\quad\mbox{ and }\quad v_{R,x}(x_2+\delta_1,\tau)<0\quad\forall \tau_1-\delta_3\le\tau\le\tau_1.
\end{equation}
By \eqref{virx-sign0} and an argument similar to the above one there exists a constant $\2{x}_1\in (-R,R)$ such that
\begin{equation}\label{vrx-sign11}
v_{R,x}(x,\tau_1-\delta_3)\ge 0\ge v_{R,x}(y,\tau_1-\delta_3)\quad\forall -R\le x\le \2{x}_1\le y\le R 
\end{equation} 
By \eqref{vrx-sign10} and \eqref{vrx-sign11}, $\2{x}_1\le x_2-\delta_2$ and
\begin{equation}\label{vrx-sign12}
v_{R,x}(x,\tau_1-\delta_3)\le 0\quad\forall x_2-\delta_2\le x\le R.
\end{equation}
By \eqref{v-n1-r-lower-upper-bd3} the equation \eqref{v-x-eqn0} for $v_{R,x}$ is uniformly parabolic on $[-R,R]\times [-a,\2{\tau}_0]$. 
Hence by \eqref{v-x-eqn0}, \eqref{vrx-sign10}, \eqref{vrx-sign12} and the strong maximum principle in $(x_2-\delta_2,x_2+\delta_2)\times (\tau_1-\delta_3,\tau_1)$,
\begin{equation*}
v_{R,x}(x_2,\tau_1)<0
\end{equation*}
and contradiction arises. Hence no such $\tau_1$ exists and $D_0(\tau)$ is a singleton for any $-a<\tau<\2{\tau}_0$.
Hence $D_0(\tau)=\{x_R(\tau)\}$ and $x_R(\tau)=\lim_{i\to\infty}x_{R,i}(\tau)$. Letting $i\to\infty$ in \eqref{virx-sign0},
\begin{equation}\label{vrx-sign13}
v_{R,x}(x,\tau)\ge 0\ge v_{R,x}(y,\tau)\quad\forall -R\le x\le x_R(\tau)\le y\le R, -a<\tau<\2{\tau}_0
\end{equation}
Since for any $-a<\tau<\2{\tau}_0$, $x_R(\tau)$ is an isolated zero of $v_R(\cdot,\tau)$, by \eqref{vrx-sign13} we get \eqref{vrx-sign0}
and the lemma follows.
\end{proof}

By a similar argument we have the two following lemmas.

\begin{lem}\label{neumann-existence-lem2}
Let $\lambda>1$, $\lambda'>1$,  $h\ge h_0$, $h'\ge h_0'$ and let $\2{\tau}_0<0$ be as in Lemma \ref{critical-pt-x-y-z-lem}. Let $a>-\2{\tau}_0$ and $v_0\in L^{\infty}(-R,R)$ be such that 
\begin{equation*}
f_{\lambda,\lambda',h,h'}(x,-a)\le v_0(x)\le\|v_0\|_{L^{\infty}(-R,R)}<1\quad\mbox{ a.e. }x\in (-R,R)
\end{equation*} 
holds. Then there exists a constant $R_0=R_0(a)>0$ such that for any $R\ge R_0$ there exists a unique solution $v_R\in C^{2,1}([-R,R]\times (-a,\tau_0])\cap L^{\infty}((-R,R)\times (-a,\2{\tau}_0))$ of 
\begin{equation*}
\left\{\begin{aligned}
&v_{\tau}=(v^m)_{xx}+v-v^m\quad\mbox{ in }(-R,R)\times(-a,\2{\tau}_0)\\
&(v^m)_x=(f_{\lambda,\lambda',h,h'}^m)_x\qquad\mbox{ on }\{\pm R\}\times(-a,\2{\tau}_0)\\
&v(x,-a)=v_0(x)\qquad\quad\mbox{ on }(-R,R)
\end{aligned}\right.
\end{equation*}
which satisfies 
\begin{equation*}
f_{\lambda,\lambda',h,h'}(x,\tau)\le v_R(x,\tau)\le \|v_0\|_{L^{\infty}(-R,R)}<1\quad\forall |x|\le R,-a<\tau<\2{\tau}_0.
\end{equation*}
Moreover if there exists $x_0\in (-R,R)$ such that \eqref{v0-behaviour} holds,
then for any $-a<\tau<\2{\tau}_0$ there exists $x_R(\tau)\in (-R,R)$ such that \eqref{vrx-sign0} holds.
\end{lem}

\begin{lem}\label{neumann-existence-lem3}
Let $\lambda>1$,   $h\ge h_0$,  $0<k\le k_0$ and let $\2{\tau}_0<0$ be as in Lemma \ref{critical-pt-x-y-z-lem}. Let $a>-\2{\tau}_0$ and $v_0\in L^{\infty}(-R,R)$ be such that 
\begin{equation*}
f_{\lambda,h,k}(x,-a)\le v_0(x)\le\|v_0\|_{L^{\infty}(-R,R)}<1\quad\mbox{a.e. } x\in (-R,R)
\end{equation*} 
holds.
Then there exists a constant $R_0=R_0(a)>0$ such that for any $R\ge R_0$ there exists a unique solution $v_R\in C^{2,1}([-R,R]\times (-a,\tau_0])\cap L^{\infty}((-R,R)\times (-a,\2{\tau}_0))$ of 
\begin{equation*}
\left\{\begin{aligned}
&v_{\tau}=(v^m)_{xx}+v-v^m\quad\mbox{ in }(-R,R)\times(-a,\2{\tau}_0)\\
&(v^m)_x=(f_{\lambda,h,k}^m)_x\qquad\quad\mbox{ on }\{\pm R\}\times(-a,\2{\tau}_0)\\
&v(x,-a)=v_0(x)\qquad\quad\mbox{ on }(-R,R)
\end{aligned}\right.
\end{equation*}
which satisfies 
\begin{equation*}
f_{\lambda,h,k}(x,\tau)\le v_R(x,\tau)\le \|v_0\|_{L^{\infty}(-R,R)}<1\quad\forall |x|\le R,-a<\tau<\2{\tau}_0.
\end{equation*}
Moreover if $v_0$ is monotone increasing on $[-R,R]$, then 
\begin{equation*}
v_{R,x}(x,\tau)>0\quad\forall x\in [-R,R), -a<\tau<\2{\tau}_0.
\end{equation*}
\end{lem}

\section{Existence and properties of ancient solutions}
\setcounter{equation}{0}
\setcounter{thm}{0}

In this section we will prove the existence and various properties of the $5$-parameters, $4$-parameters, $3$-parameters ancient solutions of \eqref{yamabe-ode}. We will also prove  the uniqueness of the $4$-parameters ancient solution $v_{\lambda,\lambda',h,h'}$.

\begin{lem}\label{time-rescale-supersoln-lem}
Let $\lambda>1$, $\lambda'>1$,  $h\ge h_0$, $h'\ge h_0'$, $0<k\le k_0$ and let $\2{\tau}_0<0$ be as in Lemma \ref{critical-pt-x-y-z-lem}.
Then both $v_{\lambda,h}(x,f(\tau))$ and $\2{v}_{\lambda,h}(x,f(\tau))$ are  supersolutions of \eqref{yamabe-ode} in $\R\times (-\infty,\2{\tau}_0]$.
\end{lem}
\begin{proof}
Let $q(x,\rho)=v_{\lambda,h}(x,f(\rho))$ and $\tau=f(\rho)$ where $f$ is given by \eqref{f-defn}. Since $\2{\tau}_0<-p/(p-1)$ and $(v_{\lambda,h})_{\tau}(x,\tau)=-\lambda v_{\lambda}'(x-\lambda\tau+h)<0$ in $\R\times (-\infty,\2{\tau}_0]$,
\begin{align*}
q_{\rho}(x,\rho)=&(v_{\lambda,h})_{\tau}(x,f(\rho))f'(\rho)=(v_{\lambda,h})_{\tau}(x,f(\rho))\left(1+q\left(1+\frac{p-1}{p}\rho\right)e^{\frac{p-1}{p}\rho}
\right)\\
\ge &(v_{\lambda,h})_{\tau}(x,f(\rho))=(v_{\lambda,h}(x,f(\rho))^m)_{xx}+v_{\lambda,h}(x,f(\rho))-v_{\lambda,h}(x,f(\rho))^m\quad\mbox{ in }\R\times (-\infty,\2{\tau}_0]
\end{align*}
Hence $v_{\lambda,h}(x,f(\tau))$ is a  supersolution of \eqref{yamabe-ode} in $\R\times (-\infty,\2{\tau}_0)$. Similarly $\2{v}_{\lambda,h}(x,f(\tau))$
is a  supersolution of \eqref{yamabe-ode} in $\R\times (-\infty,\2{\tau}_0]$.
\end{proof}

\begin{lem}\label{existence-lem}
Let $\lambda>1$, $\lambda'>1$,  $h\ge h_0$, $h'\ge h_0'$, $0<k\le k_0$ and let $\2{\tau}_0<0$ be as in Lemma \ref{critical-pt-x-y-z-lem}.
Let $a>-\2{\tau}_0$ and $v_0\in L^{\infty}(\R)$ satisfies 
\begin{equation}\label{v0-lower-upper-bd15}
f_{\lambda,\lambda',h,h',k}(x,-a)\le v_0(x)\le\2{f}_{\lambda,\lambda',h,h',k}(x,-a)\quad\mbox{a.e. } x\in\R
\end{equation} 
where $f$ is given by \eqref{f-defn}.
Then there exists a unique solution $\2{v}\in  C^{2,1}(\R\times (-a,\2{\tau}_0])$ of 
\begin{equation}\label{v-n1-soln}
\left\{\begin{aligned}
&v_{\tau}=(v^m)_{xx}+v-v^m,\quad v>0,\quad\mbox{ in }\R\times (-a,\2{\tau}_0)\\
&v(x,-a)=v_0(x)\quad\mbox{ on }\R
\end{aligned}\right.
\end{equation}
which satisfies
\begin{equation}\label{v-bar-lower-upper-bd}
\left\{\begin{aligned}
&f_{\lambda,\lambda',h,h',k}(x,\tau)\le\2{v}(x,\tau)\le\|v_0\|_{L^{\infty}(\R)}\quad\forall x\in\R,-a<\tau\le\2{\tau}_0\\
&\2{v}(x,\tau)\le\2{f}_{\lambda,\lambda',h,h',k}(x,\tau)\qquad\qquad\quad\forall x\in\R,-a<\tau\le\2{\tau}_0.
\end{aligned}\right.
\end{equation}
If there exists $x_0\in\R$ such that
\begin{equation}\label{v0-behaviour10}
\left\{\begin{aligned}
&v_0(x)\mbox{ is monotone increasing on }(-\infty,x_0]\\
&v_0(x)\mbox{ is monotone decreasing on }[x_0,\infty)
\end{aligned}\right.
\end{equation}
holds, then for any $-a<\tau<\2{\tau}_0$ there exists $x_a(\tau)\in\R$ such that 
\begin{equation}\label{vrx-sign20}
\left\{\begin{aligned}
&\2{v}_x(x,\tau)>0>\2{v}_x(y,\tau)\quad\forall x<x_a(\tau)<y, -a<\tau<\2{\tau}_0\\
&\2{v}_x(x_a(\tau),\tau)=0\qquad\quad\forall -a<\tau<\2{\tau}_0.
\end{aligned}\right.
\end{equation}
Moreover if $v_0(x)=\2{f}_{\lambda,\lambda',h,h',k}(x,-a)$ on $\R$, then $\2{v}(\cdot,\tau)$ is also a monotone decreasing function of $\tau\in [-a,\2{\tau}_0]$
and for any $a<\tau<\2{\tau}_0$ there exists $x_a(\tau)\in\R$ such that \eqref{vrx-sign20} holds.

Furthermore if $v_0(x)=\2{f}_{\lambda,\lambda',h,h',k}(x,-a)$ on $\R$ and $\lambda=\lambda'$, then $x_a(\tau)=\frac{h'-h}{2}$ and $v(x,\tau)$ is symmetric with respect to $x_0:=\frac{h'-h}{2}$ for any $-a<\tau<\2{\tau}_0$.
\end{lem}
\begin{proof}
Uniqueness of solution of \eqref{v-n1-soln} follows from Lemma \ref{sub-supersoln-comparison-lem}. Hence it remains to prove existence of solution of \eqref{v-n1-soln}.  By Lemma \ref{local-existence-lem} for any  $i\in\mathbb{N}$, there exists a unique solution $v_i\in  C^{2,1}([-i,i]\times (-a,\2{\tau}_0])\cap L^{\infty}((-i,i)\times (-a,\2{\tau}_0))$ of \eqref{v-soln-bd-domain} with $R=i$ which satisfies \eqref{v-n1-r-lower-upper-bd3} with $R=i$.
By \eqref{v-n1-r-lower-upper-bd3} the equation \eqref{yamabe-ode} for the sequence $\{v_i\}_{i=1}^{\infty}$ is uniformly parabolic on every compact subset of
$\R\times (-a,\2{\tau}_0]$. Hence by the parabolic Schauder estimates \cite{LSU} the sequence $\{v_i\}_{i=1}^{\infty}$ is equi-Holder continuous in $C^{2,1}(K)$ for any compact subset $K\subset \R\times (-a,\2{\tau}_0]$. Hence by the Ascoli Theorem and a diagonalization argument the sequence $\{v_i\}_{i=1}$ has a subsequence  which we may assume without loss of generality to be the sequence itself that converges uniformly in $C^{2,1}(K)$ for any compact subset $K$ of $\R\times (-a,\tau_0]$ to some function $\2{v}\in  C^{2,1}(\R\times (-a,\2{\tau}_0])$  as $i\to\infty$. Then $\2{v}$ satisfies \eqref{yamabe-ode} in $\R\times (-a,\tau_0]$. Putting $R=i$ in \eqref{v-n1-r-lower-upper-bd3} and letting $i\to\infty$ we get 
\begin{equation}\label{v-bar-lower-upper-bd0}
f_{\lambda,\lambda',h,h',k}(x,\tau)\le\2{v}(x,\tau)\le \|v_0\|_{L^{\infty}(\R)}<1\quad\forall x\in\R,-a\le\tau<\2{\tau}_0
\end{equation}
For any $\eta\in C_0^{\infty}(\R)$ such that supp $\eta\subset (-i_0,i_0)$ for some $i_0\in\mathbb{N}$, by \eqref{v-bar-lower-upper-bd0},
\begin{align*}
\left|\int_{-i_0}^{i_0}v_i(x,\tau)\eta(x)\,dx-\int_{-i_0}^{i_0}v_0\eta\,dx\right|=&\left|\int_{-a}^{\2{\tau}_0}\int_{-i_0}^{i_0}v_{i,\tau}(x,\tau)\eta(x)\,dx\,d\tau\right|\\
=&\left|\int_{-a}^{\tau}\int_{-i_0}^{i_0}(v_i^m\eta_{xx}+(v_i-v_i^m)\eta\,dx\,ds\right|\\
\le&C(\tau+a)\quad\forall i>i_0, -a<\tau<\2{\tau}_0\\
\Rightarrow\quad\left|\int_{\R}\2{v}(x,\tau)\eta(x)\,dx-\int_{\R}v_0\eta\,dx\right|\le&C(\tau+a)\quad\forall  -a<\tau<\2{\tau}_0\quad\mbox{ as }i\to\infty\\
\Rightarrow\quad\lim_{\tau\to -a}\left|\int_{\R}\2{v}(x,\tau)\eta(x)\,dx-\int_{\R}v_0\eta\,dx\right|=&0.
\end{align*}
This together with \eqref{v-bar-lower-upper-bd0} and the Lebesgue dominated convergence theorem implies that $\2{v}(\cdot,\tau)\to v_0$ in $L^1_{loc}(\R)$ as $\tau\to 0$. Hence  
$\2{v}$ satisfies \eqref{v-n1-soln}.

Since by Lemma \ref{time-rescale-supersoln-lem} both $v_{\lambda,h}(x,f(\tau))$ and $\2{v}_{\lambda,h}(x,f(\tau))$ are  supersolutions of \eqref{yamabe-ode} in $\R\times (-\infty,\2{\tau}_0]$, by \eqref{v0-lower-upper-bd15}, \eqref{v-n1-soln} and Lemma \ref{sub-supersoln-comparison-lem}, we get
\begin{equation}\label{v-bar-v-lambda-h-compare}
\2{v}(x,\tau)\le v_{\lambda,h}(x,f(\tau))\quad\mbox{ and }\quad\2{v}(x,\tau)\le\2{v}_{\lambda,h}(x,f(\tau))\quad\forall x\in\R,-a<\tau\le\2{\tau}_0.
\end{equation}
Similarly,
\begin{equation}\label{v-bar-xi-k-compare}
\2{v}(x,\tau)\le\xi_k(\tau)\quad\forall x\in\R,-a<\tau\le\2{\tau}_0.
\end{equation}
By \eqref{v-bar-lower-upper-bd0}, \eqref{v-bar-v-lambda-h-compare} and \eqref{v-bar-xi-k-compare} we get \eqref{v-bar-lower-upper-bd}. 

Let $R_0>0$ be as given by Lemma \ref{neumann-existence-lem}. By Lemma \ref{neumann-existence-lem} for any $i>R_0$ there exists a unique solution $\2{v}_i\in C^{2,1}([-i,i]\times (-a,\2{\tau}_0])\cap L^{\infty}((-i,i)\times (-a,\2{\tau}_0))$ of \eqref{v-neumann-eqn} with $R=i$ which satisfies \eqref{v-n1-r-lower-upper-bd3} with $R=i$. Then by a similar argument as before, $\2{v}_i$ converges uniformly in $C^{2,1}(K)$ for any compact subset $K$ of $\R\times (-a,\2{\tau}_0]$ to $\2{v}$  as $i\to\infty$. If there exists $x_0\in\R$ such that \eqref{v0-behaviour10} holds, then by Lemma \ref{neumann-existence-lem} for any $i>R_1:=\max (R_0,|x_0|)$ and $-a<\tau<\2{\tau}_0$ there exists $x_i(\tau)\in (-i,i)$ such that 
\begin{equation}\label{vrx-sign21}
\left\{\begin{aligned}
&\2{v}_{i,x}(x,\tau)>0>\2{v}_{i,x}(y,\tau)\quad\forall -i\le x<x_i(\tau)<y\le i, -a<\tau<\2{\tau}_0\\
&\2{v}_{i,x}(x_i(\tau),\tau)=0\quad\,\,\forall -a<\tau<\2{\tau}_0.
\end{aligned}\right.
\end{equation}
Then for any  $-a<\tau<\2{\tau}_0$ the sequence $\{x_i(\tau)\}_{i>R_1}$ has a subsequence $\{x_{i_l}(\tau)\}$ converging to some point $x_0(\tau)\in\R\cup\{\pm\infty\}$ as $l\to\infty$. If there exists $\tau_1\in (-a,\2{\tau}_0)$ such that $x_0(\tau_1)=\infty$, then by \eqref{vrx-sign21},
\begin{equation}\label{v-bar-x>0}
\2{v}_x(x,\tau_1)\ge 0\quad\forall x\in\R.
\end{equation}
By \eqref{v-bar-lower-upper-bd}, 
\begin{equation}\label{v-bar-x-+-infty-behaviour}
\2{v}(x,\tau)\to 0\quad\mbox{ as }|x|\to\infty\quad\forall -a<\tau<\2{\tau}_0.
\end{equation} 
Hence by \eqref{v-bar-x>0} and \eqref{v-bar-x-+-infty-behaviour}, 
\begin{equation}\label{v-bar=0}
\2{v}(x,\tau_1)=0\quad\forall x\in\R
\end{equation}
and contradiction arises since $v(x,\tau_1)>0$ for any $x\in\R$.
Hence $x_0(\tau)\ne\infty$ for any $\tau\in (-a,\2{\tau}_0)$. Similarly $x_0(\tau)\ne -\infty$. Hence $x_0(\tau)\in\R$ for any $\tau\in (-a,\2{\tau}_0)$. Let
\begin{equation*}
D_0(\tau)=\{x\in\R:\2{v}_x(x,\tau)=0\}\quad\forall\tau\in (-a,\2{\tau}_0).
\end{equation*}
Then $x_0(\tau)\in D_0(\tau)$ for any $-a<\tau<\2{\tau}_0$. Let $-a<\tau_1<\2{\tau}_0$. By \eqref{v-bar-lower-upper-bd} the equation \eqref{v-x-eqn0} for $v_x$ is uniformly parabolic on any compact subset $K$ of $\R\times [-a,\2{\tau}_0]$. 
Hence by the discussion on P.241 of \cite{SGKM} and \cite {CP} $\2{v}_x(x,\tau_1)$ is an analytic function of $x\in\R$. Hence either all the zeros of $\2{v}_x(x,\tau_1)$ are isolated zeros or 
\begin{equation}\label{vrx-all=0-1}
\2{v}_x(x,\tau_1)=0\quad\forall x\in\R.
\end{equation} 
If \eqref{vrx-all=0-1} holds, then by \eqref{v-bar-x-+-infty-behaviour} we get \eqref{v-bar=0} and contradiction arises.  Hence \eqref{vrx-all=0-1} is not possible and all the zeros of $\2{v}_x(x,\tau_1)$ are isolated zeros. Putting $i=i_l$ and letting $l\to\infty$ in \eqref{vrx-sign21},
\begin{equation}\label{vrx-sign22}
\2{v}_x(x,\tau_1)\ge 0\ge\2{v}_x(y,\tau_1)\quad\forall x\le x_0(\tau_1)\le y.
\end{equation}
By \eqref{vrx-sign22} an argument similar to the proof of Lemma \ref{neumann-existence-lem} $D_0(\tau_1)\cap (x_0(\tau_1),\infty)=\phi$ and $D_0(\tau_1)\cap (-\infty,x_0(\tau_1))=\phi$. Hence $D_0(\tau_1)=\{x_0(\tau_1)\}$. Since $\2{v}_x(x,\tau_1)$ is an analytic function of $x\in\R$, by \eqref{vrx-sign22},
\begin{equation*}
\2{v}_x(x,\tau_1)>0>\2{v}_x(y,\tau_1)\quad\forall x<x_0(\tau_1)<y
\end{equation*}
and \eqref{vrx-sign20} follows.

We now let $v_0(x)=\2{f}_{\lambda,\lambda',h,h',k}(x,-a)$. Since by Lemma \ref{elliptic-supersoln-lem} $\2{f}_{\lambda,\lambda',h,h',k}(x,-a)$ is a supersolution of \eqref{yamabe-ode}, by \eqref{v0-lower-upper-bd15} and Lemma \ref{sub-supersoln-comparison-lem},
\begin{align*}
&\2{v}(x,\tau)\le \2{f}_{\lambda,\lambda',h,h',k}(x,-a)\quad\forall x\in\R,-a\le\tau<\2{\tau}_0\\
\Rightarrow\quad&\2{v}(x,\tau_1+\tau)\le \2{v}(x,-a+\tau)\quad\forall x\in\R,-a\le\tau_1<\2{\tau}_0,0\le\tau\le\2{\tau}_0-\tau_1.
\end{align*}
Hence $\2{v}(\cdot,\tau)$ is also a monotone decreasing function of $\tau\in [-a,\tau_0]$.

Let $y(\tau)$ and $z(\tau)$ be as in Lemma \ref{critical-pt-x-y-z-lem}. By Lemma \ref{critical-pt-x-y-z-lem} $v_0(x)=v_{\lambda,h}(x,f(-a))$ is a monotone increasing function on  $(-\infty, y(-a))$, $v_0(x)=\xi_k(-a)$ is a constant function on $[y(-a),z(-a)]$, and  $v_0(x)=\2{v}_{\lambda,h}(x,f(-a))$ is a monotone decreasing function on $[z(-a),\infty)$. Hence $v_0$ satisfies  \eqref{v0-behaviour10} with $x_0=y(-a)$.  Thus by the above argument for any $-a<\tau<\2{\tau}_0$ there exists $x_a(\tau)\in\R$ such that \eqref{vrx-sign20} holds.

Suppose now $v_0(x)=\2{f}_{\lambda,\lambda',h,h',k}(x,-a)$ on $\R$ and $\lambda=\lambda'$. Let $x_0:=\frac{h'-h}{2}$. Then $v_0(x+x_0)=v_0(-x+x_0)$ for any $x\in\R$.  Since both $\2{v}(x+x_0,\tau)$ and $\2{v}(-x+x_0,\tau)$ satisfies \eqref{v-n1-soln} and \eqref{v-bar-lower-upper-bd}, by Lemma \ref{sub-supersoln-comparison-lem} and \eqref{vrx-sign20},  
\begin{align*}
&\2{v}(x+x_0,\tau)=\2{v}(-x+x_0,\tau)\quad\forall x\in\R, -a\le\tau\le\2{\tau}_0\notag\\
\Rightarrow\quad&\2{v}_x(x_0,\tau)=0\quad\forall -a\le\tau\le\2{\tau}_0\\
\Rightarrow\quad&x_a(\tau)=x_0\quad\forall -a\le\tau\le\2{\tau}_0.
\end{align*}
\end{proof}

By Lemma \ref{local-existence-lem2}, Lemma \ref{local-existence-lem3}, Lemma \ref{neumann-existence-lem2}, Lemma \ref{neumann-existence-lem3}, Lemma \ref{time-rescale-supersoln-lem} and a similar argument we have the following two results.

\begin{lem}\label{existence-lem1}
Let $\lambda>1$, $\lambda'>1$,  $h\ge h_0$, $h'\ge h_0'$ and let $\2{\tau}_0<0$ be as in Lemma \ref{critical-pt-x-y-z-lem}.
Let $a>-\2{\tau}_0$ and $v_0\in L^{\infty}(\R)$ satisfies 
\begin{equation*}
f_{\lambda,\lambda',h,h'}(x,-a)\le v_0(x)\le\2{f}_{\lambda,\lambda',h,h'}(x,-a)\quad\mbox{a.e. } x\in\R
\end{equation*} 
where $f$ is given by \eqref{f-defn}.
Then there exists a unique solution $\2{v}\in C^{2,1}(\R\times (-a,\2{\tau}_0])$ of 
\eqref{v-n1-soln} which satisfies
\begin{equation*}
\left\{\begin{aligned}
&f_{\lambda,\lambda',h,h'}(x,\tau)\le\2{v}(x,\tau)\le\|v_0\|_{L^{\infty}(\R)}\quad\forall x\in\R,-a<\tau\le\2{\tau}_0\\
&\2{v}(x,\tau)\le\2{f}_{\lambda,\lambda',h,h'}(x,\tau)\quad\forall x\in\R,-a<\tau\le\2{\tau}_0.
\end{aligned}\right.
\end{equation*}
If there exists $x_0\in\R$ such that \eqref{v0-behaviour10} holds, then for any $-a<\tau<\2{\tau}_0$ there exists $x_a(\tau)\in\R$ such that 
\eqref{vrx-sign20} holds. 

Moreover if $v_0(x)=\2{f}_{\lambda,\lambda',h,h'}(x,-a)$, then $\2{v}(\cdot,\tau)$ is a monotone decreasing function of $\tau\in [-a,\2{\tau}_0]$ and for any $a<\tau<\2{\tau}_0$ there exists $x_a(\tau)\in\R$ such that \eqref{vrx-sign20} holds.

Furthermore if $v_0(x)=\2{f}_{\lambda,\lambda',h,h'}(x,-a)$ on $\R$ and $\lambda=\lambda'$, then $x_a(\tau)=\frac{h'-h}{2}$ and $v(x,\tau)$ is symmetric with respect to $x_0:=\frac{h'-h}{2}$ for any $-a<\tau<\2{\tau}_0$. 
\end{lem}

\begin{lem}\label{existence-lem2}
Let $\lambda>1$,  $h\ge h_0$, $0<k\le k_0$ and let $\2{\tau}_0<0$ be as in Lemma \ref{critical-pt-x-y-z-lem}.  
Let $a>-\2{\tau}_0$ and $v_0\in L^{\infty}(\R)$ satisfies 
\begin{equation*}
f_{\lambda,h,k}(x,-a)\le v_0(x)\le\2{f}_{\lambda,h,k}(x,-a)\quad\mbox{a.e. } x\in\R
\end{equation*} 
where $f$ is given by \eqref{f-defn}.
Then there exists a unique solution $\2{v}\in C^{2,1}(\R\times (-a,\2{\tau}_0])$ of 
\eqref{v-n1-soln} which satisfies
\begin{equation*}
\left\{\begin{aligned}
&f_{\lambda,h,k}(x,\tau)\le\2{v}(x,\tau)\le\|v_0\|_{L^{\infty}(\R)}\quad\forall x\in\R,-a<\tau\le\2{\tau}_0\\
&\2{v}(x,\tau)\le\2{f}_{\lambda,h,k}(x,\tau)\quad\forall x\in\R,-a<\tau\le\2{\tau}_0.
\end{aligned}\right.
\end{equation*}
If $v_0$ is monotone increasing in $\R$, then 
\begin{equation*}
v_x(x,\tau)>0\quad\forall x\in\R, -a<\tau<\2{\tau}_0.
\end{equation*}
Moreover if $v_0(x)=\2{f}_{\lambda,h,k}(x,-a)$, then $\2{v}(\cdot,\tau)$ is a monotone decreasing function of $\tau\in [-a,\2{\tau}_0]$. 
\end{lem}

\begin{thm}\label{existence-thm1}
Let $\lambda>1$, $\lambda'>1$,  $h\ge h_0$, $h'\ge h_0'$, $0<k\le k_0$  and let $\2{\tau}_0<0$ be as in Lemma \ref{critical-pt-x-y-z-lem}. Let $j_0\in\mathbb{N}$ be such that $j_0>-\2{\tau}_0$ and $\{v_{0,j}\}_{j\ge j_0}\subset L^{\infty}(\R)$ be such that
\begin{equation*}
f_{\lambda,\lambda',h,h',k}(x,-j)\le v_{0,j}(x)\le\2{f}_{\lambda,\lambda',h,h',k}(x,-j)\quad\mbox{a.e. } x\in\R\quad\forall  j\ge j_0.
\end{equation*} 
Let $v_j\in C^{2,1}(\R\times (-j,\2{\tau}_0])$ be the unique solution of \eqref{v-n1-soln} in $\R\times (-j,\2{\tau}_0)$ with $a=j$ and  $v_0=v_{0,j}$ given by Lemma \ref{existence-lem}. 
Then the sequence $\{v_j\}_{j\ge j_0}$ has a subsequence $\{v_{j_l}\}$ that converges uniformly in $C^{2,1}(K)$ for any compact set $K\subset\R\times (-\infty,\2{\tau}_0]$  to a  solution $v=v_{\lambda,\lambda',h,h',k}\in C^{2,1}(\R\times (-\infty,\2{\tau}_0])$ of 
\eqref{yamabe-ode} in $\R\times (-\infty,\2{\tau}_0]$ which satisfies \eqref{v-lambda-lambda'-h-h'-k-lower-upper-bd} as $l\to\infty$.

If for each $j\ge j_0$, there exists $x_{0,j}\in\R$ such that
\begin{equation}\label{v0-behaviour25}
\left\{\begin{aligned}
&v_{0,j}(x)\mbox{ is monotone increasing on }(-\infty,x_{0,j}]\\
&v_{0,j}(x)\mbox{ is monotone decreasing on }[x_{0,j},\infty)
\end{aligned}\right.
\end{equation}
holds, then for any $\tau<\2{\tau}_0$ there exists $x_0(\tau)\in\R$ such that 
\begin{equation}\label{vrx-sign25}
\left\{\begin{aligned}
&v_x(x,\tau)>0>v_x(y,\tau)\quad\forall x<x_0(\tau)<y, \tau<\2{\tau}_0\\
&v_x(x_0(\tau),\tau)=0\qquad\quad\forall \tau<\2{\tau}_0.
\end{aligned}\right.
\end{equation}
If $v_{0,j}(x)=\2{f}_{\lambda,\lambda',h,h',k}(x,-j)$ for all $j\ge j_0$, the solution $v(x,\tau)$ is a monotone decreasing function of $\tau\in (-\infty,\2{\tau}_0]$ and for any $\tau<\2{\tau}_0$ there exists $x_0(\tau)\in\R$ such that \eqref{vrx-sign25} holds. Moreover in this case, $v_j$  will converge uniformly in $C^{2,1}(K)$ for any compact set $K\subset\R\times (-\infty,\2{\tau}_0]$  to $v_{\lambda,\lambda',h,h',k}$ as $j\to\infty$.

Furthermore if $v_{0,j}(x)=\2{f}_{\lambda,\lambda',h,h',k}(x,-j)$ on $\R$ and $\lambda=\lambda'$, then $x_0(\tau)=\frac{h'-h}{2}$ and $v(x,\tau)$ is symmetric with respect to $x_0:=\frac{h'-h}{2}$ for any $\tau<\2{\tau}_0$.
\end{thm}
\begin{proof}
By Lemma \ref{existence-lem}, 
\begin{equation}\label{vj-lower-upper-bd}
f_{\lambda,\lambda',h,h',k}(x,\tau)\le v_j(x,\tau)\le\2{f}_{\lambda,\lambda',h,h',k}(x,\tau)\quad\forall x\in\R,-j<\tau\le\2{\tau}_0,j\ge j_0.
\end{equation}
By \eqref{vj-lower-upper-bd} the equation \eqref{yamabe-ode} for the sequence $\{v_j\}_{j\ge j_0}$ is uniformly parabolic on every compact subset of
$\R\times (-\infty,\tau_0]$. Then by the parabolic Schauder estimates \cite{LSU} the sequence $\{v_j\}_{j\ge j_0}$ is equi-Holder continuous in $C^{2,1}(K)$ for any compact set $K\subset \R\times (-\infty,\2{\tau}_0]$. Hence by the Ascoli Theorem and a diagonalization argument the sequence $\{v_j\}_{j\ge j_0}$  has a subsequence $\{v_{j_l}\}_{l=1}^{\infty}$ that converges uniformly in $C^{2,1}(K)$ for any compact set $K\subset \R\times (-\infty,\2{\tau}_0]$ to some solution $v=v_{\lambda,\lambda',h,h',k}\in C^{2,1}(\R\times (-\infty,\2{\tau}_0])$ of \eqref{yamabe-ode} in $\R\times (-\infty,\2{\tau}_0)$ as $l\to\infty$. 
Letting $j=j_l\to\infty$ in \eqref{vj-lower-upper-bd}, we get \eqref{v-lambda-lambda'-h-h'-k-lower-upper-bd}.

Suppose now for each $j\ge j_0$, there exists $x_{0,j}\in\R$ such that \eqref{v0-behaviour25} holds. Then by Lemma \ref{existence-lem} for any $j\ge j_0$ and
$-j<\tau<\2{\tau}_0$ there exists $x_j(\tau)\in\R$ such that 
\begin{equation}\label{vjx-sign}
\left\{\begin{aligned}
&v_{j,x}(x,\tau)>0>v_{j,x}(y,\tau)\quad\forall x<x_j(\tau)<y, -j<\tau<\2{\tau}_0\\
&v_{j,x}(x_a(\tau),\tau)=0\qquad\quad\forall -j<\tau<\2{\tau}_0.
\end{aligned}\right.
\end{equation}
Then by \eqref{vj-lower-upper-bd}, \eqref{vjx-sign} and an argument similar to the proof of Lemma \ref{existence-lem}, $x_0(\tau):=\lim_{j\to\infty}x_j(\tau)\in\R$ exists for any $\tau<\2{\tau}_0$ and $x_0(\tau)$ satisfies \eqref{vrx-sign25}.

If $v_{0,j}(x)=\2{f}_{\lambda,\lambda',h,h',k}(x,-j)$ for all $j\ge j_0$, then  by Lemma \ref{existence-lem} for any $j\ge j_0$ the solution $v_j(\cdot,\tau)$ is a monotone decreasing function of $\tau\in (-j,\2{\tau}_0]$. Hence $v(\cdot,\tau)$ is a monotone decreasing function of $\tau\in (-\infty,\2{\tau}_0]$.
Let $y(\tau)$ and $z(\tau)$ be as in Lemma \ref{critical-pt-x-y-z-lem}. By Lemma \ref{critical-pt-x-y-z-lem} $v_{0,j}(x)=v_{\lambda,h}(x,f(-j))$ is a monotone increasing function on  $(-\infty, y(-j))$, $v_{0,j}(x)=\xi_k(-j)$ is a constant function on $[y(-j),z(-j)]$, and  $v_{0,j}(x)=\2{v}_{\lambda,h}(x,f(-j))$ is a monotone decreasing function on $[z(-j),\infty)$. Hence $v_{0,j}$ satisfies  \eqref{v0-behaviour25} with $x_{0,j}=y(-j)$.  Thus by the above argument for any $\tau<\2{\tau}_0$ there exists $x_0(\tau)\in\R$ such that \eqref{vrx-sign25} holds.

Suppose $\{v_{j_l'}\}_{l=1}^{\infty}$ is a another subsequence of $\{v_j\}_{j\ge j_0}$ which converges uniformly in $C^{2,1}(K)$ for any compact set $K\subset \R\times (-\infty,\2{\tau}_0]$ to some solution $\2{v}_{\lambda,\lambda',h,h',k}\in C^{2,1}(\R\times (-\infty,\2{\tau}_0])$  of \eqref{yamabe-ode} in $\R\times (-\infty,\2{\tau}_0)$ as $l\to\infty$ and $\2{v}_{\lambda,\lambda',h,h',k}$ also satisfies \eqref{v-lambda-lambda'-h-h'-k-lower-upper-bd}. Then by \eqref{v-lambda-lambda'-h-h'-k-lower-upper-bd} and Lemma \ref{sub-supersoln-comparison-lem},
\begin{align}
&v_{\lambda,\lambda',h,h',k}(x,-j)\le \2{f}_{\lambda,\lambda',h,h',k}(x,-j)=v_{0,j}(x)\quad\forall x\in\R,j\ge j_0\notag\\
\Rightarrow\quad&v_{\lambda,\lambda',h,h',k}(x,\tau)\le v_j(x,\tau)\qquad\qquad\qquad\forall x\in\R,-j<\tau<\2{\tau}_0,j\ge j_0\notag\\
\Rightarrow\quad&v_{\lambda,\lambda',h,h',k}(x,\tau)\le\2{v}_{\lambda,\lambda',h,h',k}(x,\tau)\qquad\qquad\forall x\in\R,\tau<\2{\tau}_0\quad\mbox{ as }j=j_l'\to\infty.\label{2-ancient-limit-compare}
\end{align}
By interchanging the role of $v_{\lambda,\lambda',h,h',k}$ and $\2{v}_{\lambda,\lambda',h,h',k}$ in \eqref{2-ancient-limit-compare},
\begin{equation}\label{2-ancient-limit-compare2}
\2{v}_{\lambda,\lambda',h,h',k}(x,\tau)\le v_{\lambda,\lambda',h,h',k}(x,\tau)\quad\forall x\in\R,\tau<\2{\tau}_0.
\end{equation}
Hence by \eqref{2-ancient-limit-compare} and \eqref{2-ancient-limit-compare2},
\begin{equation*}
v_{\lambda,\lambda',h,h',k}(x,\tau)=\2{v}_{\lambda,\lambda',h,h',k}(x,\tau)\quad\forall x\in\R,\tau<\2{\tau}_0.
\end{equation*}
Thus $v_j$ converges uniformly in $C^{2,1}(K)$ for any compact set $K\subset\R\times (-\infty,\2{\tau}_0]$  to $v_{\lambda,\lambda',h,h',k}$ as $j\to\infty$.

Suppose now $v_{0,j}(x)=\2{f}_{\lambda,\lambda',h,h',k}(x,-j)$ on $\R$ for any $j\ge j_0$ and $\lambda=\lambda'$. Then by Lemma \ref{existence-lem} $v_j$ is symmetric with respect to $x_0:=\frac{h'-h}{2}$ and $x_j(\tau)=\frac{h'-h}{2}=x_0$ for any $-j<\tau<\2{\tau}_0$. Hence $x_0(\tau)=lim_{j\to\infty}x_j(\tau)=\frac{h'-h}{2}=x_0$ and $v(x,\tau)$ is symmetric with respect to $x_0:=\frac{h'-h}{2}$ for any $\tau<\2{\tau}_0$.
\end{proof}

By  an argument similar to the proof of Theorem \ref{existence-thm1} but with Lemma \ref{existence-lem1} (Lemma \ref{existence-lem2} respectively) replacing Lemma \ref{existence-lem}  in the proof we obtain the following two theorems. 

\begin{thm}\label{existence-thm2}
Let $\lambda>1$, $\lambda'>1$,  $h\ge h_0$, $h'\ge h_0'$  and let $\2{\tau}_0<0$ be as in Lemma \ref{critical-pt-x-y-z-lem}.
Let $j_0\in\mathbb{N}$ be such that $j_0>-\2{\tau}_0$ and $\{\2{v}_{0,j}\}_{j\ge j_0}\subset L^{\infty}(\R)$ be such that
\begin{equation*}
f_{\lambda,\lambda',h,h'}(x,-j)\le\2{v}_{0,j}(x)\le\2{f}_{\lambda,\lambda',h,h'}(x,-j)\quad\forall x\in\R,j_0\le j\in\mathbb{N}.
\end{equation*} 
Let $\2{v}_j$ be the unique solution of \eqref{v-n1-soln} in $\R\times (-j,\2{\tau}_0)$ with $a=j$ and  $v_0=\2{v}_{0,j}$ given by Lemma \ref{existence-lem1}. Then the sequence $\{\2{v}_j\}_{j\ge j_0}$ has a subsequence $\{\2{v}_{j_l}\}$ that converges uniformly in $C^{2,1}(K)$ for any compact set $K\subset\R\times (-\infty,\2{\tau}_0]$  to a  solution $v=v_{\lambda,\lambda',h,h'}\in C^{2,1}(\R\times (-\infty,\2{\tau}_0])$ of 
\eqref{yamabe-ode} in $\R\times (-\infty,\2{\tau}_0]$ which satisfies \eqref{v-lambda-lambda'-h-h'-upper-lower-bd} as $l\to\infty$.

If for each $j\ge j_0$, there exists $x_{0,j}\in\R$ such that \eqref{v0-behaviour25} holds, then for any $\tau<\2{\tau}_0$ there exists $x_0(\tau)\in\R$ such that 
\eqref{vrx-sign25} holds.

If  $\2{v}_{0,j}(x)=\2{f}_{\lambda,\lambda',h,h'}(x,-j)$ for all $j\ge j_0$, the solution $v(x,\tau)$ is a monotone decreasing function of $\tau\in (-\infty,\2{\tau}_0]$ and for any $\tau<\2{\tau}_0$ there exists $x_0(\tau)\in\R$ such that \eqref{vrx-sign25} holds. Moreover in this case, $v_j$  will converge uniformly in $C^{2,1}(K)$ for any compact set $K\subset\R\times (-\infty,\2{\tau}_0]$  to $v_{\lambda,\lambda',h,h'}$ as $j\to\infty$. 

Furthermore if $v_{0,j}(x)=\2{f}_{\lambda,\lambda',h,h'}(x,-j)$ on $\R$ and $\lambda=\lambda'$, then $x_0(\tau)=\frac{h'-h}{2}$ and $v(x,\tau)$ is symmetric with respect to $x_0:=\frac{h'-h}{2}$ for any $\tau<\2{\tau}_0$. 
\end{thm}

\begin{thm}\label{existence-thm3}
Let $\lambda>1$,  $h\ge h_0$, $0<k\le k_0$  and let $\2{\tau}_0<0$ be as in Lemma \ref{critical-pt-x-y-z-lem}.
Let $j_0\in\mathbb{N}$ be such that $j_0>-\2{\tau}_0$ and  $\{\4{v}_{0,j}\}_{j\ge j_0}\subset L^{\infty}(\R)$ be such that
\begin{equation*}
f_{\lambda,h,k}(x,-j)\le\4{v}_{0,j}(x)\le\2{f}_{\lambda,h,k}(x,-j)\quad\forall x\in\R,j_0\le j\in\mathbb{N}.
\end{equation*} 
Let $\4{v}_j$ be the unique solution of \eqref{v-n1-soln} in $\R\times (-j,\2{\tau}_0)$ with $a=j$ and  $v_0=\4{v}_{0,j}$ given by Lemma \ref{existence-lem2}. Then the sequence $\{\4{v}_j\}_{j\ge j_0}$ has a subsequence $\{\2{v}_{j_l}\}$ that converges uniformly in $C^{2,1}(K)$ for any compact set $K\subset\R\times (-\infty,\2{\tau}_0]$  to a  solution $v=v_{\lambda,h,k}\in C^{2,1}(\R\times (-\infty,\2{\tau}_0])$ of 
\eqref{yamabe-ode} in $\R\times (-\infty,\2{\tau}_0]$ which satisfies  \eqref{v-lambda-h-k-upper-lower-bd} as $l\to\infty$.

If for each $j\ge j_0$, $v_{0,j}$ is monotone increasing on $\R$, then 
\begin{equation*}
v_x(x,\tau)>0\quad\forall x\in\R,\tau<\2{\tau}_0.
\end{equation*}

If  $\4{v}_{0,j}(x)=\2{f}_{\lambda,h,k}(x,-j)$ for all $j\ge j_0$, the solution $v(x,\tau)$ is a monotone decreasing function of $\tau\in (-\infty,\2{\tau}_0]$ and  $v_j$  will converge uniformly in $C^{2,1}(K)$ for any compact set $K\subset\R\times (-\infty,\2{\tau}_0]$  to $v_{\lambda,h,k}$ as $j\to\infty$.
\end{thm}

\begin{rmk}
Existence of solutions of \eqref{v-n1-soln} for $v_0(x)=\min (v_{\lambda,h}(x,\tau),\2{v}_{\lambda',h'}(x,\tau))$ or $f_{\lambda,\lambda',h,h',k}(x,-a)$ are also given in \cite{DPKS1} and \cite{DPKS2} respectively. 
Existence of solution $v_{\lambda,\lambda',h,h',k}$ of Theorem \ref{existence-thm1} with $v_{0,j}(x)=f_{\lambda,\lambda',h,h',k}(x,-j)$  is also proved in
\cite{DPKS2} and existence of solution $v_{\lambda,\lambda',h,h'}$  of  Theorem \ref{existence-thm2} with $\2{v}_{0,j}(x)=\min (v_{\lambda,h}(x,-j),\2{v}_{\lambda',h'}(x,-j))$  is also proved in \cite{DPKS1}. 
\end{rmk}

\begin{rmk}\label{comparison-rmk0}
Suppose $v_{\lambda,\lambda',h,h',k}$, $v_{\lambda,\lambda',h,h'}$ and  $v_{\lambda,h,k}$ are the solutions of \eqref{yamabe-ode} in $\R\times (-\infty,\2{\tau}_0)$ constructed in Theorem \ref{existence-thm1}, Theorem \ref{existence-thm2} and Theorem \ref{existence-thm3}
by setting $v_{0,j}(x)=\2{f}_{\lambda,\lambda',h,h',k}(x,-j)$, $\2{v}_{0,j}(x)=\2{f}_{\lambda,\lambda',h,h'}(x,-j)$ and  $\4{v}_{0,j}(x)=\2{f}_{\lambda,h,k}(x,-j)$ in Theorem \ref{existence-thm1}, Theorem \ref{existence-thm2} and Theorem \ref{existence-thm3} respectively.
Let $\2{v}_j$, $\4{v}_j$,  be the unique solutions of \eqref{v-n1-soln} in $\R\times (-j,\2{\tau}_0)$ with $a=j$ and  $v_0=\2{v}_{0,j}, \4{v}_{0,j}$, respectively given by Lemma \ref{existence-lem1} and Lemma \ref{existence-lem2}.  Let $j_0\in\mathbb{N}$ be such that $j_0>-\2{\tau}_0$. 

By \eqref{v-lambda-lambda'-h-h'-k-lower-upper-bd}, Lemma \ref{sub-supersoln-comparison-lem}, Theorem \ref{existence-thm1}, Theorem \ref{existence-thm2} and Theorem \ref{existence-thm3}, we have
\begin{align*}
&v_{\lambda,\lambda',h,h',k}(x,-j)\le\2{f}_{\lambda,\lambda',h,h',k}(x,-j)\le\2{v}_{0,j}(x)\quad\forall x\in\R, j>j_0\notag\\
\Rightarrow\quad&v_{\lambda,\lambda',h,h',k}(x,\tau)\le\2{v}_j(x,\tau)\qquad\qquad\qquad\qquad\forall x\in\R,-j<\tau<\2{\tau}_0, j>-j_0\notag\\
\Rightarrow\quad&v_{\lambda,\lambda',h,h',k}(x,\tau)\le v_{\lambda,\lambda',h,h'}(x,\tau)\qquad\qquad\qquad\forall x\in\R,\tau<\2{\tau}_0\quad\mbox{ as }j\to\infty
\end{align*}
and 
\begin{align*}
&v_{\lambda,\lambda',h,h',k}(x,-j)\le\2{f}_{\lambda,\lambda',h,h',k}(x,-j)\le\4{v}_{0,j}(x)\quad\forall x\in\R,j>j_0\notag\\
\Rightarrow\quad&v_{\lambda,\lambda',h,h',k}(x,\tau)\le\4{v}_j(x,\tau)\qquad\qquad\qquad\qquad\forall x\in\R,-j<\tau<\2{\tau}_0, j>-j_0\notag\\
\Rightarrow\quad&v_{\lambda,\lambda',h,h',k}(x,\tau)\le v_{\lambda,h,k}(x,\tau)\qquad\qquad\qquad\quad\forall x\in\R,\tau<\2{\tau}_0\quad\mbox{ as }j\to\infty.
\end{align*}
Hence \eqref{v-v-compare} and \eqref{v-v-compare2} hold.
\end{rmk}

\begin{rmk}\label{comparison-rmk1}
Let $\lambda>1$, $\lambda'>1$, $h_2\ge h_1\ge h_0$, $h_2'\ge h_2\ge h_0$ and $0<k_2\le k_1\le k_0$.
Suppose $v_{\lambda,\lambda',h_1,h_1',k_1}$ and  $v_{\lambda,\lambda',h_2,h_2',k_2}$ are the solutions of \eqref{yamabe-ode} in $\R\times (-\infty,\2{\tau}_0)$ constructed in Theorem \ref{existence-thm1} by setting $v_{0,j}(x)=\2{f}_{\lambda,\lambda',h_1,h_1',k_1}(x,-j), \2{f}_{\lambda,\lambda',h_2,h_2',k_2}(x,-j)$,  in Theorem \ref{existence-thm1} respectively. Let $v_j$ be the unique solution of \eqref{v-n1-soln} in $\R\times (-j,\2{\tau}_0)$ with $a=j$ and  $v_0=\2{f}_{\lambda,\lambda',h_2,h_2',k_2}(x,f(-j))$ given by Lemma \ref{existence-lem}. Let $j_0\in\mathbb{N}$ be such that $j_0>-\2{\tau}_0$. 
By \eqref{v-lambda-lambda'-h-h'-k-lower-upper-bd},  Lemma \ref{sub-supersoln-comparison-lem} and  Theorem \ref{existence-thm1},
\begin{align*}
&v_{\lambda,\lambda',h_1,h_1',k_1}(x,-j)\le\2{f}_{\lambda,\lambda',h_1,h_1',k_1}(x,-j)\le \2{f}_{\lambda,\lambda',h_2,h_2',k_2}(x,-j)\quad\forall x\in\R,j>j_0\notag\\
\Rightarrow\quad&v_{\lambda,\lambda',h_1,h_1',k_1}(x,\tau)\le v_j(x,\tau)\qquad\qquad\qquad\qquad\qquad\qquad\qquad\forall x\in\R,-j<\tau<\2{\tau}_0, j>-j_0\notag\\
\Rightarrow\quad&v_{\lambda,\lambda',h_1,h_1',k_1}(x,\tau)\le v_{\lambda,\lambda',h_2,h_2',k_2}(x,\tau)\qquad\qquad\qquad\qquad\qquad\forall x\in\R,\tau<\2{\tau}_0\quad\mbox{ as }j\to\infty.
\end{align*}
and (iv) of Theorem \ref{5-parameters-soln-thm} follows. 
\end{rmk}

\begin{thm}\label{uniqueness-thm}
Let $\lambda>1$, $\lambda'>1$,  $h\ge h_0$, $h'\ge h_0'$ and let $\2{\tau}_0<0$ be as in Lemma \ref{critical-pt-x-y-z-lem}. Let $V_1$, $V_2\in C^{2,1}(\R\times (-\infty,\2{\tau}_0])$, be solutions of \eqref{yamabe-ode} in $\R\times (-\infty,\2{\tau}_0)$ which satisfies \eqref{v-lambda-lambda'-h-h'-upper-lower-bd}. Then $V_1=V_2$ in $\R\times (-\infty,\2{\tau}_0]$. 
\end{thm}
\begin{proof}
Since both $V_1$ and $V_2$ satisfies \eqref{v-lambda-lambda'-h-h'-upper-lower-bd}, 
\begin{equation}
|V_1-V_2|\le\2{f}_{\lambda, \lambda',h,h'}(x,\tau)-f_{\lambda, \lambda',h,h'}(x,\tau)\quad\mbox{ in }\R\times (-\infty,\2{\tau}_0].
\end{equation}
Let $\tau\le\2{\tau}_0$ and $x(\tau)$ be as in Lemma \ref{critical-pt-x-y-z-lem}. We first claim that  $\2{f}_{\lambda, \lambda',h,h'}(\cdot,\tau)-f_{\lambda, \lambda',h,h'}(\cdot,\tau)\in L^1(\R)$. To prove the claim we observe that by \eqref{v-lambda-h-value-at-infty} and \eqref{f-bar-value0} for any $x\le x(\tau)$,
\begin{equation}\label{f-bar-minus-infty-behaviour}
\2{f}_{\lambda, \lambda',h,h'}(x,\tau)=v_{\lambda,h}(x,f(\tau))=O(e^{p(x-\lambda f(\tau) +h)})\quad\mbox{ as }\quad x\to-\infty.
\end{equation} 
On the other hand by  \eqref{v-lambda-h-value-at-infty} and \eqref{v-lambda-h-value-at-infty2}, 
\begin{align}\label{f-lambda-lambda'-h-h'-minus-infty-behaviour}
f_{\lambda, \lambda',h,h'}(x,\tau)=&v_{\lambda,h}(x,f(\tau))[1+v_{\lambda,h}(x,f(\tau))^{p-1}(\2{v}_{\lambda',h'}(x,f(\tau))^{1-p}-1)]^{-\frac{1}{p-1}}\notag\\
=&v_{\lambda,h}(x,f(\tau))\left(1-\frac{1}{p-1}v_{\lambda,h}(x,f(\tau))^{p-1}\left(\2{v}_{\lambda',h'}(x,f(\tau))^{1-p}-1\right)+o(v_{\lambda,h}(x,f(\tau))^{p-1})\right)
\end{align}
as $x\to -\infty$ where $f(\tau)$ is given by \eqref{f-defn}.
By \eqref{v-lambda-lambda'-h-h'-upper-lower-bd}, \eqref{f-bar-minus-infty-behaviour} and \eqref{f-lambda-lambda'-h-h'-minus-infty-behaviour},
\begin{align}
0\le\2{f}_{\lambda, \lambda',h,h'}(x,\tau)-f_{\lambda, \lambda',h,h'}(x,\tau)
=&\frac{1}{p-1}v_{\lambda,h}(x,f(\tau))^p
\left(\2{v}_{\lambda',h'}(x,f(\tau))^{1-p}-1\right)+o\left(v_{\lambda,h}(x,f(\tau))^p\right)
\label{f-bar--f-minus-infty-behaviour0}\\
\le&O(v_{\lambda,h}(x,f(\tau))^p)=O(e^{p^2(x-\lambda f(\tau) +h)})\quad\mbox{ as }x\to -\infty.\label{f-bar--f-minus-infty-behaviour}
\end{align}
Similarly,
\begin{align}
0\le\2{f}_{\lambda, \lambda',h,h'}(x,\tau)-f_{\lambda, \lambda',h,h'}(x,\tau)
=&\frac{1}{p-1}\2{v}_{\lambda',h'}(x,f(\tau))^p
\left(v_{\lambda,h}(x,f(\tau))^{1-p}-1\right)+o(\2{v}_{\lambda',h'}(x,f(\tau))^p)\notag\\
\le&O(\2{v}_{\lambda',h'}(x,f(\tau))^p)=O(e^{p^2(-x-\lambda'f(\tau) +h')})\quad\mbox{ as }x\to\infty.\label{f-bar--f-infty-behaviour}
\end{align}
By \eqref{f-bar--f-minus-infty-behaviour} and \eqref{f-bar--f-infty-behaviour} $\2{f}_{\lambda, \lambda',h,h'}(\cdot,\tau)-f_{\lambda, \lambda',h,h'}(\cdot,\tau)\in L^1(\R)$ and the claim follows. 
Then by Lemma \ref{sub-supersoln-comparison-lem}, 
\begin{align}\label{V1-V2-compare}
\int_{\R}|V_1(x,\tau_1)-V_2(x,\tau_1)|\,dx
\le&\int_{\R}|\2{f}_{\lambda, \lambda',h,h'}(x,\tau)-f_{\lambda, \lambda',h,h'}(x,\tau)|e^{(1-m)(\tau_1-\tau)}\,dx\quad\forall\tau<\tau_1\le\2{\tau}_0\notag\\
=&I_1+I_2
\end{align}
where
\begin{equation}\label{I-1234-defn}
\left\{\begin{aligned}
&I_1=\int_{-\infty}^{x(\tau)}|\2{f}_{\lambda, \lambda',h,h'}(x,\tau)-f_{\lambda, \lambda',h,h'}(x,\tau)|e^{(1-m)(\tau_1-\tau)}\,dx\\
&I_2=\int_{x(\tau)}^{\infty}|\2{f}_{\lambda, \lambda',h,h'}(x,\tau)-f_{\lambda, \lambda',h,h'}(x,\tau)|e^{(1-m)(\tau_1-\tau)}\,dx.\end{aligned}\right.
\end{equation}
We will prove that $I_i\to 0$ as $\tau\to -\infty$ for $i=1,2$. Since $\2{v}_{\lambda',h'}(x,f(\tau))$ is a monotone decreasing function of $x\in\R$, by Lemma \ref{critical-pt-x-y-z-lem},
\begin{equation*}
1>\2{v}_{\lambda',h'}(x,f(\tau))\ge \2{v}_{\lambda',h'}(x(\tau),f(\tau))\to 1\quad\mbox{ uniformly on } (-\infty,x(\tau)]\quad\mbox{ as }\tau\to -\infty.
\end{equation*}
Then \eqref{f-lambda-lambda'-h-h'-minus-infty-behaviour} and \eqref{f-bar--f-minus-infty-behaviour0} hold uniformly on $(-\infty,x(\tau)]$ as $\tau\to-\infty$.
Now by \eqref{gamma-eqn} and \eqref{x-tau-value} for any $x\le x(\tau)$,
\begin{align}\label{argument-to-infty}
-x-\lambda'f(\tau)+h'\ge &-x(\tau)-\lambda'f(\tau)+h'=\left(\frac{\gamma_{\lambda'}-\lambda' p-\gamma_{\lambda}}{p}\right)\tau+O(1)\notag\\
=&-\left(\frac{\gamma_{\lambda}\gamma_{\lambda'}+p-1}{p\gamma_{\lambda'}}\right)\tau+O(1)\to\infty
\end{align}
uniformly on $(-\infty,x(\tau)]$ as $\tau\to -\infty$. 
Hence by \eqref{v-lambda-h-value-at-infty2} and \eqref{argument-to-infty},
\begin{align}\label{v-bar-power-1}
\2{v}_{\lambda',h'}(x,f(\tau))^{1-p}-1=&\left(1-C_{\lambda'}e^{-\gamma_{\lambda'}(-x-\lambda'f(\tau)+h')}+o(e^{-\gamma_{\lambda'}(-x-\lambda'f(\tau)+h')})\right)^{1-p}-1\notag\\
=&(p-1)C_{\lambda'}e^{-\gamma_{\lambda'}(-x-\lambda'f(\tau)+h')}+o(e^{-\gamma_{\lambda'}(-x-\lambda'f(\tau)+h')})
\end{align}
uniformly on $(-\infty,x(\tau)]$ as $\tau\to -\infty$. 
By  \eqref{f-bar--f-minus-infty-behaviour0} and \eqref{v-bar-power-1},
\begin{equation}\label{f-f-bar-difference}
0<\2{f}_{\lambda, \lambda',h,h'}(x,\tau)-f_{\lambda, \lambda',h,h'}(x,\tau)
\le Ce^{-\gamma_{\lambda'}(-x-\lambda'f(\tau)+h')}\le C'e^{-\gamma_{\lambda'}(-x-\lambda'\tau+h')}
\end{equation}
on $(-\infty,x(\tau)]$  as $\tau\to -\infty$ for some constants $C>0$, $C'>0$. 
Then by \eqref{gamma-eqn}, \eqref{x-tau-value}, \eqref{I-1234-defn} and \eqref{f-f-bar-difference},
\begin{align}
I_1\le C\int_{-\infty}^{x(\tau)}e^{-\gamma_{\lambda'}(-x-\lambda'\tau+h')}\cdot e^{-(1-m)\tau}\,dx
\le C'e^{\lambda'\gamma_{\lambda'}\tau-(1-m)\tau+\gamma_{\lambda'}x(\tau)}
\le C''e^{\frac{\gamma_{\lambda}\gamma_{\lambda'}}{p}\tau}\to 0\quad\mbox{ as }\tau\to -\infty.\label{I1-limit}
\end{align}
Similarly,
\begin{equation}\label{I2-limit}
I_2\to 0\quad\mbox{ as }\tau\to -\infty.
\end{equation}
By \eqref{V1-V2-compare}, \eqref{I1-limit} and \eqref{I2-limit},
\begin{align*}
&\int_{\R}|V_1(x,\tau_1)-V_2(x,\tau_1)|\,dx=0\quad\forall\tau_1\le\2{\tau}_0\notag\\
\Rightarrow\quad &V_1(x,\tau_1)=V_2(x,\tau_1)\qquad\qquad\qquad\forall x\in\R,\tau_1\le\2{\tau}_0
\end{align*}
and the theorem follows.
\end{proof}

\begin{cor}\label{4-parameters-soln-cont-wrt-parameters-cor}
Let $\2{\tau}_0<0$ be as in Lemma \ref{critical-pt-x-y-z-lem}. Then the  solution $v_{\lambda,\lambda',h,h'}\in C^{2,1}(\R\times (-\infty,\2{\tau}_0])$ of \eqref{yamabe-ode} in $\R\times (-\infty,\2{\tau}_0)$ which satisfies \eqref{v-lambda-lambda'-h-h'-upper-lower-bd} given by Theorem \ref{existence-thm2} is a continuous function of $\lambda>1$, $\lambda'>1$, $h\ge h_0$, $h'\ge h_0'$.
\end{cor}
\begin{proof}
Let $\{(\lambda_i,\lambda_i',h_i,h_i')\}_{i=1}^{\infty}\subset (1,\infty)\times (1,\infty)\times [h_0,\infty)\times [h_0',\infty)$ be a sequence such that
$(\lambda_i,\lambda_i',h_i,h_i')\to (\lambda_0,\lambda_0',h_0,h_0')$ as $i\to\infty$ for some $(\lambda_0,\lambda_0',h_0,h_0')\in (1,\infty)\times (1,\infty)\times [h_0,\infty)\times [h_0',\infty)$. For each $i\in\mathbb{N}$ let $v_i:=v_{\lambda_i,\lambda_i',h_i,h_i'}\in C^{2,1}(\R\times (-\infty,\2{\tau}_0])$ be the corresponding unique solution of \eqref{yamabe-ode} in $\R\times (-\infty,\2{\tau}_0)$ which satisfies \eqref{v-lambda-lambda'-h-h'-upper-lower-bd} given by Theorem \ref{existence-thm2} and Theorem \ref{uniqueness-thm} which satisfies
\begin{equation}\label{vi-lambda-lambda'-h-h'-upper-lower-bd}
f_{\lambda_i,\lambda_i',h_i,h_i'}(x,\tau)\le v_i(x,\tau)\le \2{f}_{\lambda_i,\lambda_i',h_i,h_i'}(x,\tau)\quad\forall x\in\R,\tau<\2{\tau}_0,i\in\mathbb{N}.
\end{equation}
Since both $f_{\lambda_i,\lambda_i',h_i,h_i'}$ and $\2{f}_{\lambda_i,\lambda_i',h_i,h_i'}$ converges uniformly on any compact subset $K$ of $\R\times (-\infty,\2{\tau}_0]$ to  $f_{\lambda_0,\lambda_0',h_0,h_0'}$ and $\2{f}_{\lambda_0,\lambda_0',h_0,h_0'}$ as $i\to\infty$,
by \eqref{vi-lambda-lambda'-h-h'-upper-lower-bd} the equation \eqref{yamabe-ode} for the sequence $\{v_i\}_{i=1}^{\infty}$ is uniformly parabolic on every compact subset of $\R\times (-\infty,\2{\tau}_0]$. Then by Schauder's estimates \cite{LSU} the sequence $\{v_i\}_{i=1}^{\infty}$ is equi-Holder continuous in $C^{2,1}(K)$ for any compact set $K\subset \R\times (-\infty,\2{\tau}_0]$. Hence by the Ascoli Theorem and a diagonalization argument the sequence $\{v_i\}_{i=1}^{\infty}$ has a convergence subsequence $\{v_{i_k}\}_{k=1}^{\infty}$  that converges in $C^{2,1}(K)$ for any compact $K\subset \R\times (-\infty,\2{\tau}_0]$ to a solution $v\in C^{2,1}(\R\times (-\infty,\2{\tau}_0])$ of \eqref{yamabe-ode} in $\R\times (-\infty,\2{\tau}_0)$ as $k\to\infty$. Letting $i=i_k\to\infty$ in \eqref{vi-lambda-lambda'-h-h'-upper-lower-bd}, $v$ satisfies \begin{equation}\label{vi-lambda-lambda'-h-h'-upper-lower-bd-a}
f_{\lambda_0,\lambda_0',h_0,h_0'}(x,\tau)\le v(x,\tau)\le \2{f}_{\lambda_0,\lambda_0',h_0,h_0'}(x,\tau)\quad\forall x\in\R,\tau<\2{\tau}_0,i\in\mathbb{N}.
\end{equation}
By \eqref{vi-lambda-lambda'-h-h'-upper-lower-bd-a} and the uniqueness Theorem \ref{uniqueness-thm}, 
$v=v_{\lambda_0,\lambda_0',h_0,h_0'}$ in $\R\times (-\infty,\2{\tau}_0]$ where $v_{\lambda_0,\lambda_0',h_0,h_0'}$ is the unique solution of $\R\times (-\infty,\2{\tau}_0]$ given by Theorem \ref{existence-thm2} and Theorem \ref{uniqueness-thm} which satisfies \eqref{vi-lambda-lambda'-h-h'-upper-lower-bd-a}. Hence $v_i$ converges in $C^{2,1}(K)$ for any compact $K\subset \R\times (-\infty,\2{\tau}_0]$ to $v_{\lambda_0,\lambda_0',h_0,h_0'}$  as $i\to\infty$ and the corollary follows. 
\end{proof}

\begin{thm}\label{limit-thm}
Let $\lambda>1$, $\lambda'>1$,  $h\ge h_0$, $h'\ge h_0'$ and let $\2{\tau}_0<0$ be as in Lemma \ref{critical-pt-x-y-z-lem}. For any $0<k\le k_0$ let $v_{\lambda,\lambda',h,h',k}\in C^{2,1}(\R\times (-\infty,\2{\tau}_0])$ be a solution of \eqref{yamabe-ode} in $\R\times (-\infty,\2{\tau}_0)$ constructed in Theorem \ref{existence-thm1} which satisfies \eqref{v-lambda-lambda'-h-h'-k-lower-upper-bd}. Then $v_{\lambda,\lambda',h,h',k}$ increases and converges uniformly in $C^{2,1}(K)$  for any compact subset $K$ of $\R\times (-\infty,\2{\tau}_0]$ to the unique solution $v_{\lambda,\lambda',h,h'}\in C^{2,1}(\R\times (-\infty,\2{\tau}_0])$ of \eqref{yamabe-ode} in $\R\times (-\infty,\2{\tau}_0)$ which satisfies \eqref{v-lambda-lambda'-h-h'-upper-lower-bd} as $k\searrow 0$.
\end{thm}
\begin{proof}
By \eqref{v-lambda-lambda'-h-h'-k-lower-upper-bd}, 
\begin{equation*}\label{vi-lambda-lambda'-h-h'-ki-upper-lower-bd}
f_{\lambda,\lambda',h,h',k_0}(x,\tau)\le v_{\lambda,\lambda',h,h',k}(x,\tau)\le \2{f}_{\lambda,\lambda',h,h'}(x,\tau)\quad\forall x\in\R,\tau<\2{\tau}_0,0<k\le k_0.
\end{equation*}
Hence the equation \eqref{yamabe-ode} for the family of functions $\{v_{\lambda,\lambda',h,h',k}\}_{0<k\le k_0}$ is uniformly parabolic on every compact subset $K$ of $\R\times (-\infty,\2{\tau}_0]$. Then by Schauder's estimates \cite{LSU} the family of functions $\{v_{\lambda,\lambda',h,h',k}\}_{0<k\le k_0}$ is equi-Holder continuous in $C^{2,1}(K)$ for any compact subset $K$ of $\R\times (-\infty,\2{\tau}_0]$. Hence by Remark \ref{comparison-rmk1}, the Ascoli Theorem and a diagonalization argument   $v_{\lambda,\lambda',h,h',k}$  increases and converges uniformly in $C^{2,1}(K)$ for any compact subset $K$ of $\R\times (-\infty,\2{\tau}_0]$ to a  solution $v\in C^{2,1}(\R\times (-\infty,\2{\tau}_0])$ of \eqref{yamabe-ode} in $\R\times (-\infty,\2{\tau}_0)$ as $k\searrow 0$. Letting $k\searrow 0$ in \eqref{v-lambda-lambda'-h-h'-k-lower-upper-bd}, $v$ satisfies \eqref{v-lambda-lambda'-h-h'-upper-lower-bd}. By Theorem \ref{uniqueness-thm}
$v=v_{\lambda,\lambda',h,h'}$ is the unique solution of \eqref{yamabe-ode} in $\R\times (-\infty,\2{\tau}_0)$ which satisfies \eqref{v-lambda-lambda'-h-h'-upper-lower-bd}.
\end{proof}

We are now ready for the proof of Theorem \ref{5-parameters-soln-thm}, Theorem \ref{4-parameters-soln-thm} and Theorem \ref{3-parameters-soln-thm}.
 
{\ni{\it Proof of Theorem \ref{5-parameters-soln-thm}:}} Existence of solution $v=v_{\lambda,\lambda',h,h',k}\in C^{2,1}(\R\times (-\infty,\2{\tau}_0])$ of \eqref{yamabe-ode} in $\R\times (-\infty,\2{\tau}_0)$ satisfying \eqref{v-lambda-lambda'-h-h'-k-lower-upper-bd} such that $v(x,\tau)$ is a decreasing function of $\tau<\2{\tau}_0$ is proved in Theorem \ref{existence-thm1}. Moreover by Theorem \ref{existence-thm1} and Remark \ref{comparison-rmk1} we can construct the solution $v=v_{\lambda,\lambda',h,h',k}$ such that $v$ also satisfies (ii) and (iv) of Theorem \ref{5-parameters-soln-thm}. 

\noindent{\bf Proof of (i) of Theorem \ref{5-parameters-soln-thm}}: 
Since $v_{\lambda,h}(x,f(\tau))\to 1$, $\2{v}_{\lambda,h}(x,f(\tau))\to 1$ and $\xi_k(\tau)\to 1$  for any $x\in\R$ as $\tau\to -\infty$, we have
\begin{equation*}
f_{\lambda,\lambda',h,h',k}(x,\tau)\to 1\quad\mbox{ and }\quad \2{f}_{\lambda,\lambda',h,h',k}(x,\tau)\to 1\quad\forall x\in\R\quad\mbox{ as }\tau\to -\infty.
\end{equation*}
Hence by \eqref{v-lambda-lambda'-h-h'-k-lower-upper-bd} for any $x\in\R$, $v(x,\tau)\to 1$ as $\tau\to -\infty$ and (i) of Theorem \ref{5-parameters-soln-thm} follows. 

\noindent{\bf Proof of (iii) of Theorem \ref{5-parameters-soln-thm}}:  Let $d$ be given by \eqref{d-defn} and $x(\tau)$ be as in Lemma \ref{critical-pt-x-y-z-lem}. Then $d>\frac{p-1}{p}$. Hence by \eqref{xi-k-infinity} and (i) of Lemma \ref{critical-pt-x-y-z-lem}, there exists constant $C_{\lambda,\lambda',h,h'}>0$ such that
\begin{equation}\label{v-v-bar>xi-k}
v_{\lambda,h}(x(\tau),f(\tau))=\2{v}_{\lambda',h'}(x(\tau),f(\tau))=1-C_{\lambda,\lambda',h,h'}e^{d \tau}+o(e^{d \tau})
>\xi_k(\tau)\quad\mbox{ as }\tau\to -\infty.
\end{equation}
Since
\begin{equation*}
1-\xi_k(\tau)^{p-1}=1-\left(1-ke^{\frac{p-1}{p}\tau}\right)^p
=kpe^{\frac{p-1}{p}\tau}+o\left(e^{\frac{p-1}{p}\tau}\right)\quad\mbox{ as }\tau\to -\infty,
\end{equation*} 
by \eqref{v-v-bar>xi-k},
\begin{align}\label{f-k-lower-bd}
\max_{x\in\R}f_{\lambda,\lambda',h,h',k}(x,\tau)\ge&f_{\lambda,\lambda',h,h',k}(x(\tau),\tau)\notag\\
\ge&(3\xi_k(\tau)^{1-p}-2)^{-\frac{1}{p-1}}=\xi_k(\tau)\left(1+2(1-\xi_k(\tau)^{p-1})\right)^{-\frac{1}{p-1}}\quad\mbox{ as }\tau\to-\infty\notag\\
=&\xi_k(\tau)\left(1+2\left(kpe^{\frac{p-1}{p}\tau}+o\left(e^{\frac{p-1}{p}\tau}\right)\right)\right)^{-\frac{1}{p-1}}\qquad\qquad\qquad\quad\mbox{ as }\tau\to-\infty\notag\\
=&\xi_k(\tau)\left(1-\frac{2kp}{p-1}e^{\frac{p-1}{p}\tau}+o\left(e^{\frac{p-1}{p}\tau}\right)\right)\qquad\qquad\qquad\qquad\quad\mbox{ as }\tau\to-\infty.
\end{align}
By \eqref{v-lambda-lambda'-h-h'-k-lower-upper-bd} and \eqref{f-k-lower-bd},
\begin{align*}
&\xi_k(\tau)\ge v(x_0(\tau),\tau)=\max_{x\in\R}v(x,\tau)\ge \xi_k(\tau)\left(1-\frac{2kp}{p-1}e^{\frac{p-1}{p}\tau}+o\left(e^{\frac{p-1}{p}\tau}\right)\right)\quad\mbox{ as }\tau\to-\infty\\
\Rightarrow\quad&|v(x_0(\tau),\tau)-\xi_k(\tau)|\le O\left(e^{\frac{p-1}{p}\tau}\right)\qquad\qquad\qquad\qquad\qquad\qquad\qquad\qquad\mbox{ as }\tau\to-\infty
\end{align*}
and (iii) of Theorem \ref{5-parameters-soln-thm} follows.

\noindent{\bf Proof of (v) of Theorem \ref{5-parameters-soln-thm}}: If $c>\lambda$, then
\begin{equation}\label{v-x+c-tau-limit}
v_{\lambda,h}(x+c\tau,f(\tau))=v_{\lambda}(x+(c-\lambda-\lambda qe^{\frac{p-1}{p}\tau})\tau+h)\to 0\quad\mbox{ uniformly on }(-\infty,A]\quad\forall A\in\R
\end{equation}
as $\tau\to -\infty$.
By  \eqref{v-lambda-lambda'-h-h'-k-lower-upper-bd} and \eqref{v-x+c-tau-limit},
\begin{equation*}
0<v(x+c\tau,\tau)\le\2{f}_{\lambda,\lambda',h,h',k}(x+c\tau,\tau)\le v_{\lambda,h}(x+c\tau,f(\tau))\to 0\quad\mbox{ uniformly on }(-\infty,A]
\end{equation*}
for any $A\in\R$ as $\tau\to -\infty$ and (v)(a) of Theorem \ref{5-parameters-soln-thm} follows.

If $-\lambda'<c<\lambda$, then
\begin{equation}\label{v-x+c-tau-limit2}
\left\{\begin{aligned}
&v_{\lambda,h}(x+c\tau,f(\tau))=v_{\lambda}(x+(c-\lambda-\lambda qe^{\frac{p-1}{p}\tau})\tau+h)\to 1\\
&\2{v}_{\lambda',h'}(x+c\tau,f(\tau))=v_{\lambda'}(-x-(c+\lambda'+\lambda' qe^{\frac{p-1}{p}\tau})\tau+h')\to 1
\end{aligned}\right.
\end{equation}
uniformly on any compact subset of $\R$ as $\tau\to -\infty$. By \eqref{xi-k-infinity}, \eqref{v-lambda-lambda'-h-h'-k-lower-upper-bd} and \eqref{v-x+c-tau-limit2},
\begin{equation*}
1\ge v(x+c\tau,\tau)\ge f_{\lambda,\lambda',h,h',k}(x+c\tau,\tau)\to 1
\end{equation*}
uniformly on any compact subset of $\R$ as $\tau\to -\infty$ and (v)(b) of Theorem \ref{5-parameters-soln-thm} follows.

If $c<-\lambda'$, then
\begin{equation}\label{v-x+c-tau-limit3}
\2{v}_{\lambda',h'}(x+c\tau,f(\tau))=v_{\lambda'}(-x-(c+\lambda'+\lambda' qe^{\frac{p-1}{p}\tau})\tau+h')\to 0\quad\mbox{ uniformly on }[A,\infty)\quad\forall A\in\R
\end{equation}
as $\tau\to -\infty$. By \eqref{v-lambda-lambda'-h-h'-k-lower-upper-bd} and \eqref{v-x+c-tau-limit3},
\begin{equation*}
0<v(x+c\tau,\tau)\le\2{v}_{\lambda',h'}(x+c\tau,f(\tau))\to 0\quad\mbox{ uniformly on }[A,\infty)
\end{equation*}
for any $A\in\R$ as $\tau\to -\infty$ and (v)(c) of Theorem \ref{5-parameters-soln-thm} follows.

If $c=\lambda$, then
\begin{align}
|v_{\lambda,h}(x+c\tau,f(\tau))-v_{\lambda}(x+h)|=&|v_{\lambda}(x+h-\lambda q\tau e^{\frac{p-1}{p}\tau})-v_{\lambda}(x+h)|\notag\\
\le&\|v_{\lambda}'\|_{L^{\infty}(\R)} \lambda q|\tau| e^{\frac{p-1}{p}\tau}\to 0\quad\mbox{ uniformly on }\R\quad\mbox{ as }\tau\to -\infty
\label{ancient-limit-3}
\end{align}
and
\begin{equation}\label{v-x+c-tau-limit6}
\2{v}_{\lambda',h'}(x+c\tau,f(\tau))=v_{\lambda'}(-x-(\lambda+\lambda'+\lambda' qe^{\frac{p-1}{p}\tau})\tau+h')\to 1\quad\mbox{ uniformly on }(-\infty,A]
\end{equation}
for any $A\in\R$ as $\tau\to -\infty$. Since by \eqref{xi-k-infinity} $\xi_k(\tau)\to 1$ as $\tau\to -\infty$, by \eqref{v-lambda-lambda'-h-h'-k-lower-upper-bd}, \eqref{ancient-limit-3} and \eqref{v-x+c-tau-limit6},
\begin{equation}\label{limit1-ineqn}
\2{f}_{\lambda,\lambda',h,h',k}(x+c\tau,\tau)\le v_{\lambda,h}(x+c\tau,f(\tau))\to v_{\lambda}(x+h)\quad\mbox{ uniformly on }\R
\quad\mbox{ as }\tau\to -\infty
\end{equation}
and
\begin{equation}\label{limit2-ineqn}
f_{\lambda,\lambda',h,h',k}(x+c\tau,\tau)\to v_{\lambda}(x+h)\quad\mbox{ uniformly on }(-\infty,A]\quad\forall A\in\R\quad\mbox{ as }\tau\to -\infty
\end{equation}
By \eqref{v-lambda-lambda'-h-h'-k-lower-upper-bd}, \eqref{limit1-ineqn} and \eqref{limit2-ineqn}, (v)(d) of Theorem \ref{5-parameters-soln-thm} follows.

If $c=-\lambda'$, then
\begin{align}
|\2{v}_{\lambda',h'}(x+c\tau,f(\tau))- v_{\lambda'}(-x+h')|
=&|v_{\lambda'}(-x+h'-\lambda' q\tau e^{\frac{p-1}{p}\tau})-v_{\lambda'}(-x+h')|\notag\\
\le&\|v_{\lambda'}'\|_{L^{\infty}(\R)} \lambda' q|\tau| e^{\frac{p-1}{p}\tau}\to 0\quad\mbox{ uniformly on }\R\quad\mbox{ as }\tau\to -\infty
\label{ancient-limit-6}
\end{align}
and
\begin{equation}\label{v-x+c-tau-limit4}
v_{\lambda,h}(x+c\tau,f(\tau))=v_{\lambda}(x-(\lambda+\lambda'+\lambda qe^{\frac{p-1}{p}\tau})\tau+h)\to 1\quad\mbox{ uniformly on }[A,\infty)\\\
\end{equation}
for any $A\in\R$ as $\tau\to -\infty$. By \eqref{xi-k-infinity},  \eqref{ancient-limit-6} and \eqref{v-x+c-tau-limit4},
\begin{equation}\label{limit3-ineqn}
f_{\lambda,\lambda',h,h',k}(x+c\tau,\tau)\to v_{\lambda'}(-x+h')\quad\mbox{ uniformly on }[A,\infty)\quad\forall A\in\R\quad\mbox{ as }\tau\to -\infty
\end{equation}
and 
\begin{equation}\label{limit4-ineqn}
\2{f}_{\lambda,\lambda',h,h',k}(x+c\tau,\tau)\le\2{v}_{\lambda',h'}(x+c\tau,f(\tau))\to v_{\lambda'}(-x+h')\quad\mbox{ uniformly on }\R\quad\mbox{ as }\tau\to -\infty.
\end{equation}
By \eqref{v-lambda-lambda'-h-h'-k-lower-upper-bd}, \eqref{limit3-ineqn} and \eqref{limit3-ineqn},
(v)(e) of Theorem \ref{5-parameters-soln-thm} follows.

{\hfill$\square$\vspace{6pt}

{\ni{\it Proof of Theorem \ref{4-parameters-soln-thm}:}}
By Theorem \ref{existence-thm2}, Remark \ref{comparison-rmk0}, Theorem \ref{uniqueness-thm},   Theorem \ref{limit-thm} and an argument similar to the proof of Theorem \ref{5-parameters-soln-thm}, there exists a unique solution $v=v_{\lambda,\lambda',h,h'}\in C^{2,1}(\R\times (-\infty,\2{\tau}_0])$ of \eqref{yamabe-ode} in $\R\times (-\infty,\2{\tau}_0)$ which satisfies \eqref{v-lambda-lambda'-h-h'-upper-lower-bd}, \eqref{v-v-compare} and (i), (ii),  (v) of Theorem \ref{4-parameters-soln-thm}.
By  an argument similar to the proof of Remark \ref{comparison-rmk1}, (iv) of Theorem \ref{4-parameters-soln-thm} holds. Note the (vii) of Theorem \ref{4-parameters-soln-thm} is proved in Theorem \ref{limit-thm}.

\noindent{\bf Proof of (iii) of Theorem \ref{4-parameters-soln-thm}}: Let $x(\tau)$ be as  in Lemma \ref{critical-pt-x-y-z-lem} and $x_0(\tau)$ be given by (ii) of Theorem \ref{4-parameters-soln-thm}. By \eqref{gamma-eqn},
\begin{equation*}
-x(\tau)-\lambda' f(\tau)=-\left(\frac{\gamma_{\lambda}\gamma_{\lambda'}+p-1}{p\gamma_{\lambda'}}\right)\tau+ O(1)\to\infty\quad\mbox{ as }\tau\to -\infty.
\end{equation*}
Hence by \eqref{v-lambda-h-value-at-infty2} there exists a constant $C_{\lambda'}>0$ such that
\begin{align}\label{v-bar-power-expansion}
\2{v}_{\lambda',h'}(x(\tau),f(\tau))^{1-p}-1=&\left(1-C_{\lambda'}e^{-\gamma_{\lambda'}(-x(\tau)-\lambda' f(\tau)+h')}+o(e^{-\gamma_{\lambda'}(-x(\tau)-\lambda' f(\tau)+h')})\right)^{1-p}-1\notag\\
=&\left(1-C_{\lambda'}e^{d\tau}+o(e^{d\tau})\right)^{1-p}-1\notag\\
=&(p-1)C_{\lambda'}e^{d\tau}+o(e^{d\tau})\quad\mbox{ as }\tau\to-\infty
\end{align}
where $d$ is given by \eqref{d-defn}. By \eqref{v-lambda-lambda'-h-h'-upper-lower-bd}, \eqref{f-bar-max-value-expansion} and \eqref{v-bar-power-expansion},
\begin{align*}
v_{\lambda,h}(x(\tau),f(\tau))=&\max_{x\in\R}\2{f}_{\lambda,\lambda',h,h'}(x,\tau)\\
\ge&v(x_0(\tau),\tau)=\max_{x\in\R}v(x,\tau) \ge v(x(\tau),\tau)\ge f_{\lambda,\lambda',h,h'}(x(\tau),\tau)\\
=&v_{\lambda,h}(x(\tau),f(\tau))\left(1+v_{\lambda,h}(x(\tau),f(\tau))^{p-1}\left(\2{v}_{\lambda',h'}(x(\tau),f(\tau))^{1-p}-1\right)\right)^{-\frac{1}{p-1}}\\
=&v_{\lambda,h}(x(\tau),f(\tau))\left(1+(p-1)C_{\lambda'}v_{\lambda,h}(x(\tau),f(\tau))^{p-1}e^{d\tau}+o(e^{d\tau})\right)^{-\frac{1}{p-1}}\\
=&v_{\lambda,h}(x(\tau),f(\tau))\left(1-C_{\lambda'}v_{\lambda,h}(x(\tau),f(\tau))^{p-1}e^{d\tau}+o(e^{d\tau})\right)\quad\mbox{ as }\tau\to-\infty.
\end{align*}
Hence
\begin{equation*}
\left|\max_{x\in\R}v(x,\tau)-v_{\lambda,h}(x(\tau),f(\tau))\right|
\le (C_{\lambda'}+1)e^{d\tau}\le C e^{\frac{p-1}{p}\tau}\quad\mbox{ as }\tau\to-\infty\notag\\
\end{equation*}
and (iii) of Theorem \ref{4-parameters-soln-thm} follows.

\noindent{\bf Proof of (vi) of Theorem \ref{4-parameters-soln-thm}}: Since  $v_{\lambda,h}(x,f(\tau))$ converges to $1$ uniformly on $[A,\infty)\times [\tau_1,\2{\tau}_0]$ as $h\to\infty$ for any $A\in\R$ and $\tau_1<\2{\tau}_0$, both $f_{\lambda,\lambda',h,h'}(x,\tau)$ and $\2{f}_{\lambda,\lambda',h,h'}(x,\tau)$
converges to $\2{v}_{\lambda',h'}(x,f(\tau))$ uniformly on $[A,\infty)\times [\tau_1,\2{\tau}_0]$ for any $A\in\R$ and $\tau_1<\2{\tau}_0$ as $h\to\infty$. Hence
by \eqref{v-lambda-lambda'-h-h'-upper-lower-bd} $v_{\lambda,\lambda',h,h'}$ converges to $\2{v}_{\lambda',h'}(x,f(\tau))$ uniformly on $[A,\infty)\times [\tau_1,\2{\tau}_0]$ for any $A\in\R$ and $\tau_1<\2{\tau}_0$ as $h\to\infty$.   The proof of the other case $h'\to\infty$ then also follows by a similar argument.

{\hfill$\square$\vspace{6pt}

{\ni{\it Proof of Theorem \ref{3-parameters-soln-thm}:}} By Theorem \ref{existence-thm3}, Remark \ref{comparison-rmk0} and an argument similar to the proof of  
Theorem \ref{5-parameters-soln-thm}, there exists a  solution $v=v_{\lambda,h,k}\in C^{2,1}(\R\times (-\infty,\2{\tau}_0])$ of \eqref{yamabe-ode} in $\R\times (-\infty,\2{\tau}_0)$ which satisfies \eqref{v-lambda-h-k-upper-lower-bd}, \eqref{v-v-compare2} and (i), (ii), (v) of Theorem \ref{3-parameters-soln-thm}.  By choosing $v_{0,j}=\2{f}_{\lambda,h,k}(x,-j)$ in the construction of the solution $v_{\lambda,h,k}$ in Theorem \ref{existence-thm3} and an argument similar to the proof of Remark \ref{comparison-rmk1}, we can construct the solution $v_{\lambda,h,k}$ such that (iv) of Theorem \ref{3-parameters-soln-thm} holds. Note that (iii) of Theorem \ref{3-parameters-soln-thm} follows directly from \eqref{v-lambda-h-k-upper-lower-bd}.

Since $\xi_k$ converges to $1$ uniformly on $\R\times (-\infty,\2{\tau}_0]$ as $k\to 0$, both
$f_{\lambda,h,k}$  and $\2{f}_{\lambda,h,k}$ converges to $v_{\lambda,h}(x,f(\tau))$ uniformly on every compact subset of $\R\times (-\infty,\2{\tau}_0]$ as $k\to 0$. Hence by \eqref{v-lambda-h-k-upper-lower-bd}, (vi) of Theorem \ref{3-parameters-soln-thm} follows.

Since  $v_{\lambda,h}(x,f(\tau))$ converge to $1$ uniformly on $[A,\infty)\times [\tau_1,\2{\tau}_0]$ as $h\to\infty$ for any $A\in\R$ and $\tau_1<\2{\tau}_0$,
both $f_{\lambda,h,k}$ and $\2{f}_{\lambda,h,k}$ converges to $\xi_k$ uniformly on $[A,\infty)\times [\tau_1,\2{\tau}_0]$ as $h\to\infty$ for any $A\in\R$ and $\tau_1<\2{\tau}_0$. Hence by \eqref{v-lambda-h-k-upper-lower-bd}  $v_{\lambda,h,k}$ converges to $\xi_k$ uniformly on $[A,\infty)\times [\tau_1,\2{\tau}_0]$ as $h\to\infty$ for any $A\in\R$ and $\tau_1<\2{\tau}_0$ and (vii) of Theorem \ref{3-parameters-soln-thm} follows.

\noindent{\bf Proof of (viii) of Theorem \ref{3-parameters-soln-thm}}: Let $v_{\lambda,\lambda',h,h',k}\in C^{2,1}(\R\times (-\infty,\2{\tau}_0])$ be a solution of \eqref{yamabe-ode} in $\R\times (-\infty,\2{\tau}_0)$ given by Theorem \ref{5-parameters-soln-thm} that satisfies \eqref{v-lambda-lambda'-h-h'-k-lower-upper-bd}. Then by \eqref{v-lambda-lambda'-h-h'-k-lower-upper-bd},
\begin{equation*}
f_{\lambda,\lambda',h,h_0',k}(x,\tau)\le v_{\lambda,\lambda',h,h',k}(x,\tau)\le\2{f}_{\lambda,h,k}(x,\tau)\quad\forall x\in\R,\tau<\2{\tau}_0,h'\ge h_0'.
\end{equation*}
Hence the equation \eqref{yamabe-ode} for the family of ancient solutions $\{v_{\lambda,\lambda',h,h',k}\}_{h'\ge h_0'}$ is uniformly parabolic on every compact subset $K$ of $\R\times (-\infty,\2{\tau}_0]$. By the parabolic Schauder estimates \cite{LSU} the family  $\{v_{\lambda,\lambda',h,h',k}\}_{h'\ge h_0'}$ is uniformly equi-Holder continuous in $C^{2,1}(K)$ on every compact subset $K$ of $\R\times (-\infty,\2{\tau}_0]$. Then by the Ascoli Theorem and (iv) of Theorem \ref{5-parameters-soln-thm}, the sequence $\{v_{\lambda,\lambda',h,h',k}\}_{h'\ge h_0'}$ increases and converges in $C^{2,1}(K)$ for any compact subset $K$ of $\R\times (-\infty,\2{\tau}_0]$ as $h'\to\infty$ to a solution $v\in C^{2,1}(\R\times (-\infty,\2{\tau}_0])$  of \eqref{yamabe-ode} in $\R\times (-\infty,\2{\tau}_0)$ that satisfies \eqref{v-v-compare2} with $v_{\lambda,h,k}$ there being replaced by $v$. Letting  $h'\to\infty$ in \eqref{v-lambda-lambda'-h-h'-k-lower-upper-bd}, $v$ satisfies \eqref{v-lambda-h-k-upper-lower-bd}.

{\hfill$\square$\vspace{6pt}}

\end{document}